\documentclass[english,12pt]{article} 

\usepackage{cite}             % nicer citations
\usepackage{latexsym}         % dito
\usepackage{amsmath}          % dito
\usepackage{amssymb}          % various ams stuff
\usepackage{amsxtra}          % dito
\usepackage[latin1]{inputenc} % smart input of special characters
\usepackage{eucal}            % nicer caligraphic fonts
\usepackage{longtable}        % tables longer than one page
\usepackage{exscale}          % larger summation signs
\usepackage{bbm}              % nicer bbm fonts
\usepackage{theorem}          % nettere Theoreme
\usepackage{paralist}         % kleinere Listen COOL!
\usepackage{stmaryrd}         % more symbols & brackets
\usepackage{a4wide}
\usepackage{float}
\usepackage{url}
\usepackage{enumitem}

\usepackage{tikz}             %the real real diagrams       
\usepackage{tikz-cd}          %add on 
\usetikzlibrary{matrix, arrows}
\usepackage{leftidx}          %left dual star closer to the object

\usepackage{rotating}
\numberwithin{equation}{section}

\newcommand{\C}        {\mathbb{C}}

\newcommand{\cc}[1]      {\overline{{#1}}}

\newcommand{\Vect}   {\mathsf{Vect}}
\newcommand{\tr}         {\operatorname{\mathsf{tr}}}

\newcommand{\boxtensor}[1] {\mathbin{\boxtimes_{\scriptscriptstyle{#1}}}}
\newcommand{\circtensor} {\setbox0\hbox{\large$\circlearrowleft$} \rlap{\hbox to\wd0{\hss$\times$\hss}}\box0}
\newcommand{\circtensorsmall} {\setbox0\hbox{$\circlearrowleft$} \rlap{\hbox to\wd0{\hss$\times$\hss}}\box0}

\newcommand{\swe}[1]    {{\scriptscriptstyle{(#1)}}}

\newcommand{\id}         {\operatorname{\mathsf{id}}}
\newcommand{\unit}         {\operatorname{1}}
\newcommand{\rr}         {\rightarrow}

\newcommand{\act}        {\operatorname{\triangleright}}
\newcommand{\actrd}        {\operatorname{\triangleright}^{\#}}
\newcommand{\ractrd}      {\operatorname{\triangleleft}^{\#}}

\newcommand{\actld}        {\leftidx{{}^{\#}}\!{\operatorname{\triangleright}}{}}
\newcommand{\ractld}       {\leftidx{{}^{\#\!}}\!{\operatorname{\triangleleft}}{}}
\newcommand{\ract}       {\operatorname{\triangleleft}}

\newcommand{\Cat}[1]         {\operatorname{\mathcal{#1}}}
\newcommand{\Catc}[1]         {\operatorname{\cc{\mathcal{#1}}}}

\newcommand{\op}[1]           {\Cat{#1}^{\operatorname{\mathrm{op}}}}
\newcommand{\Fun}         {\operatorname{\mathrm{Fun}}}
\newcommand{\Funbal}         {\operatorname{\mathrm{Fun}}^{\mathrm{bal}}}
\newcommand{\Funle}         {\operatorname{\mathrm{Fun}}^{\mathrm{l.e.}}}
\newcommand{\Funre}         {\operatorname{\mathrm{Fun}}^{\mathrm{r.e.}}}
\newcommand{\Funbalre}         {\operatorname{\mathrm{Fun}}^{\mathrm{bal, r.e.}}}
\newcommand{\Hom}        {\operatorname{\mathsf{Hom}}}

\newcommand{\End}        {\operatorname{\mathsf{End}}}

\newcommand{\Obj}        {\operatorname{\mathsf{Obj}}}
\def\Algbim {\operatorname{\mathsf{BimAlg}}} %{\Cat{A}lgbim}

\newcommand{\Bimod}[5]{\sideset{^{\scriptscriptstyle{#1}}_{\scriptscriptstyle{#2}}}{^{\scriptscriptstyle{#4}}_{\scriptscriptstyle{#5}}}{\operatorname{#3}}}
\newcommand{\Bim}[3] {\Bimod{}{#1}{#2}{}{#3}}
\newcommand{\LCC}  {\Bimod{}{\Cat{C}}{\Cat{C}}{}{}}
\newcommand{\RCC}  {\Bimod{}{}{\Cat{C}}{}{\Cat{C}}}

\newcommand{\CM}  {\Bimod{}{\Cat{C}}{\Cat{M}}{}{}}

\newcommand{\MC}  {\Bimod{}{}{\Cat{M}}{}{\Cat{C}}}

\newcommand{\CN}  {\Bimod{}{\Cat{C}}{\Cat{N}}{}{}}

\newcommand{\NC}  {\Bimod{}{}{\Cat{N}}{}{\Cat{C}}}

\newcommand{\DN}  {\Bimod{}{\Cat{D}}{\Cat{N}}{}{}}

\newcommand{\CMD}  {\Bimod{}{\Cat{C}}{\Cat{M}}{}{\Cat{D}}}
\newcommand{\DMC}  {\Bimod{}{\Cat{D}}{\Cat{M}}{}{\Cat{C}}} 

\newcommand{\CMC}  {\Bimod{}{\Cat{C}}{\Cat{M}}{}{\Cat{C}}}
\newcommand{\VectMVect}  {\Bimod{}{\Vect}{\Cat{M}}{}{\Vect}}
\newcommand{\CNC}  {\Bimod{}{\Cat{C}}{\Cat{N}}{}{\Cat{C}}}
\newcommand{\CMRFC}  {\Bimod{}{\Cat{C}}{\Cat{M}}{\Rho \mathsf{F}}{\Cat{C}}}

\newcommand{\CND}  {\Bimod{}{\Cat{C}}{\Cat{N}}{}{\Cat{D}}}

\newcommand{\DNC}  {\Bimod{}{\Cat{D}}{\Cat{N}}{}{\Cat{C}}}

\newcommand{\CNE}  {\Bimod{}{\Cat{C}}{\Cat{N}}{}{\Cat{E}}}
\newcommand{\END}  {\Bimod{}{\Cat{E}}{\Cat{N}}{}{\Cat{D}}}

\newcommand{\EKD}  {\Bimod{}{\Cat{E}}{\Cat{K}}{}{\Cat{D}}}
\newcommand{\ENC}  {\Bimod{}{\Cat{E}}{\Cat{N}}{}{\Cat{C}}} 
\newcommand{\CK}  {\Bimod{}{\Cat{C}}{\Cat{K}}{}{}}
\newcommand{\CKC}  {\Bimod{}{\Cat{C}}{\Cat{K}}{}{\Cat{C}}}
\newcommand{\CLC}  {\Bimod{}{\Cat{C}}{\Cat{L}}{}{\Cat{C}}}

\newcommand{\DDD}  {\Bimod{}{\Cat{D}}{\Cat{D}}{}{\Cat{D}}}

\newcommand{\CCC}  {\Bimod{}{\Cat{C}}{\Cat{C}}{}{\Cat{C}}}

\newcommand{\tCCC}  {\Bimod{\Tau}{\Cat{C}}{\Cat{C}}{}{\Cat{C}}}

\newcommand{\CMDrd}  {\Bimod{}{\Cat{C}}{\Cat{M}}{\#}{\Cat{D}}} 
\newcommand{\CMCrd}  {\Bimod{}{\Cat{C}}{\Cat{M}}{\#}{\Cat{C}}} 
\newcommand{\CMDld}  {\Bimod{\#}{\Cat{C}}{\Cat{M}}{}{\Cat{D}}}

\newcommand{\ENCld}  {\Bimod{\#}{\Cat{E}}{\Cat{N}}{}{\Cat{C}}}

\newcommand{\Mod}      {\operatorname{\mathsf{Mod}}}

\newcommand{\Funl}[2]     {\operatorname{\mathrm{Fun}}_{\scriptscriptstyle{#1}}{(#2)}}
\newcommand{\Funlre}[2]     {\operatorname{\mathrm{Fun}}^{\mathrm{r.e.}}_{\scriptscriptstyle{#1}}{(#2)}}

\newcommand{\BimCat}       {\operatorname{\mathsf{BimCat}}} %{\mathfrak{B}}%{\underline{\underline{\mathsf{B}}}}  
\newcommand{\Categ}   {\mathsf{Cat}}

\newcommand{\cent}     {\mathcal{Z}}

\newcommand{\centCt}   {\mathcal{Z}_{\Cat{C}}^{\Tau}}
\newcommand{\centCF}   {\mathcal{Z}_{\Cat{C}}^{\mathsf{F}}}

\newcommand{\ev}[1]   {\operatorname{\mathsf{ev}}_{#1}}
\newcommand{\coev}[1]   {\operatorname{\mathsf{coev}}_{#1}}
\makeatletter
\newcommand{\Rmnum}[1]{\expandafter\@slowromancap\romannumeral #1@}
\makeatother

\theoremheaderfont{\normalfont\bfseries}
\theorembodyfont{\itshape}
\newtheorem{lemma}{Lemma}[section]
\newtheorem{proposition}[lemma]{Proposition}
\newtheorem{theorem}[lemma]{Theorem} 
\newtheorem{corollary}[lemma]{Corollary}
\newtheorem{definition}[lemma]{Definition}
\newtheorem{notation}[lemma]{Notation}

\theorembodyfont{\rmfamily}

\newtheorem{example}[lemma]{Example}
\newtheorem{remark}[lemma]{Remark}

\newenvironment{theorem-1}[1][Theorem 1]{\begin{trivlist}
  \item[\hskip \labelsep {\bfseries #1}]}{\end{trivlist}}
\newenvironment{theorem-2}[1][Theorem 2]{\begin{trivlist}
  \item[\hskip \labelsep {\bfseries #1}]}{\end{trivlist}}
\newenvironment{theorem-3}[1][Theorem 3]{\begin{trivlist}
  \item[\hskip \labelsep {\bfseries #1}]}{\end{trivlist}}
\newenvironment{cor-n}[1][Corollary]{\begin{trivlist}
  \item[\hskip \labelsep {\bfseries #1}]}{\end{trivlist}}



\newcommand{\refitem}[1] {~\textit{\ref{#1})}}

\makeatletter
\newcommand\qedsymbol{\hbox{$\boxempty$}}
\newcommand\qed{\relax\ifmmode\boxempty\else
  {\unskip\nobreak\hfil\penalty50\hskip1em\null\nobreak\hfil\qedsymbol
    \parfillskip=\z@\finalhyphendemerits=0\endgraf}\fi}
\makeatother
\newenvironment{proof}[1][{}]{\par\noindent Proof{#1}. }{\qed}

\newenvironment{propositionlist}{\begin{enumerate}}{\end{enumerate}}
\newenvironment{definitionlist}{\begin{enumerate}}{\end{enumerate}}

\usepackage{color}

\definecolor{DarkViolet} {rgb}{0.580392,0.000000,0.827450}

\tikzset{
  doublearrow/.style={draw, thin, double distance=3pt, ->, >=implies},
  ldoublearrow/.style={draw, thin, double distance=3pt, <-, >=implies},
  thirdline/.style={draw, thin, ->, >=implies}, 
  lthirdline/.style={draw, thin, <-, >=implies},
  equality/.style={draw, thin, double distance=3pt, -} 
}

\makeatletter
\def\latearrow#1#2#3#4{%
  \toks@\expandafter{\tikzcd@savedpaths\path[/tikz/commutative diagrams/every arrow,#1]}%
  \global\edef\tikzcd@savedpaths{%
    \the\toks@%
    (\tikzmatrixname-#2)% \noexpand\tikzcd@sourceanchor)%
    to%
    node[/tikz/commutative diagrams/every label] {$#4$}
    (\tikzmatrixname-#3)% \noexpand\tikzcd@targetanchor)
    ;}}
\makeatother

\def\Rho  {\mathsf{E}}
\def\Tau  {\mathsf{D}} % or {circular tensor product}
\def\ct   {ca\-te\-go\-ry-va\-lued trace}
\def\std  {\mathrm{std}}

\def\be            {\begin{equation}}
  \def\complex       {{\ensuremath{\mathbbm C}}}
  \def\complexx      {{\ensuremath{\mathbbm C}^\times_{}}}
  \def\ee            {\end{equation}}
\def\eq            {\,{=}\,}
\newcommand\eqpic[4]{\begin{eqnarray}
    \begin{picture}(#2,#3){}\end{picture}\nonumber\\
    \raisebox{-#3pt}{ \begin{picture}(#2,#3) #4 \end{picture} }
    \label{#1} \\~\nonumber \end{eqnarray} }
\def\Gdiag         {G_{\mathrm{diag}}}
\def\GGVectGHtau   {\Vect(G{\times}G)_{G;H}^{\omega;\theta}}
\def\ie            {\imath_1}
\newcommand\Includepic[1] {{\begin{picture}(0,0)(0,0)
      \scalebox{.32}{\includegraphics{fusV2_#1.eps}}\end{picture}}}
\def\iz            {\imath_2}
\def\lght          {\mathcal L_{G,H,\tau}}
\def\ohr           {\reflectbox{$\rho$}}
\def\sll           {\,{\backslash}\hspace{-2.5pt}{\backslash}\hspace{.6pt}}
\def\srr           {\hspace{-.2pt}\reflectbox{$\backslash$}\hspace{-2.5pt}\reflectbox{$\backslash$}\hspace{-.6pt}}
\def\To            {\,{\to}\,}

\begin{document}

     %        \versionno 

\begin{flushright}
  {\sf MPI/14-65}\\
  {\sf ZMP-HH/14-24}\\
  {\sf Hamburger$\;$Beitr\"age$\;$zur$\;$Mathematik$\;$Nr. 531}\\[2mm]
  %%   December 2014 
\end{flushright}

  \vskip 22mm{}

\begin{center}

  {\Large\bf A TRACE FOR BIMODULE CATEGORIES}

  \vskip 12mm{}

  {\large \  \ J\"urgen Fuchs\,$^{\,a}, \quad$ Gregor Schaumann\,$^{\,b}, \quad$ Christoph Schweigert\,$^{\,c}$
  }

  \vskip 12mm

  \it$^a$
  Teoretisk fysik, \ Karlstads Universitet\\
  Universitetsgatan 21, \ S\,--\,651\,88\, Karlstad \\[9pt]
  \it$^b$
  Max-Planck-Institut f\"ur Mathematik\\
  Vivatsgasse 7, \ D\,--\,53\,111 Bonn
  \\[9pt]
  \it$^c$
  Fachbereich Mathematik, \ Universit\"at Hamburg\\
  Bereich Algebra und Zahlentheorie\\
  Bundesstra\ss e 55, \ D\,--\,20\,146\, Hamburg

\end{center}

\vskip 3.2em

\noindent{\sc Abstract}\\[3pt]
We study a 2-functor that assigns to a bimodule category over a finite 
$\Bbbk$-linear tensor category a $\Bbbk$-linear abelian category. This 
2-functor can be regarded as a category-valued trace for 1-morphisms in the 
tricategory of finite tensor categories. It is defined by a universal property 
that is a categorification of Hochschild homology of bimodules over an algebra.
We present several equivalent realizations of this
2-functor and show that it has a coherent cyclic invariance.
\\
Our results have applications to categories associated to circles in 
three-dimensional topological field theories with defects. This is made 
explicit for the subclass of Dijkgraaf-Witten topological field theories.

\newpage

%%%%%%%%%%%%%%%%%%%%%%%%%%%%%%%%%%%%%%%%%%%%%%%%%%%%%%%%%%%%%%%%%%%%%%%%%%%%%%%% 

\tableofcontents
\newpage

%%%%%%%%%%%%%%%%%%%%%%%%%%%%%%%%%%%%%%%%%%%%%%%%%%%%%%%%%%%%%%%%%%%%%%%%%%%%%%%% 

\section{Introduction}

A useful invariant for algebras $A$ over a field $\Bbbk$ is the zeroth 
Hochschild homology $H\!H_0(A) = A/[A,A]$. 
Here the relevant structure of $A$ is the one as a bimodule over itself,
and accordingly this invariant can be generalized to the $\Bbbk$-vector space 
$H\!H_0(A,M) \,{=}\, M / [A,M]$, where $M$ is an arbitrary $A$-bimodule and
$[A,M]$ is the subspace generated by all
expressions of the form $am \,{-}\, ma$ for $m \,{\in}\, M$ and $a \,{\in}\, A$.

The natural categorical setting for Hochschild homology is the bicategory
$\Algbim_\Bbbk$ which has $\Bbbk$-algebras as objects, bimodules as 1-morphisms and
morphisms of bimodules as 2-morphisms. In this description the Hochschild homology
provides us, for any $A \,{\in}\, \Algbim$, with a functor 
  \begin{equation}
  \label{HH0AM}
  \begin{array}{rcl} 
    \End_{\Algbim}(A) = A\mbox{-bimod} &\!\! \longrightarrow \!\!& \Vect_\Bbbk \,,
    \\{}\\[-9pt]
    M &\!\! \longmapsto \!\!& H\!H_0(A,M) \,.
  \end{array} 
  \end{equation}

This is a functor from an endomorphism category of a bicategory to a category.
It is natural to think about it as a 2-trace (see \cite[Def.\,4.1]{PS2trace} for a 
formulation of this concept) on the bicategory $\Algbim_\Bbbk$, taking
values in the category $\Vect_\Bbbk$. Indeed, it is cyclically invariant:
for any pair of bimodules ${}_AM_B$ and ${}_BN_A$ there is an isomorphism
  \begin{equation}
    H\!H_0 (A,{}_AM_B \,{\otimes_B}\, {}_BN_A) \,\cong\, H\!H_0 (B,{}_BN_A \,{\otimes_A}\, {}_AM_B)
  \end{equation}
of $\Bbbk$-vector spaces.

In the present paper, we are interested in a similar structure one step higher in the
categorical ladder, i.e.\ in a tricategorical setting. More specifically, we
work in the tricategory $\BimCat$, which has finite tensor categories as objects
and bimodule categories, bimodule functors and bimodule
natural transformations as 1-, 2- and 3-morphisms, respectively.
As the higher-categorical analogue of the category $\Vect_\Bbbk$ we take the
2-category $\Categ$ of $\Bbbk$-linear categories.
Thus we wish to obtain, for any finite tensor category $\Cat C$, a 2-functor 
\begin{equation}
  \begin{array}{rcl} 
    \End_{\BimCat}(\Cat C) = \Cat C\mbox{-bimod} &\!\! \longrightarrow \!\!& \Categ 
    \\{}\\[-9pt]
    \Cat M &\!\! \longmapsto \!\!& \circtensor \Cat M 
  \end{array}
\end{equation}
from the bicategory of $\Cat C$-bimodule categories to $\Bbbk$-linear categories
that can be regarded as a generalization of Hochschild homology and shares the
cyclic invariance property of a trace. This is what we achieve in this paper.
We refer to this 2-functor, which takes values in $\Categ$, as a \ct.

The invariance property of a trace is with respect to horizontal composition. In the 
tricategory $\BimCat$, this composition is given by the relative Deligne tensor product 
of bimodule categories. In fact, many of our constructions are inspired by constructions
involving the relative tensor product. 
In order for the relative tensor product to provide a horizontal composition, 
we adopt a definition via a universal property for balanced right exact functors 
with specified adjoint equivalence (Definition \ref{definition:tensor-prod}). 
Various realizations of the relative tensor product are presented in Section 
\ref{sec:constr-tens-prod}; these realizations ensure in particular its existence.
We also present, in Proposition \ref{proposition:twisted-center-tensor}, a new realization 
in terms of a twisted center of a bimodule category, a notion that is introduced in 
Definition \ref{definition:twisted-center}. The twisting is by the double right 
dual functor. In general, we do not require the objects of $\BimCat$ to
be pivotal and clarify some subtleties that have been omitted in the literature.
In Proposition \ref{prop:rel-tensor-with-algebra}, we present another
realization of the relative tensor product, for the case that one of the module
categories is expressed as modules over an algebra object in the monoidal category.

An abstract definition of a \ct\ is given in Definition \ref{definition:circ-tensor-prod}. 
We then provide various realizations: In Theorem \ref{thm:models-circ-tensor}
as particular forms of a twisted center and of a relative tensor product, and
as certain functor categories; and in Proposition \ref{prop:AcaMmod} as
category of bimodules over suitable algebra objects. The cyclicity
  \begin{equation}
  \circtensor (\CMD \,{\boxtimes_{\Cat D}}\, \DNC) \,\simeq\,
  \circtensor (\DNC \,{\boxtimes_{\Cat C}}\, \CMD) 
  \end{equation}
of the \ct\ is established in Proposition \ref{proposition:circ-equiv-1}.

The structures discussed in Sections \ref{sec2}\,--\,\ref{sec4} are motivated by and
have applications to surface defects in three-dimensional topological field theories
of Turaev-Viro type. Turaev-Viro theories are based on spherical fusion categories.
A surface defect separating Turaev-Viro theories of type ${\mathcal A}_1$ and
${\mathcal A}_2$ is labeled by an ${\mathcal A}_1$-${\mathcal A}_2$-bimodule category.
 % XXX  Bindestrich eingefuegt.
Several such surface defects can meet in a line segment, a generalized Wilson line, 
which has to be labeled by an object of a finite abelian category. 
As we explain in Section \ref{sec:5}, the \ct\ enters crucially in the construction 
of that category. As an illustration and non-trivial check that our definition 
of a \ct\ makes sense, we compute the category of Wilson lines for a subclass of 
Turaev-Viro theories, namely Dijkgraaf-Witten theories, for which an 
independent gauge-theoretic construction is available.

In Section \ref{sec4} we put the \ct\ studied in Section \ref{sec3} in its proper 
higher-ca\-tegorical context. We introduce the notion of tricategories with 3-trace, 
show that it provides coherent cyclic equivalences on composable morphisms, and 
establish in Theorem \ref{theorem:ct-is-3trace} that the \ct\ 
provides a 3-trace on the tricategory $\BimCat$ with values in the bicategory $\Categ$.

%%%%%%%%%%%%%%%%%%%%%%%%%%%%%%%%%%%%%%%%%%%%%%%%%%%%%%%%%%%%%%%%%%%%%%%%%%%%%%%% 

\section{The relative tensor product of bimodule categories} \label{sec2} 

We start by collecting and extending pertinent results. Throughout the paper, all 
categories are assumed to be finite, abelian and linear over a fixed field $\Bbbk$.
Further we require all functors and 
natural transformations to be linear unless specified otherwise (we often encounter 
\emph{bi}linear functors out of a Cartesian product of two linear categories).

\subsection{Module categories over finite tensor categories}

First we collect relevant information about module categories over finite tensor categories. 
Recall \cite{FinTen} that a \emph{finite category} $\Cat{C}$ is an abelian category enriched over
$\Bbbk$-vector spaces such that every object has finite length and a projective cover and such that
the set $I$ of isomorphism classes of simple objects is finite.

A \emph{finite tensor category} over a field $\Bbbk$ is a finite $\Bbbk$-linear monoidal 
category $\Cat{C}$ with simple tensor unit $\unit$ and with a left and a right 
duality (see below for our conventions). We will heavily use  these dualities.
Our conventions regarding the duality in a finite tensor category 
$\Cat{C}$ with unit $\unit$ are as follows. A \emph{right dual} object $c^{\vee}$ 
comes with morphisms $\ev{c}\colon c^{\vee} \,{\otimes}\, c \,{\rightarrow}\, \unit$ and 
$\coev{c}\colon \unit \,{\rightarrow}\, c \,{\otimes}\, c^{\vee}$ that satisfy the usual 
snake identities, so that there are natural isomorphisms
  \begin{equation}
  \label{eq:adjunction-hom}
  \Hom_{\Cat{C}}(x \,{\otimes}\, c, y) \cong \Hom_{\Cat{C}}(x, y \,{\otimes}\, c^{\vee})
  \quad \text{ and } \quad
  \Hom_{\Cat{C}}(x,c \,{\otimes}\, y) \cong \Hom_{\Cat{C}}(c^{\vee} \,{\otimes}\, x,y)
  \end{equation}
for all $x,y \,{\in}\, \Cat{C}$. Our convention regarding left duals $^{\vee\!} c$ are analogous. 
A monoidal category with left and right duals is called rigid. 

Next we recall the notion of a module category over a tensor category, which categorifies the 
notion of a module over an algebra, and which is of central interest in this paper. 
For a general introduction, notation and examples, see \cite{Ostrik}.

\begin{definition}{}
  \label{definition:mod-cat}
  Let $\Cat{C}$ be a finite tensor category. 
  A $($left$)$ $\Cat{C}$-\emph{module category} is a  $\Bbbk$-linear finite abelian category $\Cat{M}$
  together with a bilinear functor 
  \begin{equation}
    \act :\quad \Cat{C} \times \Cat{M} \rr \Cat{M},
  \end{equation}
  that is exact in each argument, called the \emph{action} of $\Cat{C}$ on $\Cat{M}$, and 
  with natural isomorphisms
  \begin{equation}
    \label{eq:structures-module-category}
    \mu_{x,y,m}^{\Cat{M}}:\quad (x\otimes y) \act m \rr x\act (y \act m)  \qquad {\rm and} \qquad
    \lambda_m^{\Cat{M}}:\quad \unit \act m \rr m
  \end{equation}
for all $x, y \,{\in}\, \Cat{C}$ and all $m \,{\in}\, \Cat{M}$, called  the module
constraints, such that the obvious 
   % XXX  
     pentagon and triangle diagrams commute {\rm(}see also {\rm \cite[Def.\,2.6]{Ostrik})}.
\end{definition}

When we want to emphasize that $\Cat{M}$ is a left $\Cat{C}$-module category, we denote it 
instead by $\CM$. To emphasize in addition the structure morphisms, we sometimes write 
$(\CM,\mu^{\Cat{M}},\lambda^{\Cat{M}})$; 
whenever it is unambiguous, we denote the constraints just by $\mu$ and $\lambda$.

The notion of a right $\Cat{C}$-module category $\MC$ is analogous, involving a bilinear
functor $\ract\colon \MC \,{\times}\, \Cat{C} \,{\rr}\, \MC$. We denote the constraint for the unit  
of a right module category  by $\rho^{\Cat{M}}_{m}\colon m \ract \unit_{\Cat{C}} \,{\rightarrow}\, m$.
When it is otherwise ambiguous, we denote a left module action on a category $\Cat{M}$ by 
$\mu^{\Cat{M},l}$ or just $\mu^{l}$, and a right module action by $\mu^{\Cat{M},r}$ or $\mu^{r}$.

\begin{definition}
  A left $\Cat{C}$-module category $\CM$ over a finite tensor category $\Cat{C}$ is called \emph{exact}
  if $\Cat{M}$ is finite and if for every projective object $P \,{\in}\, \Cat{C}$ and every object 
  $m \,{\in}\, \Cat{M}$ the object $P \act m$ of $\Cat{M}$ is projective. 
\end{definition}

Exact module categories are characterized by the property that all module functors between 
them are exact \cite[Prop.\,3.16]{FinTen}.
As shown in \cite[Example 3.3]{FinTen}, a module category over a semisimple tensor category is
exact if and only if it is semisimple as an abelian category.

%%%%%%%%%%%%%%%%%%%%%%%%%%%%%%%%%%%%%%%%%%%%%%%%%%%%%%%%%%%%%%%%%%%%%%%%%%%%%%%% 

Next we recall morphisms and 2-morphisms between module categories.

\begin{definition}{\rm \cite{Ostrik}}
A $\Cat{C}$-\emph{module functor} ${\mathsf{F}}\colon \CM \,{\rr}\, \CN$ is a linear
functor ${\mathsf{F}}$ together with natural isomorphisms 
$\phi^{\mathsf{F}}_{x,m}\colon \mathsf{F}(x \act m) \,{\rr}\, x  \act \mathsf{F}(m)$,
such that the obvious 
   % XXX  
     pentagon and triangle diagrams commute {\rm(}see also {\rm \cite[Def.\,2.7(i)]{Ostrik})}.
\end{definition}

\noindent
We sometimes write $(\mathsf{F},\phi^\mathsf{F})$ for a
module functor and call $\phi^\mathsf{F}$ a left module constraint for $\mathsf{F}$.
Whenever it is unambiguous, we denote the constraint just by $\phi$. There is the analogous 
definition for module functors between right $\Cat{C}$-module categories.

%%%%%%%%%%%%%%%%%%%%%%%%%%%%%%%%%%%%%%%%%%%%%%%%%%%%%%%%%%%%%%%%%%%%%%%%%%%%%%%% 

\begin{definition} {\rm \cite{Ostrik} }
  Let  $(\mathsf{F},\phi^{\mathsf{F}}) \colon \CM \,{\rr}\,
  \CN$ and $(\mathsf{G},\phi^{\mathsf{G}})\colon \CM \rr \CN$ be module functors.
  A \emph{module natural transformation} $ \eta\colon \mathsf{F} \,{\rr}\, \mathsf{G}$ is a natural
  transformation such that the diagram
  \begin{equation}
    \label{eq:module-nat-transf}
    \begin{tikzcd}
      \mathsf{F}(x \act m) \ar{rr}{\eta_{x \act m}} \ar{d}[left]{\phi^{\mathsf{F}}_{x,m}} & {} &
      \mathsf{G}(x \act m) \ar{d}{\phi^{\mathsf{G}}_{x,m}}
      \\
      x \act \mathsf{F}(m) \ar{rr}{\id_x \act \eta_{m}} && x \act \mathsf{G}(m)
    \end{tikzcd}
  \end{equation}
  commutes for all $ x \,{\in}\, \Cat{C}$ and all $m \,{\in}\, \Cat{M}$.
\end{definition}

\noindent
The composite of module natural transformations is again a module 
natural transformation. Hence for any pair of module categories $\CM$ and $\CN$, 
the module functors and module natural transformations from $\CM$ to $\CN$ form 
an essentially small category, denoted by $\Funl{\Cat{C}}{\CM,\CN}$.

%%%%%%%%%%%%%%%%%%%%%%%%%%%%%%%%%%%%%%%%%%%%%%%%%%%%%%%%%%%%%%%%%%%%%%%%%%%%%%%% 

\subsection{Bimodule categories}\label{sec:bimodule-categories}

The notions of left and right module categories can be combined 
in a rather obvious way to the following: 

\begin{definition}
  Let $\Cat{C}$ and $\Cat{D}$ be finite tensor categories.
  A $(\Cat{C},\Cat{D})$-\emph{bimodule category} $\CMD$ is  a
  left $\Cat{C}$- and right $\Cat{D}$-module category $\CMD$ together with 
  a family of natural isomorphisms $\gamma_{x,m,y}\colon (x \act m) \ract y \,{\rr}\, x \act
  (m \ract y)$ for all $x \,{\in}\, \Cat{C}$, $y \,{\in}\, \Cat{D}$ and $m \,{\in}\, \Cat{M}$
  such that the diagrams
  \begin{equation*} \hspace*{-.5em}
    \begin{tikzcd}[column sep=large] 
      ((x\,{\otimes}\,y)\act m) \ract d \ar{r}{\gamma_{x\otimes y,m,d}} \ar{d}[left]{\mu^{l}_{x,y,m} \ract d} 
      & (x \otimes y ) \act ( m \ract d)  \ar{dd}[left]{\mu^{l}_{x,y, m \ract d}}     
      \\
      (x \act ( y \act m)) \ract d  \ar{d}[left]{\gamma_{x, y \act m, d}} &
      \\
      x \act (( y \act m) \ract d) \ar{r}{1_{x} \act \gamma_{y,m,d}} & x \act ( y \act (m \ract d))
    \end{tikzcd}
    \hspace*{.8em}
    \begin{tikzcd}[column sep=large] 
      (x \act m) \ract ( d \otimes w) \ar{r}{\gamma_{x, m, d \otimes w}}  \ar{d}[left]{\mu^{r}_{x \act m, d, w}}
      & x \act ( m \ract ( d \otimes w))  \ar{dd}[left]{1_{x} \act \mu^{r}_{m,d,w}}  
      \\
      ( x \act m) \ract d ) \ract w \ar{d}[left]{ \gamma_{x,m,d} \ract 1_{w}}  &  
      \\
      (x \act (m \ract d)) \ract w  \ar{r}{\gamma_{x, m \ract d, w}} & x \act ( ( m \ract d) \ract w)
    \end{tikzcd}
  \end{equation*}
  and
  \begin{equation}
    \label{eq:13}
    \begin{tikzcd}
      (\unit_{\Cat{C}}\act m) \ract \unit_{\Cat{D}}  \ar{d}[left]{\lambda_{m}^{\Cat{M}} \ract \unit}
      \ar{rr}{\gamma_{\unit,m,\unit}}
      & {} & \unit_{\Cat{C}} \act (m \ract \unit_{\Cat{D}}) \ar{dd}{\unit \act \rho_{m}^{\Cat{M}}}
      \\
      m \ract \unit_{\Cat{D}} \ar{d}[left]{\rho_{m}^{\Cat{M}}} &&
      \\
      m && \unit_{\Cat{C}} \act m \ar{ll}{\lambda_{m}^{\Cat{M}}}
    \end{tikzcd}
  \end{equation}
  commute for all $x,y \,{\in}\, \Cat{C}$, $d,w \,{\in}\, \Cat{D}$ and $m \,{\in}\, \Cat{M}$.
  We call the natural isomorphism $\gamma$ the \emph{bimodule constraint} of the bimodule category $\CMD$.
\end{definition}

For functors between bimodule categories one requires a compatible left and right module functor 
structure, while for natural transformations no separate compatibility requirement is needed:

\begin{definition}
  \label{definition:equiv-bimodule-fun-des}
  {\rm (i)}\,
  A \emph{bimodule functor} $\mathsf{F} \colon \CMD \,{\rr}\, \CND$ is a left and a right module functor 
  with left module constraint $\phi^l$ and right module constraint $\phi^r$ such that the diagram
  \begin{equation}
    \label{eq:bimodule-functor-charact}
    \begin{tikzcd}
      {} & \mathsf{F}((x \act m) \ract y) \ar{dl}[left,yshift=5pt]{\mathsf{F}(\gamma_{x,m,y})}
      \ar{dr}[yshift=-3pt]{\phi^{r}_{x \act m, y} } & 
      \\
      \mathsf{F}(x \act (m \ract y)) \ar{d}[left]{\phi^l_{x, m \ract y}} & & \mathsf{F}(x \act m)
      \ract y \ar{d}{\phi^l_{x,m} \ract 1_y}
      \\
      x \act \mathsf{F}(m \ract y) \ar{dr}[left,yshift=-4pt]{1_x \act \phi^r_{m,y}} & & (x \act \mathsf{F}(m))
      \ract y \ar{dl}{\gamma_{x,\mathsf{F}(m),y}} 
      \\
      & x \act (\mathsf{F}(m) \ract y) & 
    \end{tikzcd}
  \end{equation}
  commutes for all objects $x \,{\in}\, \Cat{C}$, $y \,{\in}\, \Cat{D}$ and $m \,{\in}\, \Cat{M}$.
  \\[2pt]
  {\rm (ii)}\,
  A \emph{bimodule natural transformation} $\eta\colon \mathsf{F} \,{\rightarrow}\, \mathsf{G}$ 
  between bimodule functors $\mathsf{F}$ and $\mathsf{G}$ is 
  a natural transformation $\eta\colon \mathsf{F} \rightarrow \mathsf{G}$ that is both
  a left and a right module natural transformation. 
\end{definition}

These structures define for any pair of tensor categories $\Cat{C}$ and $\Cat{D}$ a 2-category 
with $(\Cat{D},\Cat{C})$-bimodule categories as objects, 
bimodule functors as 1-morphisms and bimodule natural transformations as 2-morphisms. 
The compositions are induced from the standard compositions of functors and
natural transformations in the 2-category of categories.
We denote this 2-category by $\BimCat(\Cat{D},\Cat{C})$. 

The most basic example of a module category is the category $\Cat{C}$ regarded as left 
(or right, or bi-) module category over itself, with action given by the tensor product. 
These \emph{regular} (bi-)module categories are denoted by $\LCC$ (respectively, by $\RCC$, $\CCC$).

Given a module category $\CM$ and a monoidal functor $\mathsf{F}\colon \Cat{D} \rightarrow \Cat{C}$, 
the category $\Cat{M}$ acquires a left $\Cat{D}$-module structure by setting 
$d \act m \,{:=}\, \mathsf{F}(d) \act m$,  
with constraints obtained from the constraints of $\Cat M$ and the monoidal structure on $F$.
We call this the \emph{pull back} of $\CM$ along $\mathsf{F}$, and denote it by  $\mathsf{F}^{*}(\CM)$. 
For composable monoidal functors $\mathsf{F}$ and $\mathsf{G}$ one has
$(\mathsf{F} \,{\circ}\, \mathsf{G})^{*}(\CM) \,{=}\, \mathsf{G}^{*} (\mathsf{F}^{*}(\CM))$.
When applied to the module category $\LCC$, this yields a module category structure over $\Cat{D}$ 
on $\Cat{C}$ for every such functor $\mathsf{F}$.

%%%%%%%%%%%%%%%%%%%%%%%%%%%%%%%%%%%%%%%%%%%%%%%%%%%%%%%%%%%%%%%%%%%%%%%%%%%%%%%% 

\subsection{The relative tensor product}\label{sec:relativetensorproduct} 

It is natural to categorify also the notion of tensor product of modules over an algebra. 
In order to achieve this, the notion of a
balanced functor turns out to be essential; this is defined as follows:

\begin{definition} 
  \label{definition:balanced-functors} 
  Let $\MC$ and $\CN$ be right and left module categories, respectively, over a
  finite tensor category $\Cat{C}$, and $\Cat{A}$ be a linear category.
  \\[2pt]
  {\rm (i)}\,
  A bilinear functor $\mathsf{F}\colon \MC \,{\times}\, \CN \rr \Cat{A}$ is called
  $\Cat{C}$-\emph{balanced}, with \emph{balancing constraint} $\beta^{\mathsf{F}}$,
  if it is equipped with a family of natural isomorphisms 
  \begin{equation}
    \beta^{\mathsf{F}}_{m,c,n}:\quad \mathsf{F}(m \ract c \times n) \rr \mathsf{F}(m \times c \act n),
  \end{equation}
  such that  the pentagon diagram 
  \begin{equation}
    \label{eq:balanced-functor}
    \begin{tikzcd} 
      {} & \mathsf{F}(m \ract (x \,{\otimes}\, y) \times n)
      \ar{dl}[left,yshift=4pt]{\mathsf{F}(\mu_{m,x,y}^{\Cat{M}}\times 1_{n})}
      \ar{dr}[yshift=-3pt]{\beta^{\mathsf{F}}_{m,x\,{\otimes}\, y,n}} &
      \\
      \mathsf{F}((m \ract x) \ract y \times n) \ar{d}[left]{\beta^{\mathsf{F}}_{ m \ract x, y, n}}
      & &  \mathsf{F}(m \times (x \,{\otimes}\, y ) \act n)
      \ar{d}{\mathsf{F}(1_{m} \times \mu_{x,y,n}^{\Cat{N}})}
      \\
      \mathsf{F}(m \ract x \times  y \act n) \ar{rr}{ \beta^{\mathsf{F}}_{m,x, y \act n }}
      & & \mathsf{F}( m \times x \act (y \act n))
    \end{tikzcd}
  \end{equation}
  and the triangle diagram
  \begin{equation}
    \label{eq:balanced-unit}
    \begin{tikzcd}
      \mathsf{F}(m \ract \unit_{\Cat{C}} \times n) \ar{r}{\beta^{\mathsf{F}}_{m,\unit,n}}
      \ar{d}[xshift=-48pt]{\mathsf{F}(\rho^{\Cat{M}}_{m} \times 1_{n})}
      & \mathsf{F}( m \times \unit_{\Cat{C}} \act n) \ar{dl}{\mathsf{F}(1_{m} \times\lambda^{\Cat{N}}_{n})} 
      \\
      \mathsf{F}(m \times n) & 
    \end{tikzcd}
  \end{equation}
  commute for all objects $x,y \,{\in}\, \Cat{C}$, $m \,{\in}\, \Cat{M}$ and $n \,{\in}\, \Cat{N}$.
  \\[1pt]
  We often denote the balancing constraint $\beta^{\mathsf{F}}$  simply by $\beta$ if this is unambiguous.
  \\[3pt]
  {\rm (ii)}\,
  Let $\mathsf{F},\mathsf{G} \colon \MC \times \CN \,{\rr}\, \Cat{A}$  be balanced functors. 
  A \emph{balanced natural transformation} $\eta\colon \mathsf{F} \,{\rr}\, \mathsf{G}$
  is a natural transformation $\eta\colon \mathsf{F} \,{\rr}\, \mathsf{G}$ such that the diagrams 
  \begin{equation}
    \label{eq:balanced-nat}
    \begin{tikzcd}
      \mathsf{F}(m \ract c \times n) \ar{rr}{\eta_{m \ract c \times n}} \ar{d}[left]{\beta^{\mathsf{F}}_{m,c,n}}
      & {} & \mathsf{G}(m \ract c \times n) \ar{d}{\beta^{\mathsf{G}}_{m,c,n}} 
      \\
      \mathsf{F}(m \times c \act n) \ar{rr}{\eta_{m \times c \act n}} && \mathsf{G}(m \times c \act n)
    \end{tikzcd}
  \end{equation}
  commute for all objects $c \,{\in}\, \Cat{C}$, $m \,{\in}\, \Cat{M}$ and $n \,{\in}\, \Cat{N}$.
\end{definition}

%%%%%%%%%%%%%%%%%%%%%%%%%%%%%%%%%%%%%%%%%%%%%%%%%%%%%%%%%%%%%%%%%%%%%%%%%%%%%%%% 

Let now $\Cat{C}$ be a finite tensor category and
$\MC$ and $\CN$ be left and right  $\Cat{C}$-module categories, respectively.
A relative tensor product $\MC \boxtensor{\Cat{C}} \CN$ of $\MC$ and $\CN$ is a $\Bbbk$-linear
abelian category that is defined -- up to equivalence of categories --
by a universal property that can be regarded as the analogue of the universal property of
the tensor product of modules over an algebra. However, the relative tensor product 
is only universal with respect to \emph{right exact} functors. Accordingly we introduce the

\begin{notation} 
  For linear categories $\Cat{M}$ and $\Cat{A}$ we denote by $\Funre(\Cat{M}, \Cat{A})$ and
  $\Funle(\Cat{M}, \Cat{A})$ the categories of right and left exact functors, respectively. 
  For $\Cat{N}$ another linear category $\Cat{N}$, $\Funre(\Cat{M} \,{\times}\, \Cat{N}, \Cat{A})$
  is the category of functors that are right exact separately in each argument. 
  \\
  In the case of module categories, $\,\Funbalre(\MC {\times} \CN, \Cat{A})$ denotes the category
  of right exact balanced functors, and the category of right exact module functors
  between module categories $\CN$ and $\CK$ is denoted by $\,\Funlre{\Cat{C}}{\CN, \CK}$. 
  Corresponding versions for left exact functors are denoted in an analogous manner.
\end{notation}

In our definition of the relative tensor product below
we require a fixed adjoint equivalence as part of the data of a relative tensor product;
it thus contains more structure than the relative tensor product of \cite[Def.\,3.3]{ENOfuhom}
or of \cite[Sect.\,2.7]{DaNi}.
Recall that an adjoint equivalence between two objects $b$ and $c$ in a bicategory $\Cat{B}$
consists of two $1$-mor\-phisms $F\colon b \,{\to}\, c$ and $G\colon c \,{\to}\, b$ and
invertible $2$-morphisms $\eta\colon F \,{\circ}\, G \,{\Rightarrow}\, 1_{c}$ and 
$\rho\colon G \,{\circ}\, F \,{\Rightarrow}\, 1_{b}$ that satisfy the  snake identities  in the
Hom-categories $\Cat{B}(b,c)$ and $\Cat{B}(c,b)$.

The assumption that we have adjoint equivalences
constitutes only a mild assumption about the properties of the equivalence, since every 
equivalence in a bicategory can be turned into an adjoint equivalence, albeit not in a 
canonical way, see e.g.\ \cite{Gurski}.  The additional datum is necessary because only then 
one obtains \cite{PivThree} an algebraic tricategory \cite{Gurski} of bimodule categories, 
for which adjoint equivalences in the higher coherence data are required. 

\begin{definition}
  \label{definition:tensor-prod} 
  Let $\MC$ and $\CN$ be left and right $\Cat{C}$-module categories, respectively. 
  A \emph{relative tensor product}
  $(\MC \boxtensor{\Cat{C}} \CN, \mathsf{B}_{\Cat{M},\Cat{N}}, \Psi_{\Cat{M},\Cat{N}}, 
  \varphi_{\Cat{M},\Cat{N}}, \kappa_{\Cat{M},\Cat{N}})$ of $\MC$ and 
  $\CN$ consists of the following data:
  \begin{itemize}
  \item
  a linear category $\MC \boxtensor{\Cat{C}} \CN$ together with a $\Cat{C}$-balanced functor
  $\mathsf{B}_{\Cat{M},\Cat{N}}\colon \MC \,{\times}\, \CN \rr \MC \boxtensor{\Cat{C}} \CN$,
  such that for all linear categories $\Cat{A}$ the induced functor 
  \begin{equation}
    \label{eq:univer-prop-Box}
    \begin{split}
      \Phi_{\Cat{M},\Cat{N};\Cat{A}}:\quad \Funre(\MC \boxtensor{\Cat{C}} \CN, \Cat{A})
      & \longrightarrow \Funbalre(\MC \times \CN, \Cat{A})
      \\[1pt]
      \mathsf{G} & \longmapsto \mathsf{G} \circ \mathsf{B}_{\Cat{M},\Cat{N}} 
    \end{split}
  \end{equation}
  is an equivalence of categories.
 \item 
 For each linear category $\Cat{A}$ we have specified a quasi-inverse
  \begin{equation}
    \label{eq:Psi}
    \Psi_{\Cat{M},\Cat{N};\Cat{A}}\colon\quad
    \Funbalre(\MC {\times} \CN, \Cat{A}) \rightarrow \Funre(\MC \boxtensor{\Cat{C}} \CN, \Cat{A})  
  \end{equation}
 and an adjoint equivalence 
  \begin{equation}
  \varphi_{\Cat{M},\Cat{N};\Cat{A}}\colon\quad
  \id \rightarrow \Phi_{\Cat{M},\Cat{N};\Cat{A}} \, \Psi_{\Cat{M},\Cat{N};\Cat{A}} \quad\mbox{ and }\quad
  \kappa_{\Cat{M},\Cat{N};\Cat{A}}\colon\quad 
  \Psi_{\Cat{M},\Cat{N};\Cat{A}} \, \Phi_{\Cat{M},\Cat{N};\Cat{A}} \rightarrow \id
  \end{equation}
 between the functors $\Phi_{\Cat{M},\Cat{N};\Cat{A}}$ and $\Psi_{\Cat{M},\Cat{N};\Cat{A}}$.
  \end{itemize}
\end{definition}

It follows that a relative tensor product of (finite) module categories $\MC$ and $\CN$ 
is unique up to unique adjoint equivalence, if it exists. 
Its existence has been fully established only recently:

\begin{proposition}{\rm \cite{DSSbal}} 
  Let $\MC$ and $\CN$ be module categories over a finite tensor category $\Cat{C}$. Then the
  relative tensor product $\MC \boxtensor{\Cat{C}} \CN$ over $\Cat{C}$ exists and is unique 
  up to unique adjoint equivalence. 
\end{proposition}

Moreover, for any pair of bimodule categories $\DMC$ and $\END$, the relative tensor product 
$\END \boxtensor{\Cat{D}} \DMC$, is canonically a bimodule category. With the relative tensor
product as composition of 1-morphisms one obtains a tricategory $\BimCat$ with objects
finite tensor categories and bimodule categories, bimodule functors and bimodule
natural transformations as 1-, 2- and 3-morphisms, respectively, see \cite{PivThree}.
This tricategory is a categorification of the bicategory $\Algbim$ of algebras, bimodules 
and bimodule morphisms.

\medskip

We will discuss several models for the relative tensor product in Section 
\ref{sec:constr-tens-prod}. Also, the relative tensor product of $\Bbbk$-linear categories 
over $\Vect$ is the familiar Deligne product \cite{deli}:

\begin{proposition}
  \label{proposition:linear-bimodule-cats}
  Let $\Cat M$ and $\Cat N$ be finite linear categories. 
  Then the  categories $\Cat{M}$ and $\Cat{N}$ are canonically $\Vect$-bimodule categories $($they are exact 
  as bimodule categories iff $\Cat{M}$ and $\Cat{N}$ are semisimple as abelian categories$)$. Conversely,
  any $\Vect$-bimodule  category $\Cat{M}$ is canonically equivalent to $\Cat{M}$ equipped with
  this standard bimodule structure. Further, the relative tensor product over $\Vect$
  exists and coincides with the Deligne product $\Cat{M} \boxtimes \Cat{N}$ of $\Cat{M}$ and $\Cat{N}$.
\end{proposition}

The Deligne product allows for the following characterization of bimodule categories. For a 
monoidal category $\Cat{C}$ we denote by $\Catc{C}$ the monoidal category with the opposite tensor product 
$\cc{\otimes}$, i.e.\ $c \,{\cc{\otimes}}\, d \,{=}\, d \,{\otimes}\, c$ for objects $c,d \,{\in}\, \Cat{C}$;
associativity and unit constraints are canonically obtained from the corresponding constraints of $\Cat{C}$.
It follows that a $(\Cat{C},\Cat{D})$-bi\-mo\-dule category is the same as a left $\Cat{C} \boxtimes
\Catc{D}$-module category \cite[proof\,of\,Prop.\,1.3.10]{GreenPhd}.
Further, a  bimodule functor $\mathsf{F}\colon \CMD \,{\rightarrow}\, \CND$ and 
a bimodule natural transformation between two such functors are the same as 
module functors and module natural transformations, respectively, between the 
corresponding $\Cat{C} \boxtimes \Catc{D}$-module categories.

%%%%%%%%%%%%%%%%%%%%%%%%%%%%%%%%%%%%%%%%%%%%%%%%%%%%%%%%%%%%%%%%%%%%%%%%%%%%%%%% 

\subsection{Constructions of module categories}\label{sec:constr-module-categ}

In the sequel we provide explicit constructions of the structures that we introduced above.
First we consider constructions of module categories from given ones.
They enter in the explicit constructions of the relative Deligne product
that we will present in the beginning of subsection \ref{sec:constr-tens-prod},
and of the \ct\ in Theorem \ref{thm:models-circ-tensor}.

\smallskip

{\bf Duals of bimodule categories}:
Let $\DMC$ be a $(\Cat{D},\Cat{C})$-bimodule category over rigid monoidal categories 
$\Cat{C}$ and $\Cat{D}$. We can then use the two dualities of $\Cat{C}$ and $\Cat{D}$ to
define two $(\Cat{C},\Cat{D})$-bimodule categories $\CMDrd$ and $\CMDld$,
to be called the \emph{right} and \emph{left dual} of $\DMC$, respectively,      
as follows. As categories, they are both $\op{M}$, while the actions are given by
  \begin{equation}
  \label{eq:rd-action}
  c\, \actrd m\, \ractrd \!d := \leftidx{^\vee}\!{d}{} \,\act\, m \,\ract \leftidx{^\vee\!}{c}{}
  \end{equation}
for the case of $m \,{\in}\, \CMDrd$, and 
  \begin{equation}
  \label{eq:ld-action}
  c\, \actld \,m\, \ractld \,d := d^{\vee} \act \,m\, \ract \,c^{\vee}
  \end{equation}
for $m \,{\in}\, \CMDld$, respectively. This agrees with \cite[Def.\,3.4.4]{DSS}.

A basic property of the dual categories is the following \cite[Lemma\,4.1]{PivThree}:
If $\DMC$ and $\CNE$ are bimodule categories over rigid monoidal categories, then 
there are canonical equivalences 
  \begin{equation}
  \leftidx{^{\scriptscriptstyle{\#\!}}}{(\CMDrd)}{} \simeq \DMC \,, \quad\
  (\CMDld)^{{\scriptscriptstyle{\#}}}\simeq \DMC 
  \quad {\rm and} \quad
  \leftidx{^{\scriptscriptstyle{\#\!}}}{(\DMC \boxtensor{\Cat{C}} \CNE)}{}
  \simeq  \ENCld \boxtensor{\Cat{C}} \CMDld
  \end{equation}
of bimodule categories.

\smallskip

{\bf Functor categories as module categories}:
Next we describe various ways of obtaining the structure of
a module category on functor categories of (bi-)module categories:

\begin{itemize} 

\item 
  Let $\CM$ be a left module category and $\Cat{A}$ an arbitrary finite linear category. Then 
  the functor category $\Fun(\CM, \Cat{A})$ has a natural structure of
  a right $\Cat{C}$-module category with action
  \begin{equation}
    \label{eq:action-right-fun}
    (  \mathsf{F}\ract c)(m) := \mathsf{F}(c \act m)
  \end{equation}
  for $\mathsf{F} \,{\in}\, \Fun(\CM, \Cat{A})$, $m \,{\in}\, \Cat{M}$ and $c \,{\in}\, \Cat{C}$. To 
  see that the action constitutes an exact functor, note that
  the functor $ c \act - \colon \Cat{M} \,{\rightarrow}\, \Cat{M}$ is
  exact, which implies, according to the adjoint functor theorem (see  \cite[Prop.\,1.7]{DSSbal}), 
    % for origin see \cite[Thm.\,V.6.2]{MacLane} 
  that the left and right adjoint functors exist. For instance, the right adjoint 
  $(c \act -)^{*}$ provides for $\mathsf{F},\mathsf{G} \,{\in}\, \Fun(\CM, \Cat{A})$ an adjunction 
  \begin{equation}
    \label{eq:adj-hence-exact}
    \Fun(\mathsf{F} \,{\circ}\, (c \act -), \mathsf{G})
    \simeq \Fun(\mathsf{F}, \mathsf{G} \,{\circ}\, (c \act -)^{*}) \,.
  \end{equation}
  It follows that the precomposition with $(c \act -)$ is an exact endofunctor of  $\Fun(\CM, \Cat{A})$.
  The module constraint for $\Fun(\CM, \Cat{A})$ is obtained directly from the module constraint of 
  $\CM$, which induces a natural isomorphism between the functors $(\mathsf{F} \ract c) \ract c'$ and
  $\mathsf{F} \ract (c \otimes c')$ for any $c, c' \,{\in}\, \Cat{C}$.
\item  
  Let $\CM$ be a left module category. Then the functor category $\Fun(\Cat{A}, \!\CM)$
  has a natural structure of a left $\Cat{C}$-module category with action
  $(c \act \mathsf{G})(a) \,{:=}\, c \act \mathsf{G}(a)$.
  The exactness of the action and the module constraint are obtained in the same way as in the 
  preceding example.
\item   
  Similarly, for a right $\Cat{C}$-module category $\MC$, the category $\Fun(\MC, \Cat{A})$ is a 
  left $\Cat{C}$-module category and the category $\Fun(\Cat{A}, \!\MC)$ is a right 
  $\Cat{C}$-module category. 
\item 
  These constructions are compatible with possible bimodule structures on the categories, i.e.
  $\Funl{\Cat{C}}{\CMD, \CNE}$ is a  $ (\Cat{D}, \Cat{E})$-bimodule category and 
  $\Funl{\Cat{C}}{\DMC, \ENC}$ is a $(\Cat{E},\Cat{D})$-bimodule category.
  Additionally, the module actions restrict to actions on the subcategories
  of left and right exact functors in the examples above.
  % The so obtained bimodule categories are in general not exact; to avoid problems with
  % dualities, later on we will restrict to \emph{separable} bimodule categories, 
  % see Definition \ref{definition:separable}.
  
\item 
  The dual categories introduced above are in fact just special cases of the latter 
  constructions: For
   % XXX  
             a finite bimodule category $\DMC$ over finite tensor categories,
  there are canonical equivalences
  \begin{equation}
    \label{eq:duals-as-fun}
    \CMDld \simeq \Funlre{\Cat{D}}{\DMC, \DDD} \quad \text{ and }\quad
    \CMDrd \simeq \Funlre{\Cat{C}}{\DMC, \CCC} 
  \end{equation}
  of bimodule categories, see \cite[Prop.\,3.4.9]{DSS}. 
  The equivalences from the functor categories to the dual bimodule categories are obtained 
  by evaluating  the right adjoint of a functor on the 
  respective unit object.
\end{itemize}

\smallskip

{\bf Module categories from algebra objects}:
Finally \cite{Ostrik}, given an algebra object $A \,{\in}\, \Cat{C}$, the category 
$\Mod_{A}(\Cat{C})$ of $A$-right modules is naturally a left $\Cat{C}$-module category.
Analogously, the category $_{A\!}\Mod(\Cat{C})$ of $A$-left modules is a right $\Cat{C}$-module 
category. Conversely, every 
$\Cat{C}$-module category can be obtained this way, see \cite[Thm.\,3.17]{FinTen} and
    % XXX  genauer zitiert:     
              \cite[Thm.\,2.18]{DSSbal}: 
For every 
   % XXX  
      (finite linear) 
module category $\CM$ over a finite 
tensor ca\-tegory there exists an algebra $A \,{\in}\, \Cat{C}$ together with an equivalence
  \begin{equation}
  \label{equation:Theorem-Ostrik}
    \CM \,{\simeq}\, \Mod_{A}(\Cat{C})
  \end{equation}
of module categories.

%%%%%%%%%%%%%%%%%%%%%%%%%%%%%%%%%%%%%%%%%%%%%%%%%%%%%%%%%%%%%%%%%%%%%%%%%%%%%%%% 

\subsection{Constructions of the relative tensor product}\label{sec:constr-tens-prod}

The relative tensor product $\MC \boxtensor{\Cat{C}} \CN$ of module categories $\MC$ 
and $\CN$  has been introduced in Definition \ref{definition:tensor-prod}
by a universal property. We now present several explicit realizations of 
the relative tensor product, thereby establishing at the
same time its existence. Proposition \ref{proposition:twisted-center-tensor}
gives a new realization in terms of a twisted center.

\begin{enumerate}
\item
First we consider categories of modules internal to $\Cat{C}$. Let us choose algebra objects 
$A,B \,{\in}\, \Cat{C}$ such that $\MC \simeq {}_{A\!}\Mod(\Cat{C})$ and $\CN \simeq \Mod_{B}(\Cat{C})$. 
Then the category ${}_{A\!}\Mod_{B}(\Cat{C})$ of $(A,B)$-bimodules in $\Cat{C}$ can be 
endowed with the structure of a relative tensor product of $\MC$ and $\CN$, as shown in \cite{DSSbal}. 
The universal balanced functor 
  \begin{equation}
    \mathsf B\colon\quad {}_{A\!}\Mod \boxtimes \Mod_{B} \rightarrow {}_{A\!}\Mod_{B}
  \end{equation}
is given by $\mathsf B(m \boxtimes n) \,{=}\, m \,{\otimes}\, n$, with the obvious
$(A,B)$-bimodule structure on the object $m \,{\otimes}\, n\in\Cat{C}$ induced from
the left $A$-module structure on $m$ and the right $B$-module structure on $n$.  
It is clear that the functor $\mathsf B$ is $\Cat{C}$-balanced. 

We refer to \cite[Thm.\,3.3]{DSSbal} 
  %% the structural data of the categories are covered there,
  %% except that they don't require adjoint equivalences in the definition of tensor product. 
for the proof that this construction provides a model for the relative tensor product. Note that 
in \cite{DSSbal} adjoint equivalences are not required in the definition of a relative tensor 
product. Since an adjoint equivalence always exists, we can without loss of generality assume
that the equivalence is endowed with the structure of an adjoint equivalence, 
whereby we obtain a structure of a relative tensor 
product in the sense of Definition \ref{definition:tensor-prod}.

%%%%%%%%%%%%%%%%%%%%%%%%%%%%%%%%%%%%%%%%%%%%%%%%%%%%%%%%%%%%%%%%%%%%%%%%%%%%%%%% 

\item
Another possibility to construct a relative tensor product is to use functor categories.
Consider 
   % XXX  
         finite bimodule categories $\DMC$ and $\CNE$ over finite tensor categories. 
Then as shown in \cite[Cor.\,3.4.11]{DSS}, the functor categories
  \begin{equation}
  \label{eq:tensor-as-fun}
  \begin{array}{l}
  \DMC \boxtensor{\Cat{C}} \CNE \simeq \Funlre{\Cat{C}}{\CMDrd,\CNE} \qquad\mbox{and}
  \\{}\\[-.6em]
  \DMC \boxtensor{\Cat{C}} \CNE \simeq \Funlre{\Cat{C}}{\ENCld, \DMC} \,, 
  \end{array}
  \end{equation}
can be endowed with the structure of a relative tensor product,
with the universal balanced functors as given in \cite[Prop.\,3.5]{ENOfuhom}.

Under the equivalence $\Cat{M} \boxtimes \Cat{N} \,{\simeq}\, \Funre(\op{M},\Cat{N})$
of linear categories, the universal balancing functor 
$\mathsf{B}\colon \DMC \boxtimes \CNE \,{\to}\, \Funlre{\Cat{C}}{\CMDrd,\CNE}$ is given 
by the left adjoint of the forgetful functor 
$ \Funlre{\Cat{C}}{\CMDrd,\CNE} \,{\to}\, \Funre(\op{M},\Cat{N})$ \cite{ENOfuhom}.
For the structure of an adjoint equivalence, the same remarks as for realization $i)$ apply.

\item
In case the tensor category $\Cat{C}$ is semisimple, yet another description of the 
relative tensor product uses the object
\begin{equation}
  \label{coend-alg}
  A := \bigoplus_{u}\, c_{u} \boxtimes c_{u}^{\vee} \,\in \Cat{C} \boxtimes \Catc{C} \,,
\end{equation}
where the summation is over the isomorphism classes of simple objects of $\Cat{C}$.
This object has a natural structure of a Frobenius algebra \cite[Lemma\,6.19]{FFRS}.
As first asserted in \cite[Rem.\,3.9]{ENOfuhom}, the category 
$\MC \boxtensor{\Cat{C}} \CN \simeq \Mod_{A}(\MC \boxtimes \CN)$ (with objects being 
$A$-mo\-dules in the category $\MC \boxtimes \CN$, a notion to be described in detail 
in Definition \ref{def:AmoduleinM} below) has the structure of a relative tensor product 
of $\MC$ and $\CN$ \cite[Prop.\,3.2.9]{Schaum}. In this case the universal balanced functor
is the induction functor, and the forgetful functor is its right adjoint.

%%%%%%%%%%%%%%%%%%%%%%%%%%%%%%%%%%%%%%%%%%%%%%%%%%%%%%%%%%%%%%%%%%%%%%%%%%%%%%%% 
\end{enumerate}

We now proceed to the first result of this paper,
Proposition \ref{proposition:twisted-center-tensor}: We provide an explicit construction 
of the relative tensor product that correctly deals with the double dual functor, a 
subtlety that has been omitted in the literature. This
construction is based on what we call the \emph{twisted center} of two module categories.
We first allow for general twists by monoidal functors $\mathsf{F}$. The specific functor 
relevant for the relative tensor product is the double dual functor, see below. 

\begin{definition}
\label{definition:twisted-center}
Let  $\Cat{C}$ be a rigid monoidal category,
$\mathsf{F}  \colon \Cat{C} \,{\rightarrow}\, \Cat{C}$ a monoidal 
functor and $\CMC$  a $(\Cat{C},\Cat{C})$-bimodule category. 
Denote by $\CMC^{\!\! \mathsf{F}}$ the $(\Cat{C},\Cat{C})$-bimodule category with the
same left action and with the right action pulled back by $\mathsf{F}$. The 
\emph{$\mathsf{F}$-twisted center} is then the center of the bimodule category
$\CMC^{\!\! \mathsf{F}}$.
\end{definition}

Objects of $\,\centCF(\CMC)$ are thus pairs $(X, \gamma)$
consisting of an object $X \,{\in}\, \CMC$ and natural isomorphisms
  \begin{equation}
    \gamma_{c}\colon\quad c \act X\stackrel{\cong}{\longrightarrow}X \ract \mathsf{F}(c)
  \end{equation}  
for all $c \,{\in}\, \Cat{C}$ that satisfy the usual pentagon relation with respect to the
mixed tensor product. The morphisms from $(X,\gamma^{X})$ to $(Y, \gamma^{Y})$ are morphisms
$f\colon X \,{\rightarrow}\, Y$ in $\CMC$ that satisfy
$\gamma^{Y}_{c}( c \act f) \,{=}\,(f \ract \mathsf{F}(c) ) \,{\circ}\, \gamma^{X}_{c}$
for all $c \,{\in}\, \Cat{C}$. In particular,
the category $\mathcal{Z}_{\Cat C}^{\mathsf{Id}}(\Cat C)$ is the ordinary Drinfeld center of $\Cat C$.

The following alternative description is useful:

\begin{lemma}
  \label{lemma:twisted-center-equiv}
  Let $\mathsf{F}$ be a monoidal functor as in Definition $\ref{definition:twisted-center}$.
  Then the $\mathsf{F}$-twisted center is equivalent as a $\Bbbk$-linear abelian category
  to the category of bimodule functors from the pullback $\mathsf{F}^{*}(\CCC)$ of $\CCC$,
  with respect to the left module structure, to $\CMC$, i.e. 
  \begin{equation}
    \label{eq:twisted-center-asbimods-gen}
    \centCF(\CMC ) \simeq \Funl{\Cat{C},\Cat{C}}{ \mathsf{F}^{*}(\CCC), \CMC } \,.
  \end{equation}
\end{lemma}

\begin{proof}
  An object $X \,{\in}\, \centCF(\CMC)$ provides a functor $\mathsf{F}^{*}(\CCC) \,{\rightarrow}\, \CMC$ 
  as follows. On the  object $c \,{\in}\, \Cat{C}$ the functor is defined as
  $\mathsf{F}^{*}(\CCC) \ni c \,\,{\mapsto}\, X \ract c$. It is straightforward to see that the structure 
  of an object in the $\mathsf{F}$-twisted center on $X$ translates into the structure of a bimodule functor.
  Conversely, given a bimodule functor $\mathsf{G}\colon \mathsf{F}^{*}(\CCC) \,{\rightarrow}\, \CMC$,
  the object $\mathsf{G}(\unit_{\Cat{C}})$ defines an object in $\centCF(\CMC)$. 
  These two constructions are mutually inverse up to  canonical natural isomorphisms.
\end{proof}

\medskip

Now we turn to the particular case that is relevant for the relative tensor product of module 
categories: We denote by $\Tau\colon \Cat{C} \,{\rightarrow}\, \Cat{C}$ the monoidal functor 
that is given by the double right dual, 
  \begin{equation}
    \Tau(c) = c^{\vee\vee} . 
  \end{equation}  
Note that we do not assume that $\Cat C$ has a pivotal structure that relates $\Tau$ to the identity 
functor. (But see Section \ref{sec:ex} for a discussion of twisted centers in the presence of a pivotal 
structure.) Hence in particular the twisted center $\centCt(\CCC)$ of the trivial bimodule category
is, in general, not equivalent to the Drinfeld center of $\Cat C$; we call $\centCt(\CCC)$ 
the \emph{twisted Drinfeld center} of $\Cat{C}$. An object of $\centCt(\CCC)$ consists of an
object $X \,{\in}\, \Cat{C}$ and a family of coherent natural isomorphisms 
$\gamma^{X}_{c}\colon c \otimes X \xrightarrow{\,\simeq\,} X \otimes c^{\vee\vee}$.

The ordinary Deligne product $\MC {\boxtimes} \DN$ of two module categories $\MC$ and $\DN$
has a natural structure of a $(\Cat{D},\Cat{C})$-bimodule category. This 
bimodule structure will be implicit in the sequel.   
To simplify notation we abbreviate by $\tCCC$ the category $\Cat{C}$ with regular right module
structure and with left module structure twisted by the double dual functor $\Tau$, i.e. 
\begin{equation}
  \label{eq:Twisted-identity-bimod}
  \tCCC=\Tau^{*}(\CCC) \,.
\end{equation}
We refer to this category simply as the twisted canonical bimodule category; similarly we 
refer to $\Tau$-twisted bimodule categories just as twisted bimodule categories in the sequel.
The previous lemma thus implies: 

\begin{corollary}
  \label{corollary:twisted-center-equ-functor-cat}
Let $\MC$ and $\CN$ be two $\Cat{C}$-module categories. Then the twisted center of 
$\MC\,{\boxtimes}\,\CN$ is equivalent to the category of bimodule functors from $\tCCC$ to 
$\MC \,{\boxtimes}\, \CN$, i.e. 
  \begin{equation}
    \label{eq:twisted-center-asbimods}
    \centCt(\MC \boxtimes \CN ) \simeq \Funl{\Cat{C},\Cat{C}}{ \tCCC, \MC \boxtimes \CN } \,.
  \end{equation}
\end{corollary}

\begin{remark}
The double \emph{left} dual functor of a finite tensor category $\Cat{C}$ is the monoidal functor 
$\Rho \colon \Cat{C} \,{\rightarrow}\, \Cat{C}$ that maps objects $c \,{\in}\, \Cat{C}$ to
$\Rho(c) \,{=}\, \leftidx{^{\vee\vee\!}}{c}{}$. For general finite tensor categories the functors
$\Tau$ and $\Rho$ might not be isomorphic. However, the Radford theorem for finite tensor categories
\cite{EtinAnalogue} is equivalent to the statement that there exists an equivalence
$\Tau^{*}(\CCC) \,{\simeq}\, \Rho^{*}(\CCC)$ of bimodule categories. 
\end{remark}

Next we use the twisted center to describe the relative tensor product of module categories, thereby 
clarifying results in the literature which ignored the twisting \cite[Prop.\,3.8]{ENOfuhom}.
Note first that there is an obvious forgetful functor 
$\mathsf{U}: \centCt(\MC \boxtimes \CN) \,{\rightarrow}\, \Cat{M} \boxtimes \Cat{N}$.
Let $\Cat{A}$ be a linear category. Taking adjoints yields an equivalence
between balanced functors and functors to the twisted center. Concretely, we have

\begin{lemma}
  \label{lemma:adjunction-twisted-centre} 
The right adjoint $\mathsf{F}^{r}_{}$ of a right exact $\Cat{C}$-balanced functor
  $\mathsf{F}\colon \MC \boxtimes \CN \,{\rightarrow}\, \Cat{A}$ 
  is endowed, through the balancing, with the structure of a lift to the twisted center
  $\centCt(\MC \boxtimes \CN)$\,$:$ There exists a functor 
  $\widehat{\mathsf{F}^{r}_{}}\colon \Cat{A} \,{\rightarrow}\, \centCt(\MC \boxtimes \CN)$,
  unique up to unique natural isomorphism, such that
  $\mathsf{U} \,{\circ}\, \widehat{\mathsf{F}^{r}_{}} \,{=}\, \mathsf{F}^{r}_{}$.
The assignment $\mathsf{F} \mapsto \widehat{\mathsf{F}^{r}}$ extends to an equivalence
 \begin{equation}
    \label{eq:twisted-center-via-right-adj}
    \Funbalre(\MC \boxtimes \CN, \Cat{A}) \cong \Funle(\Cat{A}, \centCt(\MC \boxtimes \CN)).
  \end{equation}
\end{lemma}

\begin{proof}
  Assume that $\mathsf{F}\colon \Cat{M}\boxtimes \Cat{N} \,{\to}\, \Cat{A}$ is a 
  right exact functor, and choose objects $a \,{\in} \Cat{A}$ and 
  $m \boxtimes n \,{\in} \Cat{M} \boxtimes \Cat{N}$.
  Then the following diagram either defines the  natural isomorphisms 
  $\gamma^{\mathsf{F}^{r}(a)}$ (invoking the Yoneda lemma) that equip the objects 
  $\mathsf{F}^{r}(a)$ with the structure of objects in the twisted center $\centCt(\MC \boxtimes \CN)$ 
  and thus defines the functor $\widehat{\mathsf{F}^{r}_{}}$ if $\mathsf{F}$ is balanced
  or, conversely, given such a lift to $\centCt(\MC \boxtimes \CN)$ 
  it defines the balancing structure $f$ of $\mathsf{F}$:
  \begin{equation}
    \label{eq:vice-versa}
    \begin{tikzcd}
      \Hom_{\Cat{M} \boxtimes \Cat{N}}(m\boxtimes n,\mathsf{F}^{r}_{}(a) \ract c^{\vee}) \ar{r}{\gamma} \ar{d}{}
      & \Hom_{\Cat{M} \boxtimes \Cat{N}}(m\boxtimes n, \leftidx{^\vee\!}{c}{} \act \mathsf{F}^{r}_{}(a)) \ar{d}{}
      \\ 
      \Hom_{\Cat{M} \boxtimes \Cat{N}}( m \ract c \boxtimes n,\mathsf{F}^{r}_{}(a)) \ar{d}{}  
      & \Hom_{\Cat{M} \boxtimes \Cat{N}}( m \boxtimes c \act n,\mathsf{F}^{r}_{}(a)) \ar{d}{} 
      \\
      \Hom_{\Cat{A}}(\mathsf{F}(m \ract c \boxtimes n),a) \ar{r}{f} 
      & \Hom_{\Cat{A}}( \mathsf{F}(m \boxtimes c \act n),a)
    \end{tikzcd}
  \end{equation}
  Compatibility of the isomorphisms $\gamma$ with the monoidal structure follows from the 
  corresponding compatibility of the natural isomorphisms $f$, and vice versa. 
  Equipped with the structures appearing in the commutative diagram (\ref{eq:vice-versa}), 
  the adjunction between $\mathsf{F}$ and $\mathsf{F}^{r}_{}$ consists of balanced natural 
  isomorphisms; this property characterizes the functor $\widehat{\mathsf{F}^{r}_{}}$ uniquely,
  so the statement follows. 
\end{proof}

\medskip

Next we note the following useful result:

\begin{lemma}
  \label{lemma:exact-functor-cat} 
  Let $\Funl{\Cat{C}}{\CM,\CN}$ be an abelian category of module functors. 
  For any object $m \,{\in}\, \Cat{M}$, the evaluation functor
  $\mathsf{U}_{m}\colon \Funl{\Cat{C}}{\CM,\CN} \,{\to}\, \Cat{N} $ is exact.
\end{lemma}

\begin{proof}
  If we consider $\Cat{M}$ and $\Cat{N}$ as abelian categories, a sequence of functors and 
  natural transformations in $\Fun(\Cat{M},\Cat{N})$ is exact if and only if it is exact 
  at every object $m \,{\in}\, \Cat{M}$ \cite[Sec.\,5.1]{Freyd}. Hence the evaluation functor is 
  exact in this case. Further, for a natural transformation $\nu\colon \mathsf{F} \,{\to}\, \mathsf{G}$
  of \emph{module} functors, the kernel and cokernel of $\nu$, regarded as natural transformation between 
  additive functors, are canonically module functors as well. Indeed, if
  $\mathsf{K}(m) \,{\to}\, \mathsf{F}(m) \,{\xrightarrow{\,\nu(m)\,}}\, \mathsf{G}(m)$
  is the kernel of $\nu$ at $m \,{\in}\, \Cat{M}$, then it follows with the help of the
  exactness of the action of $\Cat{C}$ that both $c \ract \mathsf{K}(m)$ and
  $\mathsf{K}(c \ract m)$ are kernels for 
  $\mathsf{F}(c \ract m) \,{\xrightarrow{\,\nu(c \ract m)\,}}\, \mathsf{G}(c \ract m)$. 
  Thus the universal property of the kernel provides the module constraint of the additive functor
  $\mathsf{K}$. An analogous statement holds for the cokernel. It therefore follows again that 
  a sequence of module functors is exact if and only if it is exact at every object $m \,{\in}\, \Cat{M}$.
  Thus the evaluation functor $\mathsf{U}_{m}$ is exact. 
\end{proof}

\medskip

We are now in a position to establish
the following new realization of of a relative tensor product of module categories:

\begin{proposition}
  \label{proposition:twisted-center-tensor}
  Let $\Cat{C}$ be a finite tensor category and $\MC$ and $\CN$ module categories over
  $\Cat{C}$. Then the twisted center $\centCt(\MC \boxtimes \CN)$ can be
  endowed with the structure of a relative tensor product of $\MC$ and $\CN$
  as follows:
  \begin{propositionlist}
  \item[{\rm (i)}] \label{item:first-univ-bal}
    The universal balanced functor from $\MC \boxtimes \CN\to
    \centCt(\MC \boxtimes \CN)$ is the left adjoint $\mathsf{U}^{l}$ 
   of the forgetful functor. 
  \item[{\rm (ii)}] \label{item:sec-equiv}
    The equivalence $\Psi\colon \Funbalre(\MC \boxtimes \CN, \Cat{A}) \,{\rightarrow}\,
    \Funre( \centCt(\MC \boxtimes \CN), \Cat{A})$ for arbitrary linear categories $\Cat{A}$
    is defined using the lift $\widehat{\mathsf{F}^{r}}$ of the right adjoint $\mathsf{F}^{r}$described
    in Lemma \ref{lemma:adjunction-twisted-centre}: we set 
    $\Psi(\mathsf{F}) \,{:=}\, (\widehat{\mathsf{F}^{r}_{}})^{l}$.
  \item[{\rm (iii)}] \label{item:third-phi}
    The natural isomorphism
    $\kappa\colon \Psi(\mathsf{F})\,{\circ}\,\mathsf{U}^{l} \,{\xrightarrow{\,\cong\,}}\, \mathsf{F}$
    for a balanced functor
    $\mathsf{F}\colon \MC \boxtimes \CN \rightarrow \Cat{A}$ is provided by the isomorphism 
    \begin{equation}
      \label{eq:phi-twisted}
      \Psi(\mathsf{F}) \circ \mathsf{U}^{l} = ( \widehat{ \mathsf{F}^{r}_{}})^{l} \circ \mathsf{U}^{l}
      \xrightarrow{~\cong~} (\mathsf{U} \circ \widehat{ \mathsf{F}^{r}_{}})^{l} = (\mathsf{F}^{r}_{})^{l}_{}
      \cong \mathsf{F} \,.  
    \end{equation}
  \item[{\rm (iv)}] \label{item:forth-kappa}
    For any functor $\mathsf{G}\colon \centCt(\MC \boxtimes \CN) \,{\rightarrow}\, \Cat{A}$ the natural 
    isomorphism $\Psi( \mathsf{G} \,{\circ}\, \mathsf{U}^{l} ) \,{\cong}\, \mathsf{G}$ is obtained from
    \begin{equation}
      \label{eq:kappa-twist}
      \Psi(\mathsf{G} \circ \mathsf{U}^{l})= (\widehat{(\mathsf{G} \mathsf{U}^{l})^{r}_{}})^{l}
      \cong (\widehat{\mathsf{U} \circ G^{r}_{}})^{l} \cong (\mathsf{G}^{r})^{l}.    
    \end{equation}
    $($Here we use that we can choose $\widehat{\mathsf{U} \,{\circ}\, \mathsf{H}} \,{=}\,\mathsf{H}$ for
    $\mathsf{H}\colon \Cat{A} \,{\to}\, \centCt(\MC \boxtimes \CN)$.$)$
  \end{propositionlist}
\end{proposition}

\begin{proof}
  First we show that the forgetful functor $\mathsf{U}$ is exact, implying in particular that a 
  left adjoint exists. Under the equivalence (\ref{eq:twisted-center-asbimods}), the value 
  of $\mathsf{U}$ at a bimodule functor 
  $\mathsf{G} \in  \Funl{\Cat{C},\Cat{C}}{ \tCCC, \MC \boxtimes \CN }$ is given by 
  $\mathsf{U}(\mathsf{G}) \,{=}\, \mathsf{G}(\unit)$. It thus follows from 
  Lemma \ref{lemma:exact-functor-cat} that $\mathsf{U}$ is an exact functor.  
  \\
  Setting $\Cat{A} \,{=}\, \centCt(\MC \boxtimes \CN)$ and taking the identity functor on the 
  right hand side of (\ref{eq:twisted-center-via-right-adj}) shows that 
  the left adjoint of $\mathsf{U}$ is a balanced functor 
  $\mathsf{U}^{l}\colon \MC \boxtimes \CN \,{\rightarrow}\, \centCt(\MC \boxtimes \CN)$. 
  \\
  Lemma \ref{lemma:adjunction-twisted-centre} provides the definition $\widehat{\mathsf{F}^{r}}$. 
  For part (iii) we must show that the prescribed isomorphism
  $\kappa\colon \Psi(\mathsf{F}) \,{\circ}\, \mathsf{U}^{l} \,{\cong}\, \mathsf{F} $ is balanced. 
  This amounts to proving that the outer rectangle in the diagram 
  \begin{equation*}
    \begin{tikzcd}
      \Hom_{\Cat{A}}( (\widehat{\mathsf{F}^{r}_{}})^{l} \mathsf{U}^{l}(m \ract c \boxtimes n),a) \ar{rr}{\kappa}
      \ar{ddd}[left]{f} \ar{dr}{}  &&  \Hom_{}( \mathsf{F}(m \ract c \boxtimes n),a) \ar{ddd}{f}
      \\
      & \Hom_{}( m\boxtimes n,\mathsf{F}^{r}_{}(a) \ract c^{\vee}) \ar{d}{}  \ar{ur}{} &   
      \\
      & \Hom_{}( m\boxtimes n, \leftidx{^\vee\!}{c}{} \act \mathsf{F}^{r}(a)) \ar{dl}{}  \ar{dr}{} & 
      \\
      \Hom_{}( (\widehat{\mathsf{F}^{r}_{}})^{l} \mathsf{U}^{l}(m  \boxtimes c \act n),a) \ar{rr}{\kappa} 
      &&  \Hom_{}( \mathsf{F}(m \boxtimes c \act n),a )
    \end{tikzcd}
  \end{equation*}
  commutes. In this diagram $f$ denotes the balancing constraint of the functor $\mathsf{F}$.
  All arrows are natural isomorphisms; the arrows connecting the $\Hom_{}$ spaces in the
  inner part of the diagram with the ones in the outer rectangle are obtained by a combination of
  the duality morphisms of $\Cat{C}$ and the adjunction between $\mathsf{F}$ and $\mathsf{F}^{r}_{}$.
  It follows that the right inner quadrilateral commutes by definition of the central structure of 
  $\mathsf{F}^{r}_{}(a)$. The left inner quadrilateral commutes by definition of the
  balancing constraint of $\mathsf{F}$ and of $\mathsf{U}^{l}$. The two triangles commute 
  by definition of the natural isomorphism $\kappa$. Hence the outer diagram commutes as well. 
  \\
  Part (iv) follows from the construction in Lemma \ref{lemma:adjunction-twisted-centre}.
\end{proof}

%%%%%%%%%%%%%%%%%%%%%%%%%%%%%%%%%%%%%%%%%%%%%%%%%%%%%%%%%%%%%%%%%%%%%%%%%%%%%%%% 

\section{The \ct\ of bimodule categories}\label{sec3}

In this section we introduce the notion of a \emph{\ct} of bimodule categories 
by means of a universal property. Then we provide several models for the \ct,
which in particular shows the existence of a \ct, and then establish a few basic properties.

\subsection{The twisted center as \ct}

We start with

\begin{definition}
  \label{definition:circular-balanced}
 {\rm (i)}\,
  Let $\CMC$ be a bimodule category over a finite tensor category $\Cat{C}$, and $\Cat{A}$ a linear category. 
  A functor $\mathsf{F}\colon \CMC \,{\rightarrow}\, \Cat{A}$ is called $\Cat{C}$-\emph{balanced 
    with balancing constraint} $\beta^{\mathsf{F}}$ if it is equipped with a family of natural isomorphisms 
  \begin{equation}
    \beta^{\mathsf{F}}_{x,c}:\quad \mathsf{F}(x \ract c) \rr \mathsf{F}( c \act x)
  \end{equation}
  such that the 
   % XXX  
 diagrams
  \begin{equation}
    \begin{tikzcd}
      ~ & \mathsf{F}((x \ract c) \ract d)
      \ar{dl}[left,yshift=7pt] {\beta^{\mathsf{F}}_{x\ract c,d}}
      \ar{dr}{}[yshift=-2pt] {\mathsf{F}((\mu^r_{x,c,d})^{-1})}
      & ~
      \\
      \mathsf{F}(d \act\, (x \ract c)) \ar{d}[left] {\mathrm F(\gamma^{-1}_{d,x,c})}
      & ~ & \mathsf{F}(x \ract\, (c \,{\otimes}\, d))
      \ar{d}{} {\beta^{\mathsf{F}}_{x,c\otimes d}}
      \\
      \mathsf{F}((d \act x) \ract c)
      \ar{dr}[left,yshift=-4pt,xshift=4pt] {\beta^{\mathsf{F}}_{d\act x,c}}
      & ~ & \mathsf{F}((c \,{\otimes}\, d) \act x)
      \ar{dl}[yshift=2pt] {\mathsf{F}(\mu^l_{c,d,x})}
      \\
      ~ & \mathsf{F}(c \act\, (d \act x)) & ~
    \end{tikzcd}
  \end{equation}
 and 
  \begin{equation}
    \begin{tikzcd}
      \mathsf{F}(x \ract \unit_{\Cat C}) \ar{r} {\beta^{\mathsf{F}}_{x,\unit_{\Cat C}}}
      \ar{d}[left,xshift=-1pt] {\mathsf{F}(\rho^{\Cat M}_x)}
      & \mathsf{F}( \unit_{\Cat{C}} \act x)
      \ar{dl}[yshift=2pt] {\mathsf{F}(\lambda^{\Cat M}_x)} 
      \\
      \mathsf{F}(x) & ~ 
    \end{tikzcd}
  \end{equation}
  commute for all objects $c,d \,{\in}\, \Cat{C}$ and $x \,{\in}\, \CMC$.
 \\[2pt]
 {\rm (ii)}\,
  For two balanced functors $\mathsf{F}, \mathsf{G}\colon \CMC \,{\rightarrow}\, \Cat{A}$,
  a \emph{balanced natural transformation} $\eta\colon \mathsf{F}\rightarrow \mathsf{G}$
  is defined as a natural transformation for which a diagram analogous to the one in
  formula $(\ref{eq:balanced-nat})$ commutes.
\end{definition}

We can now define
a \ct\ of a bimodule category via a universal property with respect to balanced functors: 

\begin{definition}
  \label{definition:circ-tensor-prod}
  A \emph{\ct} $(\circtensor \CMC, \mathsf{B}_{\Cat{M}}, \Psi_{\Cat{M}}, \varphi_{\Cat{M}},
  \kappa_{\Cat{M}})$ of a $\Cat{C}$-bimodule category $\CMC$ over a finite tensor category
  $\Cat{C}$ is an abelian category  $\circtensor \CMC$ together with
  a balanced functor $\mathsf{B}_{\Cat{M}}\colon \CMC \,{\rightarrow}\, \circtensor \CMC $
  such that for every linear category $\Cat{A}$ the functor 
  \begin{equation}
    \label{eq:univer-prop-Box-circ}
    \begin{split}
      \Phi_{\Cat{M};\Cat{A}}:\quad \Funre(\circtensor \CMC, \Cat{A})
      &\,\rightarrow\, \Funbalre(\CMC, \Cat{A}) \\
      \mathsf{G} &\,\mapsto\, \mathsf{G} \circ \mathsf{B}_{\Cat{M}} 
    \end{split}
  \end{equation}
  is an  equivalence of categories, and together with the following structure:
  for any linear category $\Cat{A}$ we specify a quasi-inverse
  \begin{equation}
    \label{eq:Psi-circ}
    \Psi_{\Cat{M};\Cat{A}}: \quad \Funbalre(\CMC, \Cat{A}) \rightarrow \Funre(\circtensor \CMC, \Cat{A})  
  \end{equation}
  and an adjoint equivalence $\varphi_{\Cat{M};\Cat{A}}\colon
  \id \,{\rightarrow}\, \Phi_{\Cat{M};\Cat{A}}\, \Psi_{\Cat{M};\Cat{A}}$ and
  $\kappa_{\Cat{M};\Cat{A}}\colon \Psi_{\Cat{M};\Cat{A}}\, \Phi_{\Cat{M};\Cat{A}} \,{\rightarrow}\, \id$
  between $\Phi_{\Cat{M};\Cat{A}}$ and $\Psi_{\Cat{M};\Cat{A}}$.
\end{definition}

It follows directly that the relative tensor product of module categories 
as described in Section \ref{sec:relativetensorproduct} is a special case of the so 
defined trace, namely the one obtained for the $\Cat{C}$-bimodule category $\MC \boxtimes \CN$, i.e. 
  \begin{equation}
  \MC \boxtensor{\Cat{C}} \CN = \circtensor \MC \boxtimes \CN \,.
  \end{equation}
In Theorem \ref{thm:models-circ-tensor} below we show that also a converse statement holds, i.e.\ the \ct\ 
can be expressed in terms of the relative tensor product of bimodule categories.

The \ct\ provides an abstract characterization of the twisted center of a bimodule category 
up to unique adjoint equivalence. It will, however, prove convenient to have different descriptions 
of this category at hand. In particular it is useful to have the separate symbol $\circtensor \CMC$,
since working with the \ct\ we will need to switch between different descriptions of this category.  

In the following statement we provide a list of 
models for the \ct; this list will be extended in Section \ref{sec:circularviamodules}.  

\begin{theorem}
  \label{thm:models-circ-tensor}
  For any bimodule category $\CMC$ over a finite tensor category $\Cat{C}$, the \ct\
  $\circtensor \CMC$ of $\CMC$ exists. Moreover, the abelian category $\circtensor \CMC$ can be realized as  
  \begin{itemize}
  \item[{\rm (i)}]   the twisted center $\centCt(\CMC)$, as given in
        Definition $\ref{definition:twisted-center}$; 
  \item[{\rm (ii)}]  the functor category $\Funbalre(\CMCrd, \Vect)$; 
  \item[{\rm (iii)}] the functor category $\Funl{\Cat{C},\Cat{C}}{\tCCC, \CMC}$;
  \item[{\rm (iv)}]  
    the relative tensor product $\Bim{}{\Cat{M}}{\Cat{C} \boxtimes \Catc{C}}
    \boxtensor{\Cat{C} \boxtimes \Catc{C}} \Bim{\Cat{C} \boxtimes \Catc{C}}{\Cat{C}}{}$.
  \end{itemize}
\end{theorem}

\begin{proof}
  The proof that the category in  part (i) is a model for the \ct\ is analogous to the proof 
  of Proposition \ref{proposition:twisted-center-tensor}.
  For part (ii) note that every right exact functor $\mathsf{F}\colon \op{M} \,{\rightarrow}\, \Vect$
  can be represented by an object $m \,{\in}\, \Cat{M}$ and a natural isomorphism
  $\mathsf{F}(-) \,{\cong}\, \Hom_{\Cat{M}}(-,m)$. As in the proof of Lemma 
  \ref{lemma:adjunction-twisted-centre}, 
     this natural isomorphism allows one to transport the balancing structure on the functor to a 
     half-braiding on the object $m$ so that we have the structure of an object in $\centCt(\CMC)$.
  This way we obtain an equivalence $\centCt(\CMC) \,{\simeq}\, \Funbalre(\CMCrd, \Vect)$.
  Part (iii) follows directly from Lemma \ref{lemma:twisted-center-equiv} specialized to the case at hand. 
  To complete the proof we provide an equivalence between the category in part (iv) 
  and the category $\Funl{\Cat{C},\Cat{C}}{\tCCC, \CMC}$.
  The second equivalence in (\ref{eq:tensor-as-fun}) implies an equivalence
  \begin{equation}
    \Bim{}{\Cat{M}}{\Cat{C} \boxtimes \Catc{C}} \boxtensor{\Cat{C} \boxtimes \Catc{C}} \Bim{\Cat{C}
      \boxtimes \Catc{C}}{\Cat{C}}{} 
    \simeq \Funl{\Cat{C} \boxtimes \Catc{C}} {{}^{\scriptscriptstyle{\#}}(\!\Bimod{}{}{\Cat{C}}{}{\Cat{C}
        \boxtimes \Catc{C}}), \Bim{}{\Cat{M}}{\Cat{C} \boxtimes \Catc{C}}}
  \end{equation}
  of categories.
  Now it is shown in \cite[Lemma\,3.4.5]{DSS} that the operation of changing a right action
  of $\Cat{C}$  to a left action of $\Catc{C}$ and the operation of taking the $\#$-dual commute up to 
  a twist by a double dual functor. For the case at hand this means that seen as a
  $(\Cat{C},\Cat{C})$-bimodule category, the category
  ${}^{\scriptscriptstyle{\#}}(\!\Bimod{}{}{\Cat{C}}{}{\Cat{C} \boxtimes \Catc{C}})$ is just
  the bimodule category $\tCCC$. Hence it follows that there is a canonical equivalence 
  \begin{equation}
    \Funl{\Cat{C} \boxtimes \Catc{C}} {{}^{\scriptscriptstyle{\#}}(\!\Bimod{}{}{\Cat{C}}{}{\Cat{C} \boxtimes \Catc{C}}),
      \Bim{}{\Cat{M}}{\Cat{C} \boxtimes \Catc{C}}}
    \simeq \Funl{\Cat{C},\Cat{C}}{\tCCC, \CMC}\,. 
  \end{equation}
  This provides the claimed equivalence between the categories in part (iii) and (iv).
\end{proof}

\begin{remark}
  (i)\,
  All categories in Theorem \ref{thm:models-circ-tensor} can be regarded 
  as categorifications of the center $\cent(M) \,{=}\, \{ m \,{\in}\, M \,|\, am\,{=}\,ma 
  \text{ for all } a \,{\in}\, A\}$ of a finite-dimensional $A$-bimodule $M$.
  \\
  Indeed, the category in part (ii) corresponds to the following structure for such a
  bimodule $M$. The linear maps from the linear dual $M^{*}$ to $\Bbbk$ that are $A$-invariant 
  can be naturally identified with the center of $M$: A linear map $F\colon M^{*} \,{\to}\, \Bbbk$ 
  is given by inserting an element $m \,{\in}\, M$ into the linear forms $\alpha \,{\in}\, M^{*}$;
  $A$-invariance of $F$ translates in the condition that $m$ is in the center.
  \\
  In view of the discussion in the Introduction, it would in fact be more natural to categorify 
  instead the Hochschild homology of $M$. But note that
   % XXX  
     the zeroth Hochschild homology of an algebra $A$ is isomorphic to the center of $A$.
  For monoidal categories, we are not aware of a satisfactory theory of Hochschild homology.
  Anyhow, rigid monoidal categories can be seen as categorifications of Frobenius algebras, 
  explaining why our categorified version of the 2-trace associated with Hochschild 
  homology is directly linked to a categorification of the center.
  \\[2pt]
  (ii)\,
  The category $\centCt(\CCC)$ appears already in \cite{DSS} as the value of a $0$-framed circle
  in a fully extended framed 2-dimensional TFT  associated to a finite tensor category $\Cat{C}$.
  In \cite{DSS} this value is computed using the composite of evaluation and coevaluation 
  1-mor\-phisms for the dualizable object $\Cat{C}$, resulting in the category in
  Theorem \ref{thm:models-circ-tensor}\,(iv)  in the case that $\Cat{M}$ is 
  the regular bimodule category $\Cat{C}$.
\end{remark}

%%%%%%%%%%%%%%%%%%%%%%%%%%%%%%%%%%%%%%%%%%%%%%%%%%%%%%%%%%%%%%%%%%%%%%%%%%%%%%%% 

\subsection{Functorial properties of the category-valued trace}\label{sec:circ-tens-prod-3-trace}

In the following we establish  first properties of the \ct. We also show 
that for a $(\Vect,\Vect)$-bimodule category $\Cat{M}$, the \ct\ $\circtensor \Cat{M}$ 
is canonically equivalent to $\Cat{M}$. Finally we remark that the trace provides us with a 
functor $\Funlre{\Cat{C},\Cat{C}}{\CMC, \CNC} \to \Funre(\circtensor \CMC, \circtensor \CNC)$ 
that extends to a 2-functor between the 2-category of $\Cat{C}$-$\Cat{C}$-bimodule categories and 
the 2-category $\Categ$ of linear categories.

%%%%%%%%%%%%%%%%%%%%%%%%%%%%%%%%%%%%%%%%%%%%%%%%%%%%%%%%%%%%%%%%%%%%%%%%%%%%%%%% 

\subsubsection{The \ct\ over $\Vect$}

For an endomorphism $\Bbbk \,{\rightarrow}\, \Bbbk$ of the ground field there is a canonical identification 
with its trace. The analogous statement in our categorified situation is

\begin{lemma}
  Let $\Cat{M}$ be a $(\Vect,\Vect)$-bimodule category. There is a canonical equivalence 
  \begin{equation}
    \circtensor \Cat{M} \,\simeq\, \Cat{M}
  \end{equation}
  of the \ct\ of $\Cat{M}$ over $\Vect$ with $\Cat{M}$ itself, as linear categories.
\end{lemma}

\begin{proof}
  According to Proposition \ref{proposition:linear-bimodule-cats}, every $\Vect$-bimodule structure on
  $\VectMVect$ is equivalent to the $\Vect$-bimodule structure that is induced from the linear structure of
  $\Cat{M}$. Consider for a linear category $\Cat{A}$ a functor $\mathsf{F} \,{\in}\, \Funre(\Cat{M},\Cat{A})$.
  It is not hard to see that there is a canonical structure of a $\Vect$-balanced functor on $\mathsf{F}$
  induced from the canonical natural isomorphisms $V \act m \,{\cong}\, m \ract V $
  for all $m \,{\in}\, \Cat{M}$ and all $V \,{\in}\, \Vect$. 
  \\
  Conversely, given a balanced functor $\mathsf{F} \,{\in}\, \Funbalre(\VectMVect, \Cat{A}) $, with 
  balancing structure 
  $f_{V,m}\colon V \,{\otimes}\, \mathsf{F}(m) \,{\xrightarrow{\,\cong\,}}\, \mathsf{F}(m) \,{\otimes}\, V$ 
  for $V \,{\in}\, \Vect$ and $m \,{\in}\, \Cat{M}$, we can forget the balancing constraint and equip
  $\mathsf{F}$ with the canonical balancing constraint, thus obtaining another balanced functor 
  $\widetilde{\mathsf{F}}$. Since a balancing constraint of $\mathsf{F}$ is linear, it is uniquely 
  determined by the isomorphisms
  $f_{\C, m}\colon \C \otimes \mathsf{F}(m) \,{\to}\, \mathsf{F}(m) \,{\otimes}\, \C$ which, in turn,
  constitute the same data as a natural isomorphism $f_{m}\colon \mathsf{F}(m) \,{\to}\, \mathsf{F}(m)$. 
  It is straightforward to see that this natural isomorphism $f$ provides a balanced natural 
  isomorphism $f\colon (\mathsf{F},f) \,{\xrightarrow{\,\cong\,}}\, \widetilde{\mathsf{F}}$ 
  between the original balanced functor $\mathsf{F}$ and the balanced functor $\widetilde{\mathsf{F}}$.  
  This implies that there is a canonical equivalence 
  \begin{equation}
    \label{eq:equial-linear-bal}
    \Psi_{\Cat{M};\Cat{A}}:\quad \Funbalre(\VectMVect, \Cat{A}) \rightarrow \Funre(\Cat{M}, \Cat{A})  
  \end{equation}
  of categories.
  The claim thus follows from the definition of the \ct.
\end{proof}

%%%%%%%%%%%%%%%%%%%%%%%%%%%%%%%%%%%%%%%%%%%%%%%%%%%%%%%%%%%%%%%%%%%%%%%%%%%%%%%% 

\subsubsection{The \ct\ as a 2-functor}

Let $\CMC$ and $\CNC$ be bimodule categories over $\Cat{C}$. The
\ct\ provides a right exact $\Cat{C}$-balanced functor 
$\mathsf{B}_{\Cat{N}}\colon \CNC \,{\rightarrow}\, \circtensor \CNC$. Post-composition with 
this functor defines a functor 
\begin{equation}
  \label{eq:postcomp-functor}
  -\circ \mathsf{B}_{\Cat{N}} \colon\quad
  \Funlre{\Cat{C},\Cat{C}}{\CMC, \CNC} \rightarrow \Funbalre(\CMC, \circtensor \CNC) \,. 
\end{equation}
Hence by the universal property of the trace of $\CMC$ we obtain an induced functor 
\begin{equation}
  \label{eq:postcomp-functor-circten}
  \Funlre{\Cat{C},\Cat{C}}{\CMC, \CNC} \rightarrow  \Funre(\circtensor \CMC, \circtensor \CNC) \,.
\end{equation}
The analogous structure for algebras is the assignment of a  linear map   
between the corresponding centers of the bimodules to a bimodule map. 

Next we show that the trace $\circtensor$ extends to a 2-functor 
(see Definition \ref{lax2func}) from the 2-ca\-te\-gory 
$\BimCat(\Cat{C}, \Cat{C})$ of $(\Cat{C},\Cat{C})$-bimodule categories which we discussed 
in Section \ref{sec:bimodule-categories} to the 2-category $\Categ$ of linear categories.

To fix our notation we recall the definition of a 2-functor:

\begin{definition}\label{lax2func}
  A \emph{$2$-functor} $\mathsf{F}:\Cat B\,{\rightarrow}\, \Cat B'$ between bicategories
  $\Cat B$ and $\Cat B'$ is given by the following data:
  \begin{definitionlist}
  \item
    A function $\mathsf{F}_0\colon \Obj(\Cat B)\,{\to}\, \Obj(\Cat B')$.
  \item 
    For each pair of objects $a,b$ of $\Cat B$, a functor
    $\mathsf{F}_{a,b}\colon \Cat B_{a,b}\,{\to}\, \Cat B'_{\mathsf{F}_0(a), \mathsf{F}_0(b)}$.
  \item
    For each triple of objects $a,b,c$ of $\Cat B$, a natural isomorphism
    $\Phi_{abc}\colon\circ\;(\mathsf{F}_{b,c}\times \mathsf{F}_{a,b})\to \mathsf{F}_{a,c}\;\circ\,$.
    These isomorphisms determine, for all $1$-morphisms $H\colon a\,{\rightarrow}\, b$ and
    $G\colon b \,{\rightarrow}\, c$, an invertible $2$-morphism 
    $\Phi_{G,H}\colon \mathsf{F}_{b,c}(G) \,{\circ}\, \mathsf{F}_{a,b}(H) \,{\to}\, \mathsf{F}_{a,c}(G{\circ}H)$. 
  \item 
    For each object $a$, an invertible $2$-morphism
    $\Phi_a\colon 1_{\mathsf{F}_0(a)} \,{\to}\, \mathsf{F}_{a,a}(1_a)$.
  \end{definitionlist}
  The function $\mathsf{F}_0$, the functors $\mathsf{F}_{a,b}$ and the $2$-morphisms
  $\Phi_{G,H}$ and $\Phi_a$ are required to satisfy the following consistency conditions, 
  where for simplicity we write the diagrams for the case that  $\Cat{B}$ and $\Cat{B}'$ are 2-categories:
  \begin{definitionlist} \addtocounter{enumi}{4}
  \item\label{item:lax-2-1}
    For all $1$-morphisms $H\colon a\,{\rightarrow}\, b$ the diagram
    \begin{equation}
      \label{eq:2-fun-ax-unit}
      \begin{tikzcd}
        { \begin{array}{r}
            \mathsf{F}_{a,b}(H)=\mathsf{F}_{a,b}(H) \circ 1_{F_{0}(a)}  \\
            = 1_{F_{0}(b)}  \circ \mathsf{F}_{a,b}(H) \end{array}}
        \ar{r}{1_{\mathsf{F}_{a,b}(H)} \circ \Phi_{a}} 
        \ar{dr}{\id} \ar{d}[left]{\Phi_{b}\circ 1_{\mathsf{F}_{a,b}(H)}}
        & \mathsf{F}_{a,b}(H) \circ \mathsf{F}_{a,a}(1_{a}) \ar{d}{\Phi_{H,1_{a}}} \\
        \mathsf{F}_{b,b}(1_{b}) \circ \mathsf{F}_{a,b}(H) \ar{r}{\Phi_{1_{b},H}} 
        &\mathsf{F}_{a,b}(H \circ 1_{a})= \mathsf{F}_{a,b}(1_{b} \circ H)=\mathsf{F}_{a,b}(H)
      \end{tikzcd}
    \end{equation}
    commutes.
  \item \label{item:lax2-2}
  For all $1$-morphisms $H\colon a\,{\to}\, b$, $G\colon b\,{\to}\, c$ and $K\colon c\,{\to}\, d$, the diagram 
    \begin{equation}
      \label{2-fun-ax-3-morphism}
      \begin{tikzcd}
        \mathsf{F}_{c,d}(K)\circ \mathsf{F}_{b,c}(G)\circ \mathsf{F}_{a,b}(H)
        \ar{d}{1\circ \Phi_{G,H}} \ar{rr}{\Phi_{K,G}\circ 1}
        &~& \mathsf{F}_{b,d}(K\circ G)\circ \mathsf{F}_{a,b}(H)
        \ar{d}{\Phi_{K\circ G,H}  }\\
        \mathsf{F}_{c,d}(K)\circ \mathsf{F}_{a,c}(G\circ H) \ar{rr}{\Phi_{K,G\circ H}}
        &~& \mathsf{F}_{a,d}(K\circ G\circ H)
      \end{tikzcd}
    \end{equation}
    commutes.
  \end{definitionlist}
  A $2$-functor is called \emph{strict} if the $2$-morphisms $\Phi_{G,F}$ and $\Phi_a$
  are all identities, in which case one has
  \begin{equation}
    \mathsf{F}_{a,c}(G\circ H)=\mathsf{F}_{b,c}(G)\circ \mathsf{F}_{a,b}(H) \qquad{\rm and}\qquad
    1_{\mathsf{F}_0(a)}=\mathsf{F}_{a,a}(1_a) \,.
  \end{equation}
\end{definition}

We are now ready to state:

\begin{proposition}
  \label{proposition:circtens-2-fun}
  Let $\Cat{C}$ be a finite tensor category. 
  \begin{propositionlist}
  \item
    The \ct\ defines a $2$-functor 
    \begin{equation}
      \label{eq:2-fun-circ}
      \circtensor \colon\quad \BimCat(\Cat{C},\Cat{C}) \rightarrow \Categ.
    \end{equation}
  \item 
    The $2$-functor \eqref{eq:2-fun-circ} is representable; it
    can be represented by the twisted Drinfeld double: there exists an equivalence 
    \begin{equation}
      \label{eq:represent-circten}
      \circtensor \,\simeq\, \BimCat( \tCCC,-)
    \end{equation}
    of $2$-functors to $\Categ$.
  \end{propositionlist}
\end{proposition}

\begin{proof}
  We have seen above that the \ct\  defines a functor 
  \begin{equation}
    \label{eq:postcomp-functor-circten2}
    \Funlre{\Cat{C},\Cat{C}}{\CMC, \CNC} \rightarrow  \Funre(\circtensor \CMC, \circtensor \CNC) 
  \end{equation}
  that maps a right exact bimodule functor $\mathsf{F}\colon \CMC \,{\to}\, \CNC$ to a functor 
  $\circtensor \mathsf{F} \colon \circtensor \CMC \rightarrow \circtensor \CNC$ together with a 
  balanced natural isomorphism $\circtensor \mathsf{F} \,{\circ}\, \mathsf{B}_{\Cat{M}}
  \xrightarrow{\,\cong\,} \mathsf{B}_{\Cat{N}} \,{\circ}\, \mathsf{F}$. 
  In the notation used in Definition \ref{definition:circ-tensor-prod}, 
  $\circtensor \mathsf{F} \,{=}\, \Psi (\mathsf{B}_{\Cat{N}} \circ \mathsf{F})$, 
  and the balanced natural isomorphism is 
  \begin{equation}
    \varphi( \mathsf{B}_{\Cat{N}} \circ \mathsf{F}) ^{-1}: \quad \Phi\Psi(\mathsf{B}_{\Cat{N}} \mathsf{F}) 
    = \circtensor \mathsf{F} \circ \mathsf{B}_{\Cat{M}} \rightarrow \mathsf{B}_{\Cat{N}} \circ \mathsf{F} \,.
  \end{equation}
  Now for a pair of composable right exact bimodule functors $\mathsf{F}\colon \CMC \,{\to}\, \CNC$
  and $\mathsf{G}\colon \CNC \to \CKC$ we have two choices for the functor in the bottom line
  of the diagram of functors:
  \begin{equation}
    \label{eq:diag-bal-circten}
    \begin{tikzcd}
      \CMC \ar{r}{\mathsf{F}} \ar{d}{\mathsf{B}_{\Cat{M}}}
      & \CNC \ar{r}{\mathsf{G}} & \CKC \ar{d}{\mathsf{B}_{\Cat{K}}} \\
      \circtensor \CMC  \ar{rr}{}&& \circtensor \CKC,
    \end{tikzcd}
  \end{equation}
  by construction, both the composition  $\circtensor \mathsf{G} \circ \circtensor \mathsf{F}$ 
  and the functor $\circtensor (\mathsf{G} \circ \mathsf{F})$ fit, as both are equipped with 
  corresponding balanced natural isomorphisms that fill the diagram. The universal property 
  of the  \ct\ thus yields a unique natural isomorphism 
  $\varphi_{\mathsf{G},\mathsf{F}}\colon \circtensor \mathsf{G} \circ \circtensor \mathsf{F} 
  \,{\to}\, \circtensor (\mathsf{G} \circ \mathsf{F})$. 
  \\
  {}Further, from the uniqueness statement in the definition of $\varphi_{\mathsf{G},\mathsf{F}}$, 
  it follows that for a third bimodule functor $\mathsf{H}\colon \CKC \,{\to}\, \CLC$, the diagram 
  \begin{equation}
    \label{eq:circten-2-fun-commut}
    \begin{tikzcd}
      \circtensor \mathsf{H}  \circtensor \mathsf{G} \circtensor \mathsf{F}  
      \ar{rr}{\varphi_{\mathsf{H}, \mathsf{G}}( \circtensor \mathsf{F})} 
      \ar{d}{(\circtensor \mathsf{H} )\varphi_{\mathsf{G},\mathsf{F}}} & {} & 
      \circtensor (\mathsf{H} \mathsf{G}) \circtensor \mathsf{F} \ar{d}{\varphi_{\mathsf{H} \mathsf{G} ,\mathsf{F}}}
      \\
      \circtensor \mathsf{H} \circtensor (\mathsf{G} \mathsf{F}) \ar{rr}{\varphi_{\mathsf{H}, \mathsf{G}\mathsf{F}}}
      & {} &
      \circtensor ( \mathsf{H} \mathsf{G} \mathsf{F} )
    \end{tikzcd}
  \end{equation}
  of natural isomorphisms commutes.
  To complete the proof that $\circtensor$ yields a 2-functor, we need to consider the identity functor 
  $\id_{\Cat{M}}\colon \CMC \,{\to}\, \CMC$. Applying the trace to this particular bimodule functor 
  gives, on the one hand, a functor $\circtensor \id_{M}\colon \circtensor \CMC \,{\to}\, \circtensor \CMC$. 
  On the other hand, the identity is clearly a balanced natural isomorphism 
  $1_{\mathsf{B}}\colon \mathsf{B}_{\Cat{M}} \circ \id_{\Cat{M}} \to 
  \id_{\circtensorsmall \Cat{M}} \circ\, \mathsf{B}_{\Cat{M}}$. By the universal property of the \ct\
we thus obtain a unique natural isomorphism 
  $\varphi_{\Cat{M}}\colon \id_{\circtensorsmall \Cat{M}} \,{\to}\, \circtensor \id_{\Cat{M}}$. 
  It follows by uniqueness that $\varphi_{\Cat{M}}$ satisfies axiom \refitem{item:lax-2-1} of 
  Definition \ref{lax2func}. This completes the proof of part (i).
  \\[2pt]
  To establish part (ii), recall from Theorem \ref{thm:models-circ-tensor}
  that $\BimCat(\tCCC, \CMC)$ provides a \ct\ of $\CMC$.  We need to show that using this 
  description of the trace a functor $\mathsf{F} \colon \CMC \,{\to}\, \CNC$ gets mapped to 
  $\BimCat(\tCCC, \mathsf{F})\colon \BimCat(\tCCC, \CMC) \to \BimCat(\tCCC,\CNC)$ under the 
  2-functor in (\ref{eq:2-fun-circ}). By Theorem \ref{thm:models-circ-tensor}
  the universal balancing functor
  $\mathsf{B}\colon \CMC \to \BimCat(\tCCC, \CMC)$ is the left adjoint to the 
  evaluation functor $\mathsf{U}$ at $\unit \,{\in}\, \Cat{C}$. Clearly, the diagram 
  \begin{equation}
    \label{eq:comm-diag-rep}
    \begin{tikzcd}[column sep=large]
      \CNC \ar{rr}{\mathsf{F}^{r}} & {} & \CMC \\
      \BimCat(\tCCC, \CNC) \ar{rr}{\BimCat(\tCCC, \mathsf{F}^{r})} \ar{u}{\mathsf{U}} & {} & \CMC \ar{u}{\mathsf{U}}
    \end{tikzcd}
  \end{equation}
  of functors commutes. 
  The left adjoint of this diagram shows that $\BimCat(\tCCC, \mathsf{F})$ is indeed 
  the value of $\mathsf{F}$ under the 2-functor $\circtensor$ in the current description of the trace. 
  It follows from the universality of the trace that the equivalences 
  $\circtensor \CMC \,{\simeq}\, \BimCat(\tCCC, \CMC)$ from Theorem \ref{thm:models-circ-tensor}
  extend to an equivalence $ \circtensor \,{\cong}\, \BimCat( \tCCC,-)$ of 2-functors.
\end{proof}

%%%%%%%%%%%%%%%%%%%%%%%%%%%%%%%%%%%%%%%%%%%%%%%%%%%%%%%%%%%%%%%%%%%%%%%%%%%%%%%% 

\subsection{The \ct\ via modules over an algebra}\label{sec:circularviamodules}

In the theory of module categories it is often convenient to describe a module category 
over a finite tensor category $\Cat{C}$ in terms of an algebra object in $\Cat{C}$, 
compare (\ref{equation:Theorem-Ostrik}). It is therefore desirable to express also the 
\ct\ in terms of an operation on algebra objects. To achieve this we make use of the notion
of a module object in $\CM$ over an algebra $A \,{\in}\, \Cat{C}$. This notion appears 
   % XXX  
      in \cite{DaNi} (before Lemma 3.2) and
implicitly in \cite{ENOfuhom}; for the reader's convenience we state the full definition here:

\begin{definition}
  \label{def:AmoduleinM}
  Let $\CM$ be a module category over a monoidal category $\Cat{C}$
  and $A \,{\in}\, \Cat{C}$ an algebra. A (left) \emph{module object}
  $(m, \mu_{m})$ in $\CM$ over $A$ is an object $m \,{\in} \CM$ together with a morphism
  $\mu_{m} \colon A \act m \,{\rightarrow}\, m$ in $\Cat{M}$, such that the usual compatibilities 
  with respect to multiplication and unit of $A$ are satisfied.
  Given two module objects $(m, \mu_{m})$ and $(n,\mu_{n})$ in $\CM$ over $A$, a \emph{module morphism}
  $f\colon(m, \mu_{m}) \,{\rightarrow}\, (n,\mu_{n}) $ is a morphism 
  $f \colon m \,{\rightarrow}\, n$ in $\Cat{M}$ such that $\mu_{n} (\id_A \act f) \,{=}\, f \mu_{m}$. 
\end{definition}

\begin{notation}
  The category of left $A$-module objects in $\CM$ and module morphisms is denoted by
  $\Mod_{A}(\CM)$. Analogously we write $\Mod_{A}(\NC)$ for a right $\Cat{C}$-module category $\NC$. 
\end{notation}

The notion of module object allows for yet another novel description 
of the relative tensor product of module categories
   % XXX
     (compare \cite[Lemma\,3.2]{DaNi} for a related statement):

\begin{proposition}
  \label{prop:rel-tensor-with-algebra}
  Let $\MC$ be a right module category over a finite tensor category
  $\Cat{C}$, $A \,{\in}\, \Cat{C}$ an algebra, and $\Mod_{A}(\Cat{C})$ the left
  $\Cat{C}$-module category of $A$-right modules. The relative
  tensor product of the module categories $\MC$ and $\Mod_{A}(\Cat{C})$ can again be
  expressed as a category of modules: it is equivalent to the category of $A$-module objects in $\MC$,
  \begin{equation}
    \label{eq:rel-tensor-cat-modules}
    \MC \boxtensor{\Cat{C}} \Mod_{A}(\Cat{C}) \simeq \Mod_{A}(\MC) \,. 
  \end{equation}
\end{proposition}

\begin{proof}
  According to Equation (\ref{equation:Theorem-Ostrik}), we can choose an algebra $B \,{\in}\, \Cat{C}$ 
  and an equivalence $\MC \,{\simeq}\, {}_{B\!}\Mod(\Cat{C})$ of module categories. Then the
  relative tensor product $ \MC \boxtensor{\Cat{C}} \Mod_{A}(\Cat{C})$ is equivalent to the category
  ${}_{B\!}\Mod(\Cat{C})_{\!A}$ of $(B,A)$-bimodules in $\Cat{C}$. It is straightforward to see that an
  object $x \,{\in}\, {}_{B\!}\Mod(\Cat{C})_{\!A} $ is the same as a $B$-left module object in $\Cat{C}$
  together with an $A$-module action $\mu_{x,A}\colon x \,{\otimes}\, A \,{\rightarrow}\, x$ that is 
  a morphism of $B$-left modules. Thus an object in  $ \MC \boxtensor{\Cat{C}} \Mod_{A}(\Cat{C})$ can
  equivalently be described as an object in $\Mod_{A}({}_{B}\!\Mod(\Cat{C})) \,{\simeq}\,  \Mod_{A}(\MC)$. 
  This assignment extends to the morphisms in the respective categories and thus provides
  the asserted equivalence. 
\end{proof}

\medskip

We need a convenient notation for objects in Deligne products. As an example, consider
the bimodule category $\CCC$ as a left module category over the monoidal category
$\Cat{C} \boxtimes \Catc{C}$. The action gives a monoidal functor
$\Cat{C} \boxtimes \Catc{C}\,{\to}\, \End(\Cat{C})$, and such a functor is determined
by a functor $\Cat{C}\times \Catc{C} \,{\to}\, \End(\Cat{C})$. It is thus sufficient to
know the action of objects of the form $x \boxtimes \widetilde x$;
these act as $c \,{\mapsto}\, x \,{\otimes}\, c \,{\otimes}\, \widetilde x$. 
(This is still true for non-semisimple categories, even though in that case
not every object of $\Cat{C} \boxtimes \Catc{C}$
is a direct sum of factorized objects $x \,{\boxtimes}\, \widetilde{x}$.)
For brevity, we will denote the action of a general object $x \,{\in}\, \Cat{C} \boxtimes \Catc{C}$ 
on $c$ by $x_{\swe{1}} \otimes c \otimes x_{\swe{2}}$.

In the next proposition we provide two characterizations of the \ct\ 
in terms of algebra objects. For the first characterization we use the \emph{canonical algebra} 
$A^{\mathrm{can}}$ of $\Cat{C} \boxtimes \Catc{C}$, which by definition is the algebra 
object in $\Cat{C} \boxtimes \Catc{C}$ for which the category of $A^{\mathrm{can}}$-modules 
in $\Cat{C} \boxtimes \Catc{C}$ is equivalent to the canonical bimodule category $\CCC$. 
Such an algebra exists by equation (\ref{equation:Theorem-Ostrik}). If $\Cat{C}$ is 
semisimple, then  $A^{\mathrm{can}} \,{=}\, \bigoplus_{u} c_{u} \boxtimes c_{u}^{\vee}$ 
is the algebra already considered in \eqref{coend-alg}.

\begin{proposition}
  \label{prop:AcaMmod}
  Let $\CMC$ be a bimodule category over a finite tensor category $\Cat{C}$, and let 
  $A \,{\in}\, \Cat{C} \,{\boxtimes}\, \Catc{C}$ be
  an algebra object such that $\Bim{}{\Cat{M}}{\Cat{C} \,{\boxtimes}\, \Catc{C}}$ 
  is equivalent to ${}_{A\!}\Mod(\Cat{C} \,{\boxtimes}\, \Catc{C})$. Then  
  the \ct\ $\circtensor \CMC$ is equivalent to the following categories: 
  \\[4pt]
  {\rm (i)} the category ${}_{A^{\mathrm{can}\!}}\Mod_{A}(\Cat{C} \boxtimes \Catc{C})$
  of $(A^{\mathrm{can}},A)$-bimodules in $\Cat{C} \boxtimes \Catc{C}$;
  \\[4pt]
  {\rm (ii)} the category  $\Mod_{A}(\CCC)$, with objects being pairs consisting
  of an object $c \,{\in}\, \Cat{C}$ and a morphism 
  $A_{\swe{1}} \,{\otimes}\, c \,{\otimes}\,  A_{\swe{2}} \,{\rightarrow}\, c$
  that is compatible with the multiplication and the unit of $A$. 
\end{proposition}

\begin{proof}
  For part (i) we use the equivalence
  $\circtensor \CMC \,{\simeq}\, \Bim{}{\Cat{M}}{\Cat{C} \boxtimes \Catc{C}}
  \boxtensor{\Cat{C} \boxtimes \Catc{C}} \Bim{\Cat{C} \boxtimes \Catc{C}}{\Cat{C}}{}$ from 
  Theorem \ref{thm:models-circ-tensor}.  The statement then follows from the construction of 
  the relative tensor product as a category of bimodules, see Section \ref{sec:constr-tens-prod}. 
  Part (ii) is implied by the description of the trace in Theorem \ref{thm:models-circ-tensor}\,(iv) 
  when combined with Proposition \ref{prop:rel-tensor-with-algebra}.
\end{proof}

%%%%%%%%%%%%%%%%%%%%%%%%%%%%%%%%%%%%%%%%%%%%%%%%%%%%%%%%%%%%%%%%%%%%%%%%%%%%%%%% 

\subsection{Twisted center and \ct\ in illustrative situations}\label{sec:ex}

Let us now discuss a few specific contexts in which one encounters the structures studied 
above. First we consider the relation between the twisted Drinfeld center and (generalized) 
pivotal structures on a finite tensor category $\Cat{C}$. 

\smallskip

{\bf Monoidal structures on \ct s via pivotal structures}.
Let $\Cat{C}$ be a finite tensor category equipped with an equivalence 
$\mathsf{P}\colon \tCCC \,{\xrightarrow{\,\simeq\,}}\, \CCC$ of bimodule categories.
Such an equivalence is for instance induced, via pullback of the left module structure,
by a pivotal structure on $\Cat C$, i.e.\ by a monoidal natural isomorphism 
$x^{\vee\vee} \,{\cong}\, x$ for all $x \,{\in}\, \Cat{C}$. More generally, according to Lemma 
\ref{lemma:twisted-center-equiv}, the object $P \,{:=}\, \mathsf{P}(\unit) \,{\in}\, \Cat{C}$ 
comes equipped with a family of natural isomorphisms
  \be
  \label{xP=Pxvv}
  p_{x} \colon \quad x \otimes P \xrightarrow{~\cong~} P \otimes x^{\vee\vee}
  \ee
for $x \,{\in}\, \Cat{C}$ that is compatible with the monoidal structure,
and can thus be regarded as an object in $\centCt(\CCC)$.
Assume that $\mathsf{Q} \colon \CCC \,{\xrightarrow{\,\simeq\,}}\, \tCCC$ is a bimodule 
functor that together with $\mathsf{P}$ forms an equivalence of bimodule categories.
Again by  Lemma \ref{lemma:twisted-center-equiv}, $\mathsf{Q}$ is determined by the 
object $Q \,{:=}\, \mathsf{Q}(\unit) \,{\in}\, \cent(\tCCC)$ which is 
equipped with a family of coherent natural isomorphisms 
$q_{x} \colon x^{\vee \vee} \,{\otimes}\, Q \,{\xrightarrow{~\cong~}}\, Q \,{\otimes}\, x$.
The natural isomorphisms 
$\mathsf{P} \,{\circ}\, \mathsf{Q} \,{\xrightarrow{~\cong~}}\, \id_{\!\CCC}$ and 
$\mathsf{Q} \,{\circ}\, \mathsf{P} \,{\xrightarrow{~\cong~}}\, \id_{\tCCC}$ 
then yield in particular
  %% (using first the unit constraint in $\Cat{C}$ and then that $\mathsf{P}$ is a module functor)
an isomorphism 
  \be
  \mathsf{P}(\mathsf{Q}(\unit)) \cong \mathsf{P}( Q \act \unit) \cong Q \act \mathsf{P}(\unit)
  = Q \otimes P  \xrightarrow{~\cong~} \unit \,,
  \ee
and similarly an isomorphism $P \,{\otimes}\, Q \,{\xrightarrow{\,\cong\,}} \unit$,
which establish $Q$ as inverse of $P \,{\in}\, \Cat{C}$. 

Conversely, an invertible object ${P} \,{\in}\, \Cat{C}$ with a family \eqref{xP=Pxvv}
of natural isomorphisms determines an equivalence $\mathsf{P}\colon \tCCC \,{\to}\, \CCC$.
We refer to these data as a \emph{generalized pivotal structure}.

A generalized pivotal structure $\mathsf{P}$ induces for each bimodule category $\CMC$ 
an equivalence $\circtensor \Cat{M} \,{\simeq}\, \cent(\CMC)$. In particular there is an 
equivalence $\mathsf{F}_{\mathsf{P}}\colon \centCt(\CCC) \,{\simeq}\, \cent(\Cat{C})$. 
This equivalence $\mathsf{F}_{\mathsf{P}}$ is monoidal for the monoidal structure 
$\otimes_{\mathsf{P}}$ on $\centCt(\CCC)$ that is obtained by pull-back of the 
monoidal structure of $\cent(\Cat{C})$ along $\mathsf{F}_{\mathsf{P}}$, i.e.
  \be
  x\otimes_{\mathsf{P}} y = x \otimes P \otimes y 
  \ee
for $x,y \,{\in}\, \centCt(\CCC) $, with the structure of $x \,{\otimes}\, P \,{\otimes}\, y$ 
as an object of $\centCt(\CCC)$ induced by the respective structures of $x,y$ and $P$.

Next assume we are given a bimodule category $\CMC$ and a monoidal functor 
$\mathsf{F}\colon \Cat{C} \,{\to}\, \Cat{C}$. We can  recover the twisted $\mathsf{F}$-twisted 
center of $\CMC$ from Definition \ref{definition:twisted-center} as a trace: For every 
monoidal endofunctor $\mathsf G$ of $\Cat C$ there is an equivalence 
$\mathcal{Z}_{\Cat{C}}^{\mathsf{D}} (\Bimod{}{\Cat{C}}{\Cat{M}}{\mathsf{G}}{\Cat{C}})
\,{\simeq}\, \mathcal{Z}_{\Cat{C}}^{\mathsf{D\circ G}} (\Bimod{}{\Cat{C}}{\Cat{M}}{}{\Cat{C}}) $.
Combining Theorem \ref{thm:models-circ-tensor}\,(i) with the isomorphism 
$\Tau \,{\circ}\, \Rho \,{\cong}\, \id$, where $\Rho$ is the double left dual functor, 
thus implies that
  \begin{equation}
    \label{eq:F-twisted-as-ct}
    \centCF(\CMC) \simeq \circtensor \CMRFC \,.   
  \end{equation}
Here $\CMRFC$ is the bimodule category $\CMC$ with right $\Cat{C}$-action twisted by 
$\Rho \,{\circ}\, \mathsf{F}$.

\smallskip

{\bf The example of vector spaces graded over a group}.
As another illustration, consider $G$-graded vector spaces over an algebraically closed
field $\Bbbk$, a situation that will be studied again in Section \ref{sec:5} from a different 
perspective. For $G$ a finite group, a three-cocycle $\omega\,{\in}\, Z^3(G,\Bbbk^{\times})$
defines the structure of a finite tensor category  $\Cat{C} \,{=}\, \Vect(G)^\omega_{\Bbbk}$ 
on the linear category of $G$-graded $\Bbbk$-vector spaces, with the associator being 
furnished by the cocycle $\omega$.

The indecomposable bimodule categories over $\Cat{C}$ are classified in \cite[Example\,2.1]{OsFin}
(see also \cite[Prop.\,4.1]{FinTen}): they are characterized by a subgroup
embedding $\imath\colon H \,{\rightarrow}\, G \,{\times}\, G$ together with a 2-cochain $\theta$
on $H$ satisfying ${\rm d} \theta \eq \imath_1^*\omega \,{\cdot}\, \imath_2^*\omega^{-1}$, with
$\imath_1 \,{=}\, p_1 \,{\circ}\, \imath $ and $\imath_2 \,{=}\, p_2 \,{\circ}\, \imath$,
$p_1$ and $p_2$ being the projections from $G \,{\times}\, G$ to its two factors.
These data define a $(\Cat{C},\Cat{C})$-bimodule category $\Cat{M}(H, \theta)$.
Such a bimodule category can, in turn, be realized as the category of right modules over an 
algebra (determined up to Morita equivalence) $A_{H,\theta}$ in the finite tensor category
$\Vect(G)^\omega \boxtimes\Vect(G)^{\omega^{-1}}$. As an object,
  \be
  \label{AHtheta}
  A_{H,\theta} = \bigoplus_{h\in H}\, S_{(\ie(h)^{-1},\iz(h)^{-1})} \,,
  \ee 
where $S_{(\gamma_1^{},\gamma_2^{})} \,{\cong}\, \Bbbk$ with 
$(\gamma_1^{},\gamma_2^{}) \,{\in}\, G\,{\times}\, G$ is the ground field in homogeneous 
degree $(\gamma_1^{},\gamma_2^{})$, and the algebra structure on this object is 
the one of a twisted group algebra with the twist given by the cochain $\theta$.

\begin{example}
  \label{example:DW-algebraic}
  Invoking Proposition \ref{prop:AcaMmod}, we learn that the trace  
  of the bimodule category $\Cat{M}(H, \theta)$ has the following explicit description. 
  There is an equivalence
  \be
  \circtensor \Cat{M}(H, \theta) \,\simeq\, 
  {}_{A_{H,\theta}\!}\Mod_{A_{\Gdiag}}\!\! \big(\Vect(G)^\omega \boxtimes\Vect(G)^{\omega^{-1}}\big)
  \ee
  to the category of $A_{\Gdiag}$-$A_{H,\theta}$-bimodules in 
  $\Vect(G)^\omega \boxtimes\Vect(G)^{\omega^{-1}}$, with the algebra $A_{\Gdiag}$ given by
  \be
  \label{Gvtheta}
  A_{\Gdiag} = \bigoplus_{g\in G} S_{(g,g)} \,,
  \ee
  as an object of $\Vect(G)^\omega \boxtimes\Vect(G)^{\omega^{-1}}$ 
  and with the product on $A_{\Gdiag}$ being the one of the group algebra of $G$. 
\end{example}

%%%%%%%%%%%%%%%%%%%%%%%%%%%%%%%%%%%%%%%%%%%%%%%%%%%%%%%%%%%%%%%%%%%%%%%%%%%%%%%% 

\section{Categories assigned to circles in  TFT with defects}\label{sec:5}

One area in which some of the structures presented in Sections  \ref{sec2}\,--\,\ref{sec3} 
appear naturally is topological field theory or, more precisely, topological field theories 
with defects. Specifically, a three-dimensional 3-2-1-extended topological field theory
is a symmetric monoidal 2-func\-tor from an extended cobordism bicategory 
$\mathsf{Cob}$ to the bicategory 2-$\Vect$. In particular, it associates a linear 
category $\mathrm{tft}(\Sigma)$ to any compact one-manifold $\Sigma$. Several
variants of cobordism categories can be considered. In the simplest case, objects
are compact oriented smooth one-manifolds, 1-morphisms are two-manifolds with boundary
(with the boundary components to be interpreted as ``gluing boundaries'') and 
2-morphisms are three-manifolds with corners. 

Among the more general cobordism categories considered in the literature are in particular 
manifolds with singularities or with distinguished submanifolds of various codimension, 
which are to be thought of as the loci of various defects
(see e.g.\ \cite[Sect.\,4.3]{Lurie} and \cite{AFT}).
In this setup the objects are one-manifolds with distinguished points at which defects
are located. (It is then in fact natural to extend the cobordism bicategory even further
in such a way that intervals with distinguished points are objects as well; this amounts 
to the inclusion of boundary conditions. This extension of the cobordism category is,
however, not necessary for understanding the motivation of the present paper 
from topological field theory). 

By invoking the monoidal structure of $\mathsf{Cob}$
as well as fusion of defects and fusion of defects to boundaries, one can restrict
one's attention to two basic one-manifolds: an interval
without defect points and a circle with one defect point.

More specifically, what we have in mind are three-dimensional topological field theories 
of Turaev-Viro type. Such theories are constructed from spherical fusion categories 
$\Cat C$. It is well understood that the Turaev-Viro theory based on $\Cat C$ associates 
to a circle without defects the Drinfeld center $\cent(\Cat C)$. Surface defects 
between possibly different Turaev-Viro theories are well understood, too, see e.g.\
\cite{FuSchwVal}: a surface defect that separates two three-dimensional regions labeled
by $\cent(\Cat C)$ and $\cent(\Cat D)$, respectively, is labeled by a
$\Cat C$-$\Cat D$-bimodule category. We regard these bimodule categories
as 1-morphisms in a 3-category of spherical fusion categories. Fusion of defects is 
then composition of these 1-morphisms, and is thus the relative tensor product, as given
in Definition \ref{definition:tensor-prod}, over the relevant fusion category.
As a consequence, the data associated to a circle $S$ with a defect point $p$ consist 
of a spherical fusion category $\Cat C$ assigned to the open interval $S{\setminus}p$
and a $\Cat C$-bimodule category $\CMC$ assigned to the defect point. We now explain
that the $\Bbbk$-linear category $\mathrm{tft}(S,p)$ which the topological field theory 
assigns to such a circle should be the trace of the bimodule category,
  \be
  \label{cat4circle}
  \mathrm{tft}(S,p) = \circtensor \CMC .
  \ee

We can think about the labeled circle $(S,p)$ as a diagram in a symmetric monoidal tricategory 
as follows: it is a coevaluation for the fusion category $\Cat C$, followed by the 1-morphism
$\CMC$ and an evaluation for $\Cat C$. Since the coevaluations are given by $\Cat C$,
seen as an appropriate bimodule category, we naturally end up with the $\Bbbk$-linear 
abelian category $\Bim{}{\Cat{M}}{\Cat{C} \boxtimes \Catc{C}}
\boxtensor{\Cat{C} \boxtimes \Catc{C}} \Bim{\Cat{C} \boxtimes \Catc{C}}{\Cat{C}}{}$,
i.e.\ with the model (iv) for the trace in Theorem \ref{thm:models-circ-tensor}.

Alternatively, we may consider what happens when the circle is folded to an interval 
in such a way that one of the resulting end points is $p$. The category associated to
an interval can be described in two ways: as a Hom-category in a bicategory (see e.g.\
\cite[Eq.\,(2.14)]{FSV14}), leading to the model (ii) of 
Theorem \ref{thm:models-circ-tensor}, or as a relative Deligne tensor product of a 
right module and left module category, yielding the model (iv) of Proposition 
Theorem \ref{thm:models-circ-tensor}.
On the other hand, results in other contexts suggest a (derived version of) Hochschild 
homology, see \cite{GTZ}, which in our setting reduces to a Drinfeld center as in model
(i) of Theorem \ref{thm:models-circ-tensor}.

Note that for the particular case of the `invisible' defect, corresponding to the regular
$\Cat C$-bi\-mo\-dule $\CCC$, for spherical $\Cat{C}$ the \ct\ \eqref{cat4circle} 
is just the Drinfeld center of $\Cat C$ (see Theorem \ref{thm:models-circ-tensor}\,(i)), 
as appropriate.

%%%%%%%%%%%%%%%%%%%%%%%%%%%%%%%%%%%%%%%%%%%%%%%%%%%%%%%%%%%%%%%%%%%%%%%%%%%%%%%% 

\subsection{Dijkgraaf-Witten theories}\label{sec:gauge}

We now turn to a particular subclass of Turaev-Viro theories, Dijkgraaf Witten theories.
These theories admit an explicit gauge-theoretic construction which allows to compute
the category $\mathrm{tft}(S,p)$ independently and thus to corroborate formula
\eqref{cat4circle} further. A similar strategy has been employed
for categories associated to intervals in Section 3.5 of \cite{FSV14}.

The input data for (a three-dimensional stratum of) a Dijkgraaf-Witten theory are
a finite group $G$ and a 3-cocycle $\omega\,{\in}\, Z^3(G,\complexx)$. Cohomologous
cocycles $\omega$ yield equivalent theories, so that without loss of generality 
$\omega$ can be assumed to be normalized,
  \be
  \omega(e,g_2^{},g_3^{}) = 1 \,,
  \label{omeganorm}
  \ee
as well as (see e.g.\ \cite[Sect.\,3]{BarWes} and \cite{FGSV99}) 
   % bawe2 e.g. text after lemma 3.11
to carry the symmetries of a tetrahedron, i.e.\
  \be\label{s4symm}
  \varpi(a,b,c) = \varpi(a^{-1},ab,c)^{-1}_{} = \varpi(ab,b^{-1},bc)^{-1}_{} = \varpi(a,bc,c^{-1})^{-1}_{} .
 %% \
 %% also derived identities:    
 %% = \varpi(c^{-1},b^{-1},a^{-1}) = \varpi(abc,c^{-1},b^{-1})^{-1}_{} = \varpi(b^{-1},a^{-1},abc)^{-1}_{} 
 %%                              %% \tau_{23}\tau_{34}\tau_{23}         \tau_{12}\tau_{23}\tau_{12}
 %% the three transformations in the first line correspond to $\tau_{12}, \tau_{23}$
 %% and $\tau_{34}$ in (7.8) of \cite{FGSV99}.
  \ee 

In the sequel we determine the category that a 3-2-1-extended Dijkgraaf-Witten theory 
assigns to a circle with a defect point through a gauge-theoretic construction.
To compare it to the algebraic description, we keep in mind that, as a 
Turaev-Viro theory, a Dijkgraaf-Witten theory is based on the
spherical fusion categories given by $G$-graded vector
spaces with associativity constraint twisted by $\omega$ and with their 
canonical spherical structure. As we will see, the result coincides with the \ct\ 
of the relevant bimodule categories obtained in Example \ref{example:DW-algebraic}.

%%%%%%%%%%%%%%%%%%%%%%%%%%%%%%%%%%%%%%%%%%%%%%%%%%%%%%%%%%%%%%%%%%%%%%%%%%%%%%%% 
\medskip

The gauge-theoretic construction of the Dijkgraaf-Witten theory based on a finite group $G$ and 
3-cocycle $\omega$ as a 3-2-1-extended topological field theory proceeds in two steps: 
for the first step, one notes that 1-morphisms in the extended cobordism category are spans
of manifolds and 2-morphisms are (equivalence classes of) spans of spans. Applying the groupoid-valued
contravariant 2-functor $\mathrm{Bun}_G$ to this structure yields groupoids, cospans of groupoids
and equivalences of cospans of cospans.
The actual values of the field theory are then obtained via twisted linearization, in which
the twist involves a groupoid cocycle obtained \cite{Fr95,Mo14} by transgressing $\omega$ in 
the sense of \cite{Wi08}.
This construction can be extended \cite{FSV14}, with help of the notions of relative 
manifolds and relative bundles, to three-manifolds with boundaries or defect surfaces.

In this geometric description the category for a circle $S$ with one marked point $p$
is obtained as follows (see Section 3.1 of \cite{FSV14}).
To the interval $S{\setminus}\{p\} \,{\subset}\, S$ one has to associate the group
$G$ with cocycle $\omega$, which has the physical interpretation of a (topological) 
bulk Lagrangian. To the marked point one assigns a group homomorphism 
$\imath\colon H\To G\,{\times}\, G$ and a 2-co\-chain $\theta \,{\in}\, C^2(H,\complexx)$ 
on $H$, with the interpretation as a (topological) defect Lagrangian, that satisfies
  \be
  \label{dtheta...}
  {\rm d} \theta = \imath^*(p_1^*\omega {\cdot} (p_2^*\omega)^{-1})
  = \imath_1^*\omega \,{\cdot}\, (\imath_2^*\omega)^{-1} 
  \qquad{\rm with}\qquad
  \imath_1 := p_1 \circ \imath \,, \quad \imath_2 := p_2 \circ \imath \,,
  \ee
where $p_1$ and $p_2$ are the projection from $G \,{\times}\, G$ to the first and
second copy of $G$, respectively. Note that $\theta$ is only relevant up to coboundaries.

Further, due to the normalization of $\omega$, $\theta$ can be taken to be normalized as well,
  \be
  \theta(e,h_2^{}) = 1 \,.
  \ee
To this collection of data the construction of \cite{Mo14,FSV14} associates 
in a first step a groupoid, namely the action groupoid
  \be\label{G|GG|H}
  G\sll G \,{\times}\, G \srr_{\!\imath^-} H \,;
  \ee
where $G$ acts on $G \,{\times}\, G$ from the left as the diagonal subgroup while, as
indicated by the symbol $\imath^-$, $H$ acts by right multiplication after 
mapping it to $G \,{\times}\, G$ by $\imath$ and taking the inverse (whereby one thus
deals with a left action as well). 
In the sequel we restrict our attention to group homomorphisms $\imath$ that
are subgroup embeddings (and thus often tacitly suppress the symbol $\imath$
altogether). These lead to indecomposable module categories while, as discussed
in \cite[App.\,A]{FSV14}, in the generic case one deals with decomposable module categories.

In a second step, the \complex-linear abelian category associated to the pointed 
circle with the data as given above is obtained from the groupoid \eqref{G|GG|H} 
by a twisted linearization process: it is a certain category
  \be\label{lght}
  \lght :=
  \big[\, G\sll G \,{\times}\, G \srr_{\!\imath^-} H,\Vect \,{\big]}^{\tau_{\omega,\theta}^{}} 
  \ee
of  functors from the groupoid \eqref{G|GG|H} to vector spaces, which is defined as follows.
The twisting groupoid cocycle $\tau(\omega,\theta)$,
representing a class in $H^2(G\sll G \,{\times}\, G \srr_{\!\imath^-} H,\complexx)$,
is obtained by an appropriate transgression prescription, which we will present in the 
next paragraph. It depends both on the three-cocycle $\omega\,{\in}\, Z^3(G,\complexx)$ 
and on the 2-cochain $\theta \,{\in}\, C^2(H,\complexx)$.

To describe the category $\lght$ given in \eqref{lght} explicitly,
we first recall the general prescription for twisted linearizations:
Given a finite group $K$ and a subgroup $L\,{\le}\,K$, a 3-cocycle $\varpi$ 
on $K$ and a 2-cochain $\vartheta$ on $L$ satisfying 
$\mathrm d\vartheta \eq \varpi|_{L\times L\times L}^{}$, 
the twisted linearization $[L\sll K,\Vect_\complex]^\tau_{}$ is the following
category \cite[Def.\,3.3]{FSV14}: objects are finite-dimensional $K$-graded vector spaces
which carry a projective linear action $\rho$ of $L$,
modifying the $K$-gra\-ding by left multiplication, while morphisms are $K$-homogeneous
maps commuting with the $L$-action. The groupoid cocycle $\tau \eq \tau_{\varpi,\vartheta}$ 
arises in the composition law of the action $\rho$, according to
  \be
  \label{FSV14-3.16}
  \rho_{l_1 l_2}^{}|_{V_k}^{} = \tau(k;l_1,l_2)\, \rho_{l_1}^{}|_{V_{l_2 k}}^{}
  \,{\circ}\, \rho_{l_2}^{}|_{V_k}^{}
  \ee
for $k\,{\in}\, K$ and $l_1,l_2 \,{\in}\, L$.

In the case of the category $\lght$ of our interest, the objects are $G{\times}G$-graded vector spaces
  \be
  V = \bigoplus_{g_1,g_2\in G} V_{(g_1^{},g_2^{})}
  \label{VGG}
  \ee
endowed with a left $G \,{\times}\, H$-action $\pi$ such that
  \be
  \pi_{g,h}:\quad V_{(g_1^{},g_2^{})} \to V_{(gg_1^{},gg_2^{})\,\imath(h)^{-1}}
  \label{pigh}
  \ee
for $g\,{\in}\,G$ and $h\,{\in}\,H$. The cocycle $\tau$ can be computed 
by an algorithm which utilizes three-di\-mensional diagrams and their decomposition 
into simplices. To this end, the objects and morphisms of the groupoid $L\sll K$ are 
represented, respectively, by one- and two-dimensional graphical elements
with edges labeled by group elements, subject to the holonomy condition that the
product of group elements along a closed curve equals the neutral element.  
This algorithm is formulated as a set of rules which associate an algebraic expression
to a piecewise-linear three-manifold, by chopping the manifold into tetrahedra and
triangles and multiplying the expressions associated to those; these rules are spelled 
out in \cite[Sect.\,1.2.1]{Wi08} and \cite[Sect.\,3.4]{FSV14}.

Concretely, an oriented 3-simplex with edges labeled by elements of the group $K$
with 3-cocycle $\varpi$ stands for a number obtained by evaluating $\varpi$
according to \cite[(3.34),(3.37)]{FSV14}
\eqpic{Figure5}{270}{49} {
  \put(0,0)     {\Includepic{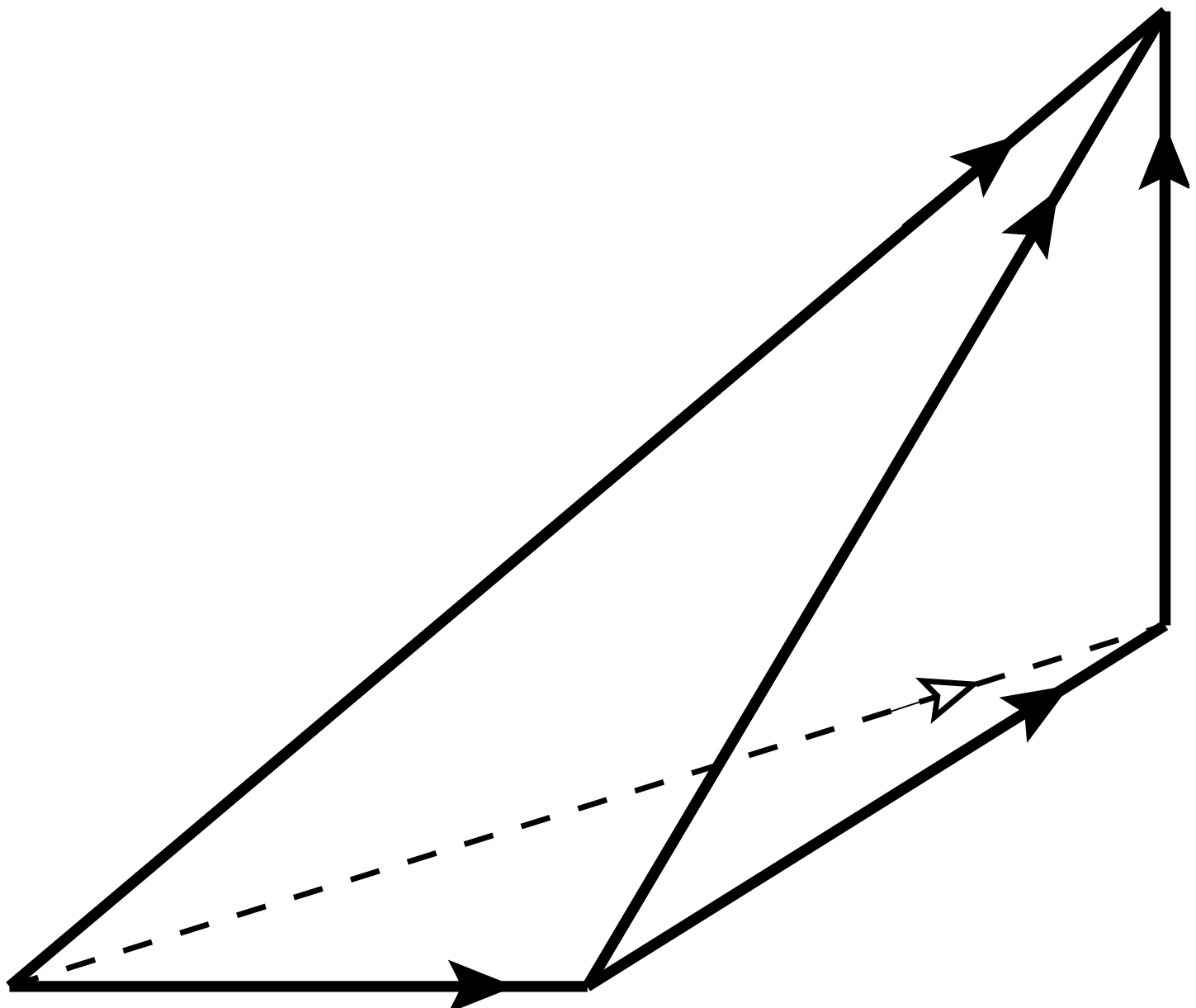}
    \put(42,-3.6) {\scriptsize$ k_1^{} $}
    \put(94,37.6) {\begin{turn}{20}\scriptsize$ k_2^{} k_1^{} $\end{turn}}
    \put(82,78.2) {\begin{turn}{42}\scriptsize$ k_3^{} k_2^{} k_1^{} $\end{turn}}
    \put(99,71.5){\begin{turn}{62}\scriptsize$ k_3^{} k_2^{} $\end{turn}}
    \put(108,23.3){\scriptsize$ k_2^{} $}
    \put(133.5,88){\scriptsize$ k_3^{} $}
    \put(158,49)  {$ \longmapsto ~~~ \varpi(k_3,k_2,k_1) \in \complexx , $}
  } }
while a triangle with edges labeled by elements of $L$ with 2-cochain $\vartheta$
evaluates as \cite[(3.36),(3.38)]{FSV14} 
\eqpic{Figure6}{250}{14} {
  \put(0,2)     {\Includepic{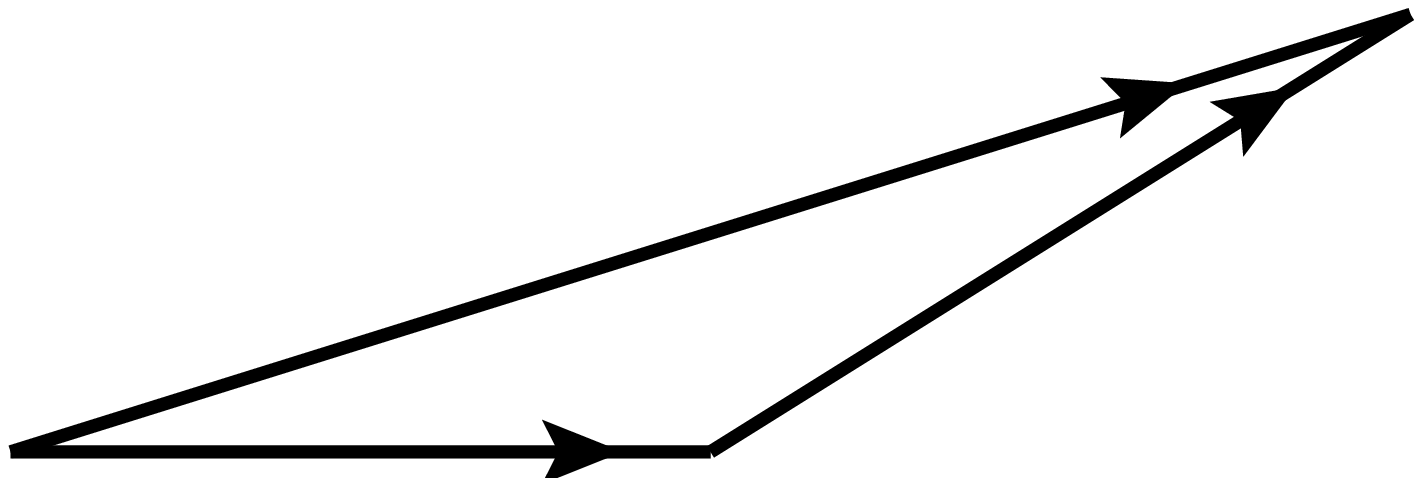}
    \put(43,-5.5) {\scriptsize$ l_1^{} $}
    \put(82.4,37) {\scriptsize$ l_2^{} l_1^{} $}
    \put(110,22.3){\scriptsize$ l_2^{} $}
    \put(145,15)  {$ \longmapsto ~~~ [\,\vartheta(l_2^{},l_1^{})\,{]}^{-1}_{} \in \complexx . $}
  } }

Note the convention for the order of multiplication of the group elements appearing
in these diagrams. Also recall that we choose $\omega$ such that it satisfies
the identities \eqref{s4symm}, i.e.\ manifestly has the symmetries of the tetrahedron. 
Moreover, we compatibly choose the 2-cochain $\vartheta$ in such a way that it
manifestly has the symmetries of the triangle, i.e.\
  \be
  \vartheta(b,a)= \vartheta(a,a^{-1}b^{-1}) = [\vartheta(a^{-1},b^{-1}) {]}^{-1} .
  \label{theta-syms}
  \ee

In \cite{FSV14} an algorithm is outlined which allows one to obtain the 2-cocycle 
$\tau_{\omega,\theta}$ in \eqref{lght} as the evaluation of a simplicial 3-manifold labeled by elements
of $K$ with certain 2-simplices labeled by elements in $L$ as in the figure \eqref{Figure6}. 
Thereby one evaluates such special 2-simplices as the product of the evaluation of triangles 
as in \eqref{Figure6}, and their 3-dimensional complement as the product of the evaluation of 
3-simplices as in \eqref{Figure5}, using the group homomorphism of $L$ into $K$.  
Applied to the case at hand, we obtain a graphical representation
of the 2-cocycle $\tau_{\omega,\theta}$ in \eqref{lght} as the evaluation of the following
piecewise-linear three-manifold: 
\eqpic{Figures2}{330}{57} {
  \put(0,0)     {\Includepic{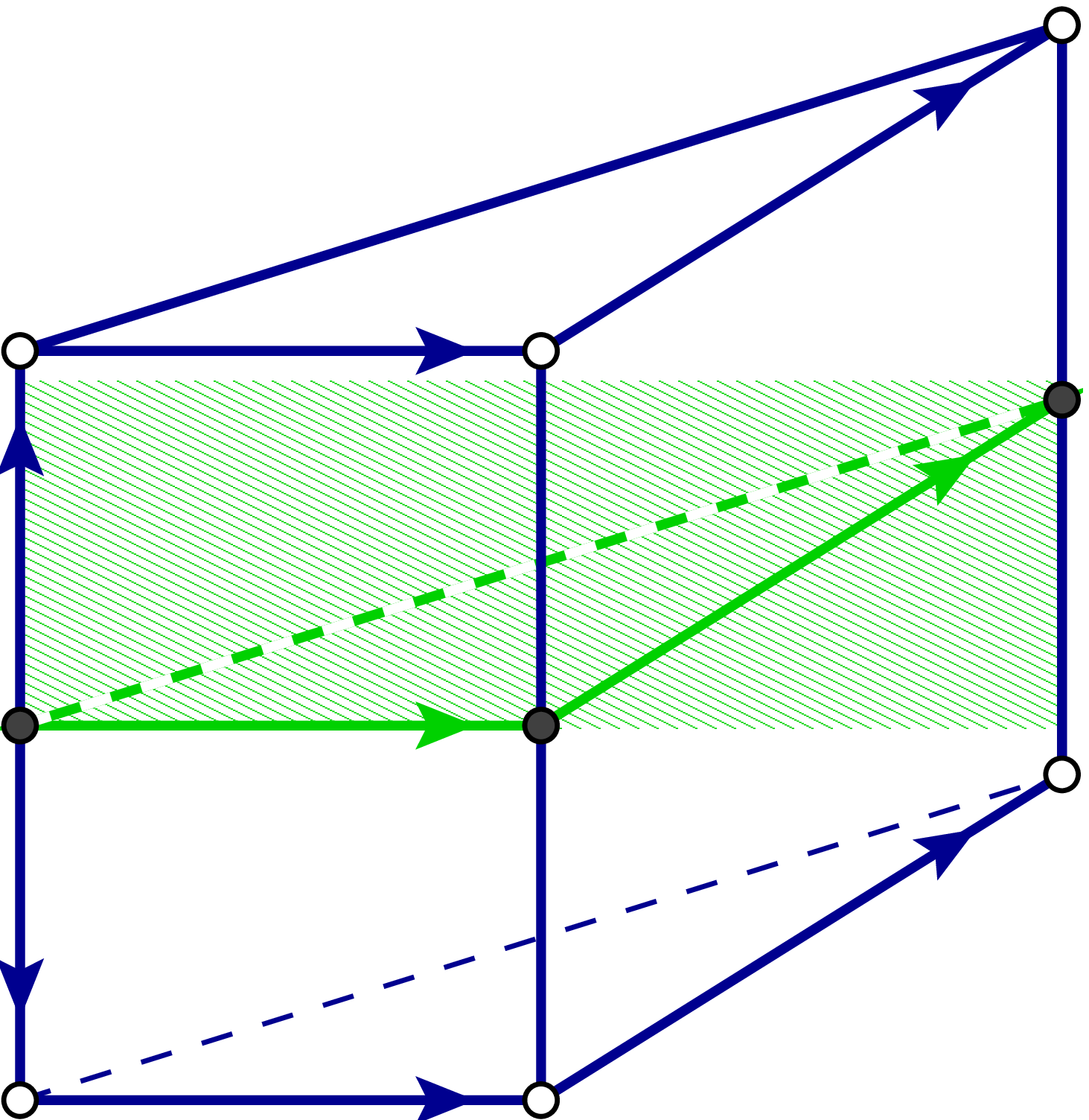}
    \put(-5.1,75) {\scriptsize $ \gamma $}
    \put(-4.5,21) {\scriptsize $ \delta $}
    \put(44,-4.3) {\scriptsize $ g_1^{} $}
    \put(44,88.7) {\scriptsize $ g_1^{} $}
    \put(44,41.3) {\scriptsize $ h_1^{} $}
    \put(110,22.5) {\begin{turn}{33.5}\scriptsize $ g_2^{} $\end{turn}}
    \put(110,115.5){\begin{turn}{33.5}\scriptsize $ g_2^{} $\end{turn}}
    \put(110,68)  {\begin{turn}{33.5}\scriptsize $ h_2^{} $\end{turn}}
  }
  \put(159,59) {$ \longmapsto ~~~ \tau_{\omega,\theta}^{}((\gamma,\delta);(g_1^{},h_1^{}),(g_2^{},h_2^{})) $}
}
Here $(\gamma,\delta)\,{\in}\, G \,{\times}\, G$ is an object and
$(g_1,h_1),(g_2,h_2) \,{\in}\, G \,{\times}\, H$ are morphisms of the action groupoid \eqref{G|GG|H},
and the identification of the top and bottom faces of \eqref{Figures2} reflects the fact that
we deal with the category for a circle, while the two different types of circles on the vertices 
indicate that the corresponding edges which connect vertices of the same type are labeled by 
two different types of gauge transformations, corresponding to the two factors in the morphisms 
of the action groupoid. Further, when evaluating the upper half of the three-manifold 
$h_i$ (with $i \eq 1,2$) stand for $\imath_1(h_i) \,{=}\, p_1(\imath(h_i))$, while when evaluating
the lower half they stand for $\imath_2(h_i) \,{=}\, p_2(\imath(h_i))$.
According to \eqref{Figure6} the triangle in the middle plane of \eqref{Figures2}
evaluates to $[\theta(h_2^{},h_1^{}){]}^{-1}$. To evaluate the upper part
we decompose it into three simplices according to 
\eqpic{Figure9}{400}{45} {
  \put(0,0)   {\Includepic{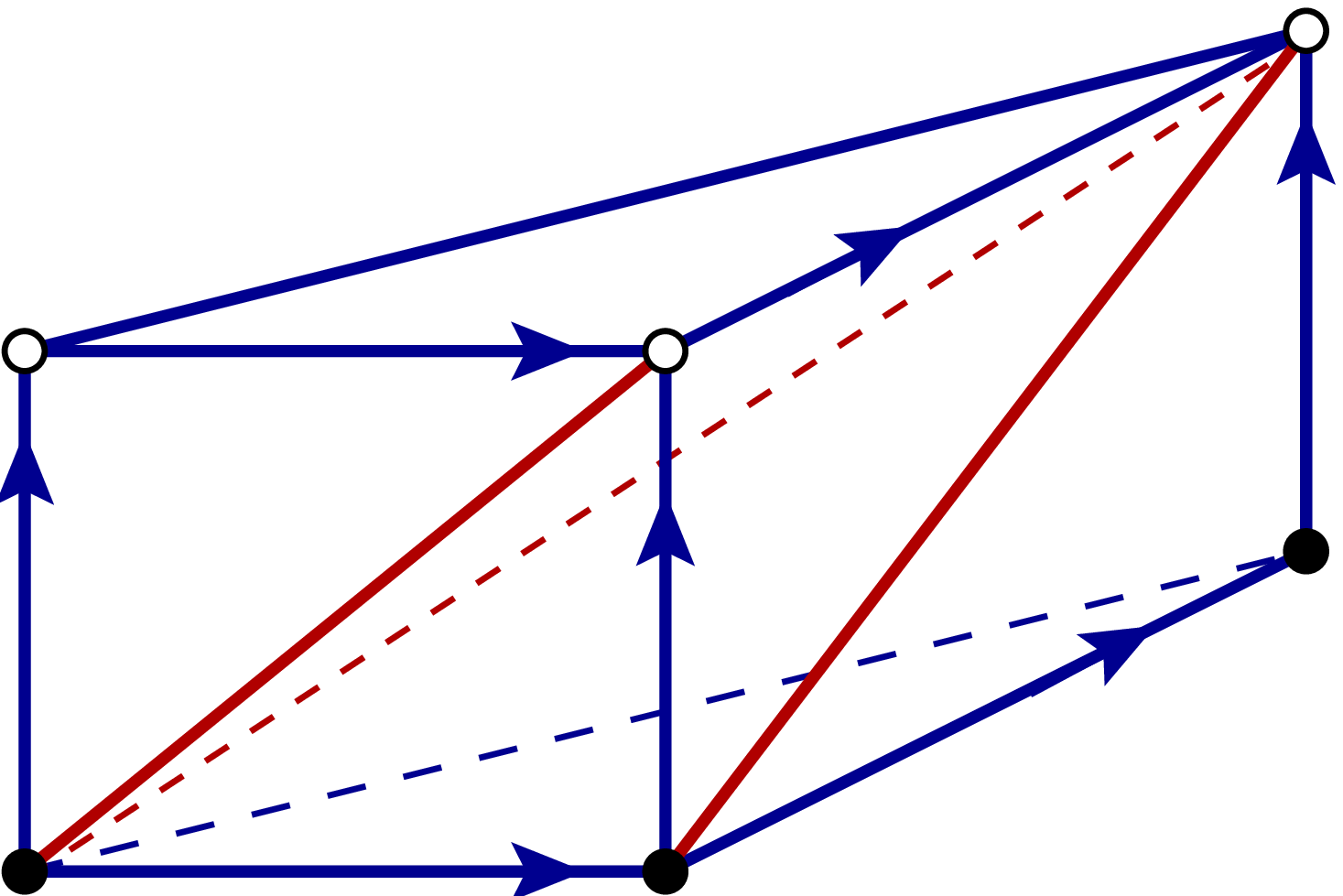}
  }
  \put(160,41) {=}
  \put(195,30)  {\Includepic{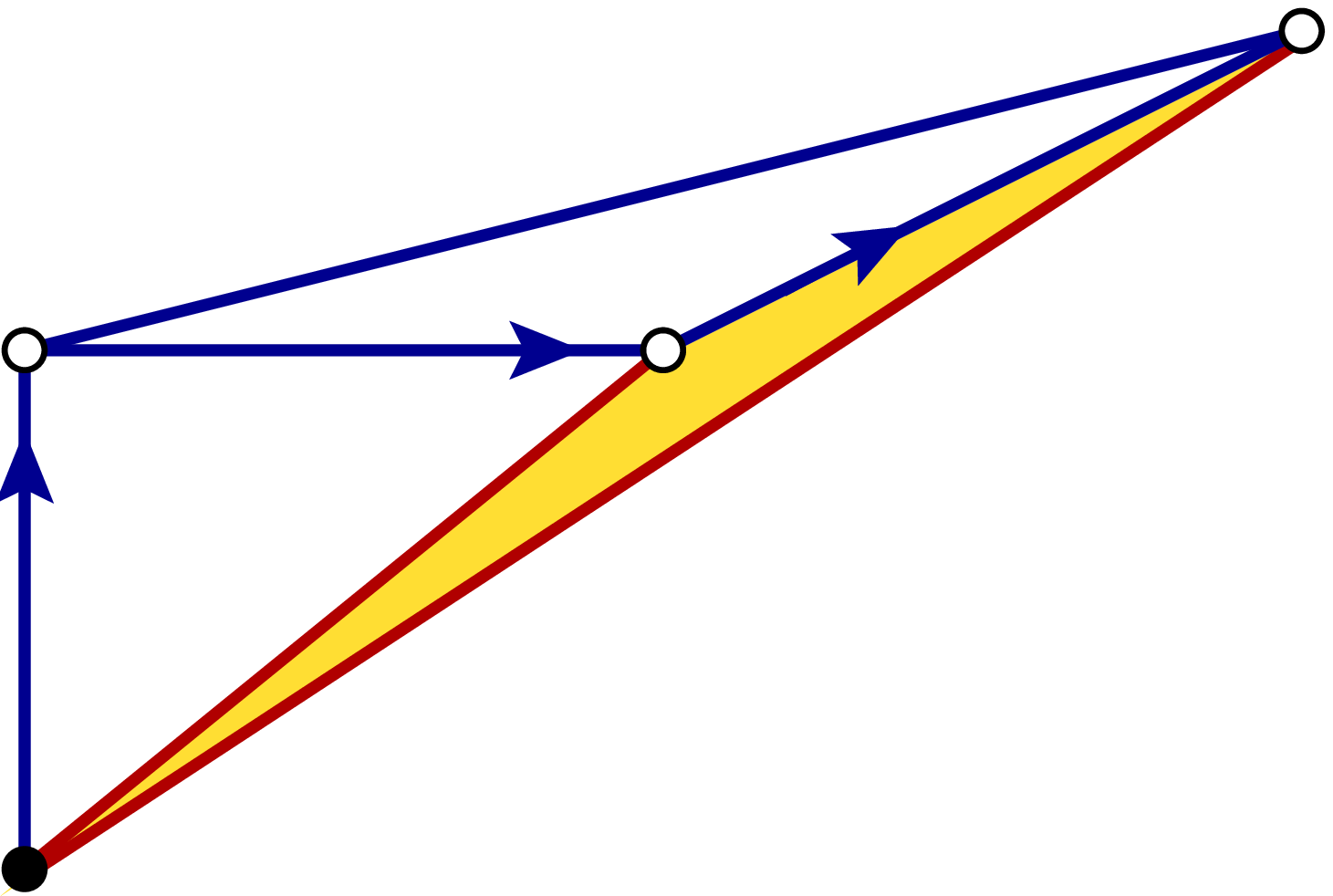}
  }
  \put(215,20)  {\Includepic{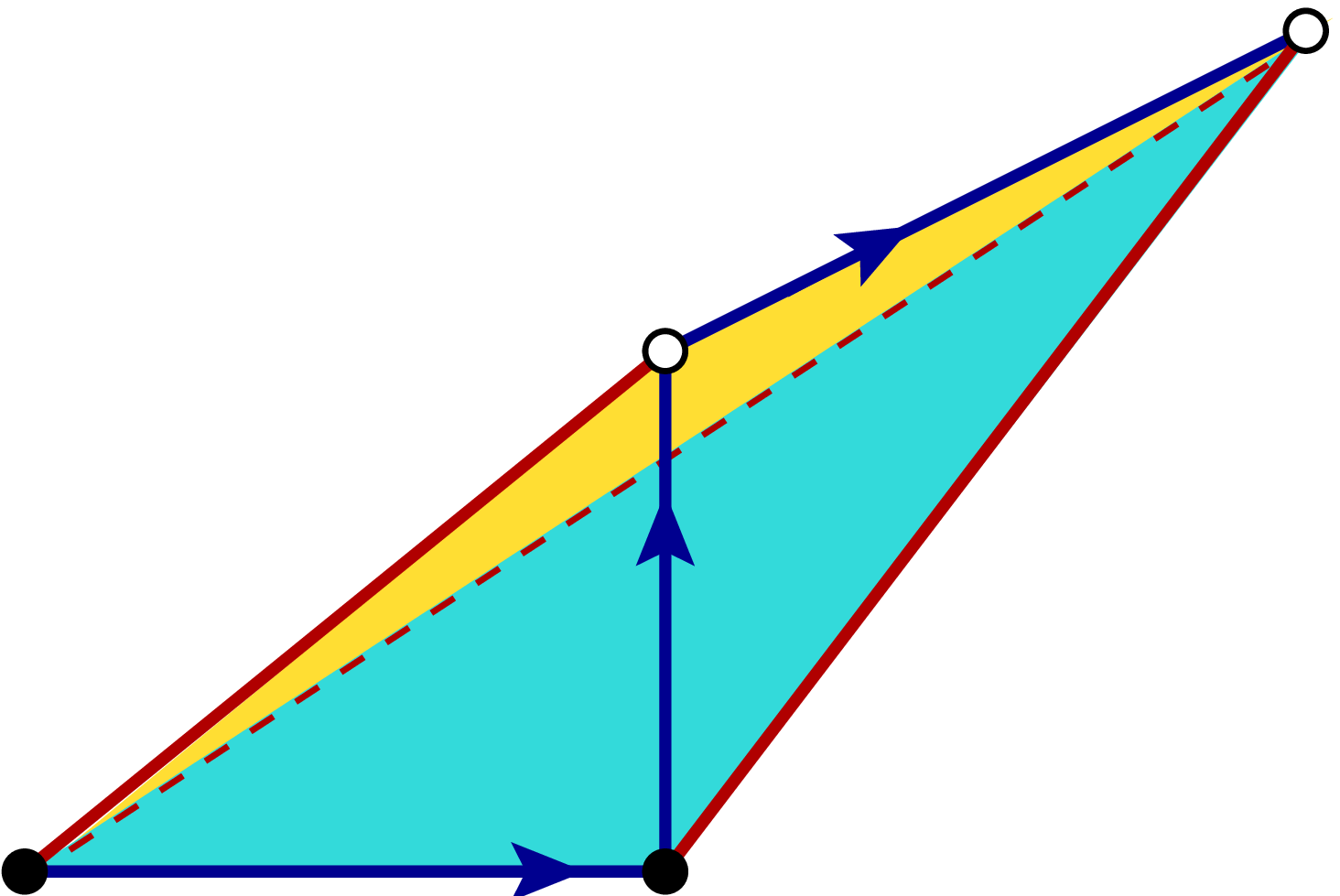}
  }
  \put(255,-4)  {\Includepic{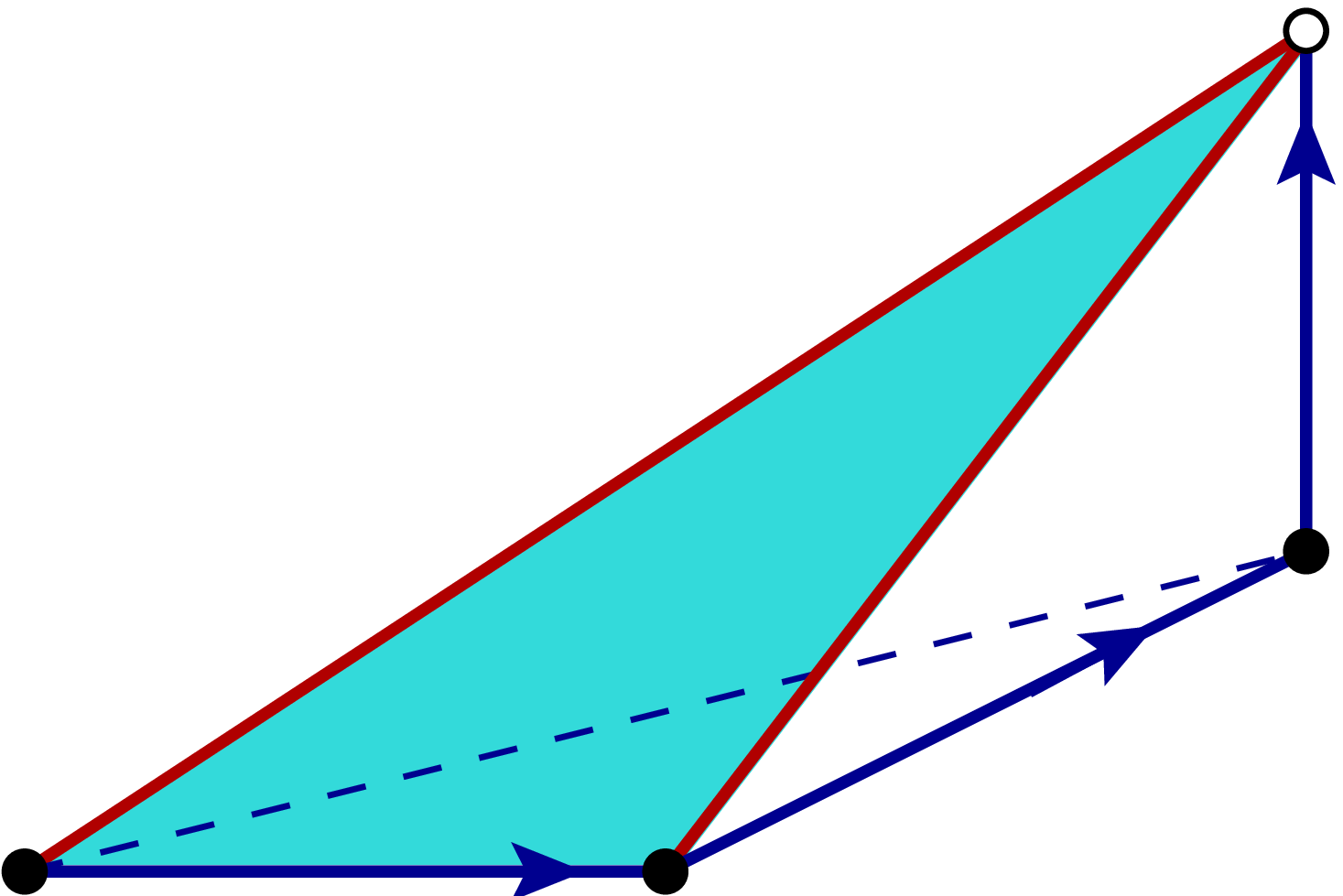}
  } }
so that by \eqref{Figure5} we get the number
  \be
  \tau_{\rm top}(\gamma;g_1^{},h_1^{},g_2^{},h_2^{})
 %% = \tilde\omega(\gamma,g_1,g_2) \, \tilde\omega(h_1^{},g_1^{}\gamma h_1^{-1},g_2)^{-1} \,
 %% \tilde\omega(h_1^{},h_2^{},g_2^{}g_1^{}\gamma h_1^{-1}h_2^{-1}) \,,
  = \omega(g_2^{},g_1^{},\gamma) \, \omega(g_2,g_1^{}\gamma h_1^{-1},h_1^{})^{-1} \,
  \omega(g_2^{}g_1^{}\gamma h_1^{-1}h_2^{-1},h_2^{},h_1^{}) \,,
  \ee
where as pointed out above, on the right hand side $h_i$ stands for $\imath_1(h_i)$.
In the same way we get $ \tau_{\rm bottom}(\delta;g_1^{},h_1^{},g_2^{},h_2^{}) 
\eq \tau_{\rm top}(\delta;g_1^{},h_1^{},g_2^{},h_2^{}) ^{-1}_{} $,
where on the right hand side $h_i$ now stands for $\imath_2(h_i)$.

Combining these results we end up with
\be
\hspace*{-.9em}\begin{array}{l}
  \tau_{\omega,\theta} ((\gamma_1^{},\gamma_2^{});(g,h),(g',h'))
  = [\theta(h',h){]}^{-1} \, \tau_{\rm top}(\gamma_1^{};g,h,g',h')\, \tau_{\rm bottom}(\gamma_2^{};g,h,g',h')
  \\{}\\[-8pt]\hspace*{1.8em}
  = [\theta(h',h){]}^{-1} \,
  \\{}\\[-9pt]\hspace*{2.8em}
  \omega(g',g,\gamma_1^{}) \, [\omega(g',g\gamma_1^{}\ie(h)^{-1}_{},\ie(h)) {]}^{-1} \, 
  \omega(g'g\gamma_1^{}\ie(h)^{-1}_{}\ie(h')^{-1}_{},\ie(h'),\ie(h))
  \\{}\\[-8pt]\hspace*{2.7em}
  [\omega(g',g,\gamma_2^{}) {]}^{-1} \, \omega(g',g\gamma_2^{}\iz(h)^{-1}_{},\iz(h)) \, 
  [\omega(g'g\gamma_2^{}\iz(h)^{-1}_{}\iz(h')^{-1}_{},\iz(h'),\iz(h)) {]}^{-1}\,.
  \end{array}
\label{tau_omegatheta}
\ee

Let us summarize: There is an equivalence 
  \be
  \label{lght=GGVectGHtau}
  \lght \,\simeq\, \GGVectGHtau
  \ee
between the twisted linearization  $\lght \,{=}\,
\big[\, G\sll G \,{\times}\, G \srr_{\!\imath^-} H,\Vect \,{\big]}^{\tau_{\omega,\theta}^{}}$
and the category $\GGVectGHtau$ of $G{\times}G$-graded vector spaces of the form \eqref{VGG}
with a left $G \,{\times}\, H$-action as in \eqref{pigh}, twisted by the cocycle \eqref{tau_omegatheta}.

It is also worth noting that the result (\ref{tau_omegatheta}) implies that  
\be
\tau_{\omega,\theta} ((\gamma_1^{},\gamma_2^{});(g,e),(e,h)) = 1 \,.
\ee
This tells us that the action $\pi$ of the product of the groups $G$ and $H$ is related to 
the action of the individual factors as  $\pi_{g,h} \eq \pi_{e,h} \,{\circ}\, \pi_{g,e}$.

%%%%%%%%%%%%%%%%%%%%%%%%%%%%%%%%%%%%%%%%%%%%%%%%%%%%%%%%%%%%%%%%%%%%%%%%%%%%%%%% 

\subsection{Comparison with the algebraic result}

We now compare the results of Section \ref{sec:gauge} with the description of the \ct\ 
of the bimodule category $\Cat{M}(H, \theta)$ in Example \ref{example:DW-algebraic}
(we continue to work with vector spaces over $\complex$, which is conventional
in the context studied in Section \ref{sec:gauge}, but the result holds for any
algebraically closed field $\Bbbk$):

\begin{theorem}
  The category of $A_{\Gdiag}$-$A_{H,\theta}$-bimodules in 
  $\Vect(G)^\omega \boxtimes\Vect(G)^{\omega^{-1}}$ is equivalent to the functor category 
  $\big[\, G\sll G \,{\times}\,G \srr_{\!\imath^-} H,\Vect \,{\big]}^{\tau_{\omega,\theta}^{}} $
  obtained by the gauge-theoretic considerations. 
\end{theorem}

\begin{proof}
  Invoking the equivalence \eqref{lght=GGVectGHtau}, we need to
  construct an equivalence from the category of $A_{\Gdiag}$-$A_{H,\theta}$-bimodules in 
  $\Vect(G)^\omega \boxtimes\Vect(G)^{\omega^{-1}}$
  to the category $\GGVectGHtau$ of $G{\times}G$-graded vector spaces with
  $G \,{\times}\, H$-action $\pi$ twisted by $\tau_{\omega,\theta}$ described above.
  We first define an equivalence functor on objects. As it turns out, this can be done in
  such a way that for each object the underlying $G{\times}G$-graded vector space 
  $V \,{=}\, \bigoplus_{\gamma_1,\gamma_2\in G}V_{(\gamma_1,\gamma_2)}$ gets mapped to itself:
  the left action $ A_{\Gdiag}\otimes V \,{\to}\, V $ amounts to a linear map
    \be
  \big( \bigoplus_{g\in G}\, \complex_g \,\big) \otimes
  \big( \!\! \bigoplus_{\gamma_1,\gamma_2\in G} V_{(\gamma_1,\gamma_2)} \,\big)
  \longrightarrow \bigoplus_{\gamma'_1\gamma'_2}V_{(\gamma'_1,\gamma'_2)}
    \ee
  and thus to a family of linear maps
    \be
  \complex_g\otimes V_{(\gamma_1,\gamma_2)} \longrightarrow V_{(\gamma'_1,\gamma'_2)} 
    \ee
  for all $g \,{\in}\, G$ which are non-zero only if $(\gamma_1',\gamma_2') \,{=}\, (g\gamma_1,g\gamma_2)$.
  Using the unit constraint for the monoidal unit $\complex$ in the
  category of complex vector spaces, we then obtain a linear map
    \be
  \rho(\gamma_1,\gamma_2;g):\quad V_{(\gamma_1,\gamma_2)}\to V_{(g\gamma_1,g\gamma_2)} \,.
    \ee
  In an analogous manner we get for any $h \,{\in}\, H$ a linear map
    \be
  \ohr(\gamma_1,\gamma_2;h):\quad V_{(\gamma_1,\gamma_2)}\to V_{(\gamma_1,\gamma_2) \iota(h^{-1})} \,. 
    \ee
  The main point is now to show that this defines a projective action of $G \,{\times}\, H$
  for the same 2-co\-cycle on the groupoid $G\sll G \,{\times}\,G \srr_{\!\imath^-} H$ as 
  the one found in Section \ref{sec:gauge}, by setting 
    \be 
  \label{definingActionGH}
  \pi(\gamma_1^{},\gamma_2^{};(g,h)) 
  := \ohr(g \gamma_1^{},g \gamma_2^{};h) \circ \rho(\gamma_1^{},\gamma_2^{};g) \,.
  \ee
  To this end we note that compatibility of $\rho$ with the product of $A_{\Gdiag}$ 
  amounts to the relation
  \be 
  \rho(\gamma_1^{},\gamma_2^{};g'g) 
  = \omega(g',g,\gamma_1^{}) \, [\omega(g',g,\gamma_2^{}) {]}^{-1} \,
  \rho(g\gamma_1,g\gamma_2;g') \circ \rho(\gamma_1,\gamma_2;g)
  \ee
  for all quadruples $g,g',\gamma_1,\gamma_2 \,{\in}\, G$ (compare \cite[(3.55)]{FSV14}).
  Similarly, compatibility of $\ohr$ with the action of $A_{H,\theta}$ gives
  % (compare \cite[(3.56)]{FSV14})
  \be 
  \hspace*{-.4em}\begin{array}{r}
    \ohr(\gamma_1^{},\gamma_2^{};hh') = [\theta(h^{-1},h'{}^{-1}) {]}^{-1} \,
    [\omega(\gamma_1^{},\ie(h)^{-1},\ie(h')^{-1}) {]}^{-1} \, \omega(\gamma_2^{},\iz(h)^{-1},\iz(h')^{-1}) \,
    \\{}\\[-8pt]
    \ohr(\gamma_1^{},\gamma_2^{})\imath(h);h') \circ \, \ohr(\gamma_1,\gamma_2;h) 
  \end{array}
  \ee
  for $\gamma_1,\gamma_2 \,{\in}\, G$ and $h,h' \,{\in}\, H$, while
  the requirement that the two actions commute amounts to 
  \be
  \begin{array}{r}
    \ohr(g\gamma_1^{},g\gamma_2^{};h) \circ \rho(\gamma_1^{},\gamma_2^{};g)
    = \omega(g,\gamma_1^{},\ie(h)^{-1}) \, [\omega(g,\gamma_2^{},\iz(h)^{-1}) {]}^{-1} \hspace*{3.7em}
    \\{}\\[-8pt]
    \rho((\gamma_1^{},\gamma_2^{})\imath(h);g) \circ \ohr(\gamma_1^{},\gamma_2^{};h)
  \end{array}
  \ee
  (compare \cite[(3.57)]{FSV14}).
  By comparison with the formula \eqref{tau_omegatheta} for the groupoid cocycle $\tau_{\omega,\theta}$
  (and also invoking the normalization of $\omega$ and $\theta$), these relations can be rewritten as 
  \be
  \begin{array}{l}
    \rho(\gamma_1^{},\gamma_2^{};g'g) = \tau_{\omega,\theta} ((\gamma_1^{},\gamma_2^{});(g,e),(g',e)) \,
    \rho(g\gamma_1,g\gamma_2;g') \,{\circ}\, \rho(\gamma_1,\gamma_2;g) \,,
    \\{}\\[-7pt]
    \ohr(\gamma_1^{},\gamma_2^{};hh') = \tau_{\omega,\theta} ((\gamma_1^{},\gamma_2^{});(e,h),(e,h')) \,\,
    \ohr(\gamma_1^{},\gamma_2^{})\imath(h);h') \,{\circ}\, \ohr(\gamma_1,\gamma_2;h) \quad \text{and}
    \\{}\\[-7pt]
    \ohr(g\gamma_1^{},g\gamma_2^{};h) \,{\circ}\, \rho(\gamma_1^{},\gamma_2^{};g)
    \\{}\\[-9pt]
    \hspace*{5.6em} = \tau_{\omega,\theta} ((\gamma_1^{},\gamma_2^{});(e,h),(g,e)) \,
    \rho((\gamma_1^{},\gamma_2^{})\imath(h);g) \,{\circ}\, \ohr(\gamma_1^{},\gamma_2^{};h) \,.
  \end{array}
  \ee
In other words, the constraints for the actions of the algebras \eqref{AHtheta} and \eqref{Gvtheta}
are indeed precisely implemented by $\tau_{\omega,\theta}$, implying that the prescription
\eqref{definingActionGH} defines a projective action of $G \,{\times}\, H$ with cocycle
$\tau_{\omega,\theta}$.
  \\[2pt]
Finally, since $A_{\Gdiag}$-$A_{H,\theta}$-bimodule morphisms
are exactly the ones which commute with the projective $G \,{\times}\, H$-action,
we define the functor to be the identity on morphisms, so that the functor is obviously 
fully faithful. It is also clear that the functor is essentially surjective.
\end{proof}

%%%%%%%%%%%%%%%%%%%%%%%%%%%%%%%%%%%%%%%%%%%%%%%%%%%%%%%%%%%%%%%%%%%%%%%%%%%%%%%% 
 
\section{Tricategories with 3-trace}\label{sec4}

\subsection{Coherence for 3-traces}

The purpose of this section is to formalize the structure of the \ct\ in the framework of 
the tricategory $\BimCat$ of bimodule categories. This leads to the concept of a 3-trace $\tr$
on the endomorphism bicategories of a tricategory $\Cat{T}$ with values in a fixed bicategory $\Cat{B}$.

Such a 3-trace associates in particular to every 1-endomorphism $f$ in $\Cat{T}$ an 
object $\tr(f)$ in $\Cat{B}$. In the case of our primary interest, $\Cat{T} \,{=}\, \BimCat$ and 
$\Cat{B} \,{=}\, \Categ$, the bicategory of small categories, and $\tr(\CMC) \,{=}\, \circtensor\Cat{M}$.  
We emphasize that the notion of a 3-trace is applied to Hom-bicategories
and thus does not require the existence of a symmetric monoidal structure on the tricategory $\Cat{T}$. 

A benefit of this more abstract formulation is that it provides a generalization of 
the Drinfeld center and that it allows us to prove a coherence statement for 3-traces:
The trace of a composition of several 1-morphisms comes equipped with cyclic equivalences 
in $\Cat{B}$, which are unique up to unique natural isomorphisms. 
This provides a conceptual background for the value a TFT assigns to a marked circle: 
the trace then depends only on the cyclic order of the 1-morphisms.

Regarding tricategories, we mostly work with their Hom-bicategories,
so that we content us to briefly remark on the coherence data that are used in the sequel,
following the conventions of \cite{Gurski}:
We use the symbol $\Box$ for the composition that is defined on \mbox{1-,} 2- and 3-morphisms
in a tricategory $\Cat T$, and we write $\circ$ for the composition of 2- and 3-morphisms and $\cdot$ 
for the composition of 3-morphisms. The associator in $\Cat{T}$ includes in particular 2-morphisms 
$a_{f,g,h}\colon (f \Box g) \Box h \,{\Rightarrow}\, f \Box (g \Box h)$ for all composable 
triples $f,g,h$ of 1-morphisms. The units include 2-morphisms 
$r_{f}\colon f \Box 1 \,{\Rightarrow}\, f$ and $l_{f}\colon 1 \Box f \,{\Rightarrow}\, f$. Furthermore 
there are invertible 3-morphisms $\mu, \lambda , \rho$ expressing the compatibility of the associator 
with the units, as well as an invertible 3-morphism $\pi$ replacing the identity in the  
pentagon axiom for the associator $a$. These structures possess (weak) inverses, which are
denoted by $a^{-}$, $r^{-}$, $l^{-}$, respectively. There are three axioms relating these structures.

In the sequel we call an $n$-tuple $f_{1},f_2,...\,, f_{n}$ of $1$-morphisms in $\Cat{T}$
\emph{cyclically composable} if the morphisms are composable and, for $f_{1}$ a morphism 
from $x_{1}$ to $x_{2}$, $f_{n}$ is a morphism from $x_{n-1}$ to $x_{1}$.

\begin{definition}
  \label{definition:3-trace}
  Let $\Cat{T}$ be a tricategory and $\Cat{B}$ a bicategory. A \emph{3-trace $\tr$
    on $\Cat{T}$ with values in $\Cat{B}$} is a collection of $2$-functors 
  \begin{equation}
    \tr_{x}:\quad \Cat{T}(x,x) \rightarrow \Cat{B}
  \end{equation}
  for all objects $x$ of $\,\Cat{T}$, together with structural morphisms $\varphi$, $m$ 
  and $\kappa$ as follows:
  \begin{definitionlist}[leftmargin=1.9em, label=\emph{(\roman*)}]
  \item 
    For any pair of objects $x,y \,{\in}\, \Cat{T}$ there is an adjoint equivalence $\varphi(x,y)$ between 
    the induced $2$-functors 
    \begin{equation}
      \label{eq:B-bimodule}
      \begin{tikzcd}[column sep=large]
        & \Cat{T}(x,x) \ar[shorten >=15pt, shorten <=15pt, Rightarrow]{dd}{\varphi(x,y)} \ar[bend left]{dr}{\tr_{x}} 
        & \\
        \Cat{T}(x,y) \times \Cat{T}(y,x) \ar[bend left=25]{ur} \ar[bend right=25]{dr}  & & \Cat{B} ,\\ 
        &  \Cat{T}(y,y) \ar[bend right]{ur}{\tr_{y}}& 
      \end{tikzcd}
    \end{equation}
    where the unnamed arrows are given by the composition in $\Cat{T}$. This means in particular that for all  
    $f \,{\times} g \in \Cat{T}(x,y) \,{\times}\, \Cat{T}(y,x)$ we are given adjoint equivalences 
    \begin{equation}
      \varphi(f,g):\quad \tr_{y}(f \Box g) \xrightarrow{~\simeq~} \tr_{x}(g \Box f)
    \end{equation}
    in $\Cat{B}$.
  \item 
    For any triple of cyclically composable morphisms $f_{1}\colon x\,{\to}\, y$, 
    $f_{2}\colon y \,{\to}\, z$ and $f_{3}\colon z\,{\to}\, x$ there is an invertible modification 
    $m(f_{3},f_{2},f_{1})$ between the following composites of pseudo-natural transformations:
    \begin{equation}
      \label{eq:25-eval}
      \begin{tikzcd}%[column sep=tiny]
        \tr_{x}((f_{3} \Box f_{2}) \Box f_{1}) \ar{dr}{\varphi(f_{3} \Box f_{2},f_{1})} \ar{dd}{\simeq}  
        && \\
        &\tr_{y}(f_{1} \Box (f_{3} \Box f_{2})) \ar{dr}{\simeq}
        & \\
        \tr_{x}(f_{3} \Box (f_{2} \Box f_{1})) \ar{dd}[name=A]{\varphi(f_{3},f_{2} \Box f_{1})}
        &&  \tr_{y}( (f_{1} \Box f_{3})\Box f_{2} ) \ar{dl}[yshift=--2pt]{\varphi(f_{1} \Box f_{3},f_{2})} 
        \ar[shorten <= 40pt, shorten >=40pt,Rightarrow]{ll}[above]{m(f_{3},f_{2},f_{1})}
        \\
        &  \tr_{z}(f_{2} \Box (f_{1}\Box f_{3})) \ar{dl}{\simeq} 
        &\\
        \tr_{z}((f_{2} \Box f_{1} )\Box f_{3}) && 
      \end{tikzcd}
    \end{equation}
    Here the unlabeled arrows are obtained from applying the $2$-functors $\tr$ to 
    the respective associators in $\Cat{T}$.

  \item 
    $\kappa$ is an invertible modification between the composite of pseudo-natural transformations. 
    For all $1$-morphisms $f \colon x \,{\to}\, x$ it consists of an invertible $2$-morphism 
    in $\Cat{B}$ as follows:
    \begin{equation}
      \label{eq:kappa}
      \begin{tikzcd}
        {} &\tr(1 \Box f) \ar{r}[name=A]{\varphi(1,f)} &\tr(f \Box 1) \ar{dr}[yshift=-4pt]{\tr(r_{f}^{-})} &
        \\
        \tr(f)\ar{ur}[yshift=-3pt]{\tr(l_{f})}  \ar{rrr}[name=B, below]{1} &&& \tr(f)\,.
        \arrow[shorten <= 8pt, shorten >=5pt, Rightarrow,to path=(A)-- (B) \tikztonodes]{}{\kappa_{f}}
      \end{tikzcd}
    \end{equation}
  \end{definitionlist}
  These data are required to satisfy the following axioms:
  \begin{definitionlist}
  \item 
    Consider any quadruple $f_1,f_2,f_3, f_4$ of cyclically composable $1$-morphisms in $\Cat{T}$;
    for better readability we abbreviate them just by $1,2,3,4$ and also omit the associators.
    Then the following diagram must be the identity $2$-morphism on the $1$-mor\-phisms
    that are given by the outer arrows:
    \begin{equation}
      \label{eq:axiom-3trace}
      \begin{tikzcd}[column sep=large]
        \tr(4321) \ar{r}{\varphi(432,1)}  \ar[bend left=40]{rr}[name=B]{\varphi(43,21)}
        \ar[bend left=60]{rrr}[name=A]{\varphi(4,321)}   \ar[bend right=60]{rrr}[name=D]{\varphi(4,321)}& 
        \tr(1432) \ar{r}{\varphi(143,2) }   \ar[bend right=40]{rr}[name=C,below]{\varphi(14,32)}  
        \arrow[shorten <= 5pt, shorten >=5pt, Rightarrow,to path=-- (B) \tikztonodes]{}[yshift=-2pt]{m(43,2,1)\,}
        \arrow[shorten <= 23pt, shorten >=18pt, Rightarrow,to path=-- (D) \tikztonodes]{}[left,yshift=-5pt]{m(4,32,1)}
        &
        \tr(2143) \ar{r}{\varphi(214,3)} 
        \arrow[shorten <= 5pt, shorten >=5pt, Rightarrow,to path=-- (C) \tikztonodes]{}[yshift=1pt]{\,m(14,3,2)}
        \arrow[shorten <= 25pt, shorten >=25pt, Rightarrow,to path=-- (A) \tikztonodes]{}[right,yshift=2pt]{\,m(4,3,21)}
        & 
        \tr(3214).  
      \end{tikzcd}
    \end{equation}

  \item  
    The following diagram must provide an adjoint equivalence between $\varphi(f,g)$ and $\varphi(g,f)$:
    \begin{equation}
      \label{eq:adj-equiv-varphi}
      \hspace*{-1.4em}
      \begin{tikzcd}
        \tr(fg) \ar{r}{\tr(l_{f})} \ar{drr}[below,yshift=-2pt]{\varphi(f,g)} \ar[bend left]{rrrrr}[name=A]{1}
        & \tr((1f)g) \ar{r}{\varphi(1f,g)} \ar[bend left]{rrr}[name=B,below]{\varphi(1,fg)} 
        & \tr(g(1f)) \ar{r}[name=C]{\tr(a)} \ar{d}[left]{\tr(l_{f}^{-})} & \tr((g1)f)
        \ar{r}{\varphi(g1,f)} \ar{dl}[above,xshift=-6]{\tr(r_{g}^{-})}& \tr(f(g1)) \ar{r}{\tr(r_{g}^{-})}
        &\tr(fg)\\
        && \tr(gf) \ar{urrr}[below]{\varphi(g,f)}. &&&
        \arrow[shorten <= 8pt, shorten >=5pt, Rightarrow,to path=(A)-- (B) \tikztonodes]{}[xshift=2pt]{\kappa_{fg}^{-1}}
        \arrow[shorten <= 2pt, shorten >=2pt, Rightarrow,to path=(B)-- (C) \tikztonodes]{}[xshift=2pt]{m^{-1}}
      \end{tikzcd}
    \end{equation}
    Here the first and the last of the three triangles on the bottom are filled with $2$-iso\-morphisms 
    from the structure of the $2$-functor $\varphi$, while the middle triangle is filled with a 
    $2$-isomorphism coming from the tricategory $\Cat{T}$.
  \end{definitionlist}
\end{definition}

Note that using diagram \eqref{eq:adj-equiv-varphi}, the modification $\kappa$ induces,
for every 1-endomorphism $f$, an isomorphism in $\Cat{B}$ from $\varphi(f,1)$ 
composed with the units to the identity, as in diagram \eqref{eq:kappa}. Thus the existence 
of such isomorphisms needs not to be imposed separately for a 3-trace in $\Cat{T}$.

The main point of the concept of a 3-trace is that the trace of a composition of an arbitrary 
number of 1-morphisms has automatically a coherent cyclic invariance. To make this precise 
we need some further terminology.

First we define how to implement on a string $f_1,f_2,...\,, f_n$ of $n$ cyclically 
composable morphisms a \emph{bracketing} $b$ of the string and a \emph{cut point} 
$i \,{\in}\,  \{1,2,...\,,n\}$ (which in the composition $f_n{\Box}\cdots{\Box}f_1$ is 
counted from the right): $(b,i)(f_1,f_2,...\,, f_n)$ is the composite 
of 1-morphisms in $\Cat{T}$ that starts with $f_i$ and is bracketed according to $b$. 
As an example consider the bracketing $b \,{=}\, ((-\Box -)\Box -)$ of a string of three 
elements; then $(b,2)(f_{3},f_{2},f_{1})= ((f_{1}\Box  f_{3} )\Box f_{2})$, 
while $(b,3)(f_{3},f_{2},f_{1})= ((f_{2}\Box  f_{1} )\Box f_{3})$.

By a \emph{trace of $n$ cyclically composable morphisms} $f_1,f_2,...\,,f_n$ we refer to 
the value $\tr((b,i)(f_{n},...\,,f_{1}))$ for some choice of bracketing $b$ and cut point $i$.
An \emph{admissible equivalence} between two traces of the string $f_1,f_2,...\,,f_n$ 
is an invertible 1-morphism in $\Cat{B}$ that is a composite of 1-morphisms obtained 
formally from the adjoint equivalence $\varphi$ and the structure of the tricategory $\Cat{T}$. 
As usual in coherence statements, to rule out coincidences of morphisms,  we consider only 
\emph{formally composable} morphisms here, i.e.\ morphisms that come from the free algebraic 
structure of a given kind. The statement that we want to prove is that equivalences between 
traces corresponding to different bracketings
and cut points always exist and are unique up to unique isomorphisms. 

The strategy of the proof is as follows. For fixed $n$ we first define a 2-graph
$\Cat{G}$, which is a  two-dimensional cell complex with vertices given by possible 
bracketings of strings in $n$ free variables and with  a cut point, with 1-cells given 
by formally composable equivalences between traces of $n$ 1-morphisms, and with their
isomorphisms as 2-cells. This allows us to express the statement as a statement about the
connectedness of $\Cat{G}$. Then we reduce to the case that $\Cat{T}$ is a Gray-category, i.e.\ 
a tricategory in which the only non-trivial constraint datum is an invertible 3-morphism
replacing the identity in the interchange law for 1- and 2-morphisms, see e.g.\
\cite{Gurski}. Finally we prove the statement for 3-traces on Gray-categories. 

The formally composable morphisms are described as follows:
  \begin{definition}
\label{definition:2-graph}
For a tricategory $\Cat{T}$ with $3$-trace the \emph{formal $2$-graph $\Cat{G}$ of
traces of a string of length $n$} is defined as follows:
  \begin{itemize}
  \item 
    Vertices $v$ of $\Cat{G}$ are: All words $v$ in the binary operation $\Box$ and the 
    $0$-ary operation $I$ (corresponding to the unit 1-morphisms in $\Cat{T}$), which results 
    in a word $v$ consisting of $|v|$ terms (with $|v|$| the number that the operation $\Box$ 
    appears plus one) and of \emph{length} $n$ (which is $|v|$ minus the number of times that 
    $I$ appears), together with a bracketing $b$ of the $|v|$ terms and a cut point 
    $i \,{\in}\, \{1,2,,...\,, |v|\}$. An example with $i \,{=}\, 2$ is 
    $((-\Box-)\Box -_{c})\Box I$, where the subscript $c$ indicates the position of the cut point. 

  \item 
    $1$-cells of $\Cat{G}$ are one of the following: associators $a$ and their inverses, 
    unit $2$-mor\-phisms $r$ and $l$ and their inverses, with the obvious vertices as source and 
    target as well as the equivalences $\varphi$, which keep the bracketing but shift the cut point.
    The latter $1$-cells start at a bracketing of $|v|$ terms of the form $(b_{1})(b_{2})$, where
    $b_{1}$ and $b_{2}$ are themselves bracketings of $|b_{1}|$ and $|b_{2}|$-terms, respectively. 
    The $1$-cell $\varphi$  shifts the cut point $i$ by $|b_{2}|$, according to
      \begin{equation}
        (((b_1)(b_2)),i) \xrightarrow{~\varphi((b_1),(b_2))~} (((b_2)(b_1)),i{+}|b_2|\,{\bmod}\, n) \,.
      \end{equation}

  \item 
    $2$-cells between two strings of 1-cells are either of the following: 

    The $2$-cells representing the constraint data of $\Cat{T}$, i.e.\ the pentagon cell $\pi$, 
    the cells expressing the compatibility of the units and the associators, $\rho$, $\lambda$ and 
    $\mu$, as well as the $2$-cells that relate the $1$-cells $a$, $r$, $l$ and their inverses to identities.  

    The $2$-cells defined by the modifications $m$ and $\kappa$ in the definition 
    {\rm\ref{definition:3-trace}}
    of a $3$-trace. Furthermore all $2$-cells expressing naturality of the various $1$-cells.
  \end{itemize}
 \noindent
      \end{definition}

Any vertex $v \,{=}\, (b,i)$ of $\Cat{G}$ can be \emph{evaluated}  on a string $f_1,f_2,...\,,f_n$
of $n$ cyclically composable 1-morphisms to give the object 
  \be \label{equation:evaluation-string}
  v(f_{1}, ...\,,f_{n}) = \tr\big((b,i)(f_{1}, ...\,,f_{n}\big)
  \ee
in $\Cat{B}$. (In (\ref{equation:evaluation-string}) there can be additional unit $1$-morphisms 
on the right hand side; these are suppressed in our notation). Moreover, any sequence of composable
1-cells from $v$ to $u$ in $\Cat{G}$ can be evaluated to give a 1-morphism in $\Cat{B}$ from
$v(f_{1}, ...,f_{n})$ to $u(f_{1}, ...,f_{n})$ using the structures that represent the labels
of the 1-cells in $\Cat{G}$. Finally, any sub-2-graph $G$ consisting of a sequence of 2-cells between
two sequences of 1-cells from $v$ to $u$ gives a 2-morphism $\underline{G}(f_1,f_2,...,f_n)$ in
$\Cat{B}$ between the 1-morphisms corresponding to the evaluation of the 1-cells on the boundary of $G$.

To reduce the situation to the case of Gray-categories we need

\begin{lemma}
\label{lemma:trace-functorial}
  The notion of a $3$-trace $\tr_{x}\colon~ \Cat{T}(x,x) \,{\rightarrow}\, \Cat{B}$
  is functorial in the bicategory  $\Cat{B}$ and in the tricategory $\Cat{T}$:
  \begin{itemize}
\item 
  For any $2$-functor $\Psi\colon \Cat{B} \,{\to}\, \Cat{B}'$, composition with $\Psi$ yields 
  a $3$-trace $\Psi \,{\circ}\, \tr$ on $\Cat{T}$ with values in the bicategory $\Cat{B}'$. 
\item
  Given a $3$-functor $\Phi\colon \Cat{T}' \,{\to}\, \Cat{T}$
  {\rm(}see {\rm \cite[Def.\,3.3.1]{Gurski})}, $\Cat{T}'$ is endowed with an induced 
  $3$-trace to $\Cat{B}$ as the composite
  \begin{equation}
    \tr': \quad \Cat{T}'(x',x') \xrightarrow{~\Phi~} \Cat{T}(\Phi(x'), \Phi(x'))
    \xrightarrow{~\tr_{\Phi(x')}~} \Cat{B}.
  \end{equation}
  \end{itemize}
\end{lemma}
\begin{proof}
  The first statement is clear from the definition of a 3-trace and the composition of 
2-functors as well as the composition of natural 2-transformations and modifications with a 2-functor. 
\\[2pt] 
 Concerning the second part, 
   the structural data $\varphi'$, $m'$ and $\kappa'$ for $\tr'$ are obtained from the 
  corresponding data of $\tr$ and the structure of the 3-functor $\Phi$. It can be checked that the 
  axioms of a 3-trace then follow from the axioms for $\tr$ and the axioms of the 3-functor 
  $\Phi$.
\end{proof}

The technical statement about the coherence of a 3-trace is now the following:

\begin{proposition}
  \label{proposition:2-graph-connected}
 A $3$-trace on a tricategory $\Cat{T}$ is \emph{coherent}:
The cell complex $\Cat{G}$ associated to a tricategory $\Cat{T}$ with $3$-trace is connected 
  and $1$-con\-nected, and its evaluations are $2$-connected. Concretely:
  For any two vertices (of length $n$) and any two $1$-cells there exist 
  formal $2$-graphs connecting them and, moreover, for formal $2$-graphs $G_{1}$ and $G_{2}$
  that start and end at one and the same string of $1$-cells, the $2$-morphisms obtained by 
  evaluation on an arbitrary string $f_1,f_2,...\,,f_n$ satisfy 
  $\underline{G_{1}}(f_1,f_2,...,f_n) \,{=}\, \underline{G_{2}}(f_1,f_2,...,f_n)$.
\end{proposition}

\begin{proof}
  If $\Phi$ is an equivalence of bicategories, then the statement of the proposition  
  holds for the 3-trace $\tr$ if and only if it holds for $\Phi \circ \tr$. Since every 
  bicategory is equivalent to a strict bicategory, i.e.\ a bicategory in which 
  all $2$-morphisms $\omega^{\Cat{B}}_{H,G,F}$, $\lambda_{F}^{\Cat{B}}$ and $\rho^{\Cat{B}}_{F}$ 
  are identities, we can thus assume without loss of generality that $\Cat{B}$ is a strict bicategory. 
  We now prove a Lemma that will allow us to restrict ourselves to
  the case that the tricategory $\Cat{T}$ is a  Gray-category.

  \begin{lemma}
\label{lemma:coherence-stable}
    Let $\Phi\colon \Cat{T}' \,{\to}\, \Cat{T}$ be a triequivalence 
    {\rm (}see {\rm \cite[Def.\,4.4.1]{Gurski})} with pseu\-do-inverse $\Phi^{-}$.
    Then the $3$-trace on $\Cat{T}$ is coherent if and only if the $3$-trace on $\Cat{T}'$ is coherent. 
  \end{lemma}
  \begin{proof}
      Denote by $\underline{G}^{\Cat{T}}$ and $\underline{G}^{\Cat{T}'}$ the evaluations 
  of a 2-graph $G$ using the 3-tra\-ces $\tr$ on $\Cat{T}$ and $\tr'$ on $\Cat{T}'$, respectively. 
  We argue that for two formal 2-graphs $G_{1}$ and $G_{2}$ and a string $f_1,f_2,...,f_n$
  of 1-morphisms as above, the equality
  $\underline{G_{1}}^{\Cat{T}}(f_1,...,f_n) \,{=}\, \underline{G_{2}}^{\Cat{T}}(f_1,...,f_n)$ 
  holds if and only if $\underline{G_{1}}^{\Cat{T}'}(\Phi^{-}(f_{1}),...,\Phi^{-}( f_{n}))
  \,{=}\, \underline{G_{2}}^{\Cat{T}'}(\Phi^{-}(f_{1}),...,\Phi^{-}(f_{n}))$:
  \\
  The evaluation of a 2-graph $G$ using the 3-trace on the 1-morphisms $f_{i}$ in $\Cat{T}$ 
  is a 2-morphism $\underline{G}^{\Cat{T}}$ in $\Cat{B}$ between two 1-morphisms 
  $G^{s}, G^{t}\colon v \,{\to}\, w$. Analogously, the evaluation using the 3-trace on 
  $\Phi^{-}f_{i}$ in $\Cat{T}'$ is  a 2-morphism $\underline{G}^{\Cat{T}'}$ between 1-morphisms
  $G'^{s}, G'^{t}\colon v' \,{\rightarrow}\, w'$. The vertices $v'$ and $w'$ of this evaluation are 
  for example $\tr(\Phi \Phi^{-}(f_{n}) \Box \cdots \Box \Phi \Phi^{-} (f_{1}))$ with some bracketing. 
  Using that $\Phi$ and $\Phi^{-}$ are triequivalences, we obtain invertible 2-morphisms 
  $\Phi \Phi^{-}(f_{i}) \,{\to}\, f_{i}$ in $\Cat{T}$ for all $i \,{\in}\, \{1,2,..., n\}$.  
  Since the 3-trace consists by definition of 2-func\-tors  
  $\Cat{T}(x_{1}, x_{2}) \,{\times} \cdots {\times}\, \Cat{T}(x_{n},x_{1}) \,{\to}\, \Cat{B}$, 
  these yield invertible 1-morphisms $v'\,{\to}\, v$  and  $w' \,{\to}\, w$ in $\Cat{B}$,  
  together with 2-morphisms on the front and back of the diagram
  \begin{equation}
    \label{eq:sqare}
    \begin{tikzcd}[row sep=scriptsize, column sep=scriptsize]
      & v' \arrow{dl}[left,yshift=2pt]{=}\arrow{rr}[name=A, xshift=4]{G'^{s}} \arrow{dd} 
      & & w' \arrow{dl}[yshift=2pt]{=}\arrow{dd} \\
      v' \arrow[crossing over]{rr}[name=B, xshift=13pt]{G'^{t}} \arrow{dd} & & w' \\
      & v\arrow{dl}[left,yshift=2pt]{=}\arrow{rr}[name=C, xshift=-10]{G^{s}} & & w \arrow{dl}[yshift=2pt]{=} \\
      v \arrow{rr}[name=D]{G^{t}} & & w  \arrow[crossing over, leftarrow]{uu}\\
 \arrow[shorten <= 8pt, shorten >=3pt, Rightarrow,to path=(A)-- (B) \tikztonodes]{}[yshift=6pt]{\underline{G}^{\Cat{T}'}}
 \arrow[shorten <= 8pt, shorten >=3pt, Rightarrow,to path=(C)-- (D) \tikztonodes]{}[yshift=4pt]{\underline{G}^{\Cat{T}}}
    \end{tikzcd}
  \end{equation}
  such that this diagram of 2-morphisms in $\Cat{B}$ commutes.
This concludes the proof of the lemma. 
  \end{proof}

Since every tricategory is triequivalent to a Gray-category,
  in the proof of Proposition \ref{proposition:2-graph-connected}
  we can thus assume that $\Cat{T}$ is a Gray-category. In this case 
  the set of formal 2-graphs considerably simplifies.  On the vertices of the  2-graph $\Cat{G}$ 
  we then only deal with a  number $i \,{\in}\, \{1,2,...\,, n\}$ indicating the cut point. 
  The only 1-cells are now shift operators $S_{j}$ for $j \,{\in}\, \{0,1,...\,, n\}$ from a vertex 
  $i$ to a vertex $(i{+}j\,{\bmod}\,n)$ obtained from $\varphi$ and identities.  
  The shifts $S_{0}$ and $S_{n}$ correspond to $\varphi(-,1)$ and to $\varphi(1,-)$, respectively.
  The 2-cells are given by $m_{i,j}\colon S_{i} \,{\circ}\, S_{j} \,{\Rightarrow}\, S_{i+j}$ 
  whenever $i+j\leq n$ and $\kappa\colon S_{n} \,{\Rightarrow}\, 1$ and their inverses.
  The first statement of the proposition is now proven: for any two vertices of length $n$, 
  there exists a sequence of 1-cells connecting them by using the appropriate shift 1-cell. 
  Moreover, the shift by $i{-}j$ provides a canonical 1-cell from a vertex with cut point $i$ 
  to a vertex with cut point $j\,{<}\,i$. It is clear, by the repeated use of $m$ and $\kappa$, 
  that there exists a 2-cell from  any other sequence of 1-cells from $i$ to $j$ 
  to this canonical 1-cell. This implies that $\Cat{G}$ is 1-connected. 
  \\[3pt]
  The coherence statement can be shown invoking a 2-dimensional graphical  calculus. 
  We can encode the information of a 2-graph in a 2-dimensional string diagram in 
  the $(x,y)$-plane that is read from top to bottom: The strings in a string diagram  
  have a regular projection to the $y$-axis and carry labels  $j \,{\in}\, \{0,1,...\,, n\}$ 
  representing the shifts $S_{j}$. 
  There are trivalent vertices which join two strings to one string, representing $m_{i,j}$,
  with $m_{i,j}$ adding the labels of the two strings, and its inverse representing 
  $m^{-1}_{i,j}$. Furthermore, there is a one-valent vertex representing $\kappa$ which 
  starts at a string with label $n$ and ends at the empty string. From such a string diagram 
  we recover a 2-graph starting at a vertex $v$ with cut point $i$ as follows:
  The right-most string on the top determines a vertex $v'$ and a 1-cell labeled 
  with the corresponding shift. Inductively, this fixes vertices and 1-cells for 
  all strings. The vertices of the string diagrams now determine the corresponding 2-cells. 
  \\[3pt]
  Next we introduce moves on a string diagram that do not change the evaluation of the 
  corresponding 2-graph: 
  For strings labeled with $\alpha$ and $\beta$, we write $ \alpha \,{\circ}\, \beta $ to
  indicate that the strings are horizontally neighboring, while the symbol $\cdot$ combines 
  vertically neighboring vertices; with these conventions, the allowed moves are the following,
  \begin{itemize}
  \item 
    canceling of inverse vertices that are vertically neighboring;
  \item 
    associativity moves for $m$ and $m^{-1}$, i.e.\ for $S_{i} \,{\circ}\, S_{j} \,{\circ}\, S_{k}$ 
    with $i{+}j{+}k \,{\leq}\, n$ one has
    $ m_{i+j,k} \cdot (m_{i,j} \circ 1) = m_{i, j+k} \cdot( 1 \circ m_{j,k}) $  
    and analogously for $m^{-1}$;
  \item 
    slide moves
    $(\alpha \circ \id)\cdot (\id \circ\, \beta) \,{=}\, (\id \circ\, \beta) \cdot(\alpha \circ \id)$;
  \item 
    the \emph{adjunction move} from equation (\ref{eq:adj-equiv-varphi}):
    on the string $S_{n-i} \,{\circ}\, S_{i} \,{\circ}\, S_{n-i}$ we have the identity
    \begin{equation}
      \label{eq:adjunction-move}
      (\id_{n-i} \circ\, \kappa) \cdot (\id_{n-i} \circ\, m_{i,n-i})
      = (\kappa \circ \id_{n-i}) \cdot (m_{n-i,i} \circ \id_{n-i}) \,.
      \end{equation}
  \end{itemize}
  We proceed to show that  two string diagrams starting and ending at the same strings can be
  transformed into each other using the moves listed above. This implies the coherence statement. 
  First we wish to reduce the description to the case of string diagrams that do not involve
  any $\kappa$-vertices. To achieve this we use the adjunction move, by which we can move 
  any $\kappa$- or $\kappa^{-1}$-vertex past a string that is on the left of 
  it, thereby introducing new vertices with $m$ and $m^{-1}$. 
  By repeated use of this prescription, and canceling $\kappa$ against $\kappa^{-1}$,
  we can assume that we are left with only one type, say $\kappa$-vertices, which are moreover 
  located to the left and below all other vertices.  Performing this step on all such pairs 
  of diagrams leaves us with the task of proving that any two diagrams $G_{1}$ and $G_{2}$
  from strings $z_{1}$ to $z_{s}$ that involve only trivalent vertices can be transformed 
  into each other. That this is possible is already quite obvious from the associativity 
  rule for $m$; a more formal argument goes as follows.
  Let $z \,{=}\,S_{j_{1}} \,{\circ} \cdots {\circ}\, S_{j_{r}}$ be the string 
  at which the diagram $G_{1}$ starts.
  We can assume that the sum $t_{z} \,{:=}\, \sum_{k=1}^{r} j_{k}$ of labels satisfies 
  $t_z \,{<}\, n$, and we first define a \emph{standard} string diagram $G_{\std}$ from 
  $z$ to $S_{t_{z}}$ (otherwise, if $t_{z} \,{>}\, n$,  there is an analogous definition
  of a standard diagram to $(S_{1})^{t_{z}}$, a composite of shifts by $1$, and the proof
  continues along the same lines exchanging $m$  and $m^{-1}$ in the subsequent arguments).
  Then we prove that every string diagram from $z_{1}$ to $z_{s}$ has the same evaluation 
  as the composite of the standard string diagram  from $z_{1}$ to $S_{t_{z_{1}}}$ with the 
  inverse standard string diagram to $z_{s}$. This implies the general coherence statement. 
  The standard string diagram $G^{\std}$ from $z$ to $S_{t_{z}}$ is defined by using the 
  vertex $m$ to inductively combine the first two strings to a single string, until 
  at last only the string $S_{t_{z}}$ remains. Let $G$ be a string diagram from the 
  string $z_{1} \,{=}\, S_{j_{1}} \,{\circ} \cdots {\circ}\, S_{j_{r}}$ 
  to $z_{s} \,{=}\, S_{d_{1}} \,{\circ} \cdots {\circ}\, S_{d_{p}}$ via intermediate 
  strings $z_2,...\,, z_{s-1}$. By our assumption, $t_{z_{1}}\,{=}\ldots{=}\,t_{z_{s}}$. It
  suffices to show that, using the allowed moves, we can transform $G$ into the string diagram 
  that consists of the standard diagram from $z_{1}$ to $S_{t_{z}}$ composed with the 
  inverse of the standard diagram to $z_{s}$. $G$ might contain vertices labeled by $m$, 
  which we call \emph{$m$-vertices} and $m^{-1}$-vertices representing $m^{-1}$.
  If we agree that an arrow from left to right represents an  $m$-vertex, while an arrow 
  from right to left represents the inverse vertex, we can write the composite of vertices 
  in $G$ as the upper line of the following diagram:
  \begin{equation}
    \label{diag:Gstd}
    \begin{tikzcd}
      z_{1}  \ar{d}{G_{\std}} \ar{r}{} & z_{2} \ar{d}{G_{\std}}  \ar{r}  & \ldots \ar{d}{G_{\std}}
      &  z_{s} \ar{d}{G_{\std}} \ar{l}{}\\
      S_{t} \ar{r}{=}& S_{t} \ar{r}{=} & \ldots  \ar{r}{=} & S_{t}.
    \end{tikzcd}
  \end{equation}
  The horizontal arrows represent all the respective standard diagram. 
  We need to show that the outer rectangle in the diagram \eqref{diag:Gstd} commutes, 
  i.e.\ that the upper row can be transformed into the composite of the lower row and the
  vertical rows. This, in turn, follows if each of the squares in the diagram commutes 
  in this sense. To show the latter, it is sufficient to consider two string diagrams to $S_{t}$ 
  given by $m$-arrows, i.e.\ involving only the trivalent vertices $m$,
  and to show that they can be transformed into each other. This is clear once we 
  realize that the datum of a string diagram with $m$-vertices from 
  $z=S_{j_{1}} \circ \cdots \circ S_{j_{r}}$ to $S_{t}$ is encoded in a bracketing of the string 
  $S_{j_{1}} \circ \cdots \circ S_{j_{r}}$: the bracket $(S_{i} \circ S_{j})$ means 
  that the vertex $m$ is applied to the strings labeled ${i}$ and ${j}$. 
  Now the associativity rule implies that we can change the brackets arbitrarily. 
  This concludes the proof of Proposition \ref{proposition:2-graph-connected}.
\end{proof}

The result just obtained allows one to deduce that the 3-trace of cyclically composable 
1-morphisms $\{f_{i}\}_{i \in \{1,\ldots, n\}}$ depends only on the cyclic set $\{f_{i}\}$.
In the rest of this section we make this statement precise. We first promote 
the 2-graph $\Cat{G}$ from Definition \ref{definition:2-graph} to a bicategory 
$\Cat{X}$, such that the  evaluations (\ref{equation:evaluation-string}) of $\{f_{i}\}$ 
constitute a 2-functor $\mathsf{E}\colon \Cat{X} \,{\to}\, \Cat{B}$. 
Proposition \ref{proposition:2-graph-connected} shows that $\Cat{X}$ is contractible: 
it is biequivalent to a bicategory with just one object, one 1-morphism and one 2-morphism. 
It follows that the 2-limit of the evaluation $\mathsf{E}$ depends 
only on the cyclic set $\{f_{i}\}$.

A 2-limit of a 2-functor $\mathsf{E}\colon \Cat{X} \,{\to}\, \Cat{B}$ (see e.g.\
\cite[p.\,153]{Street2fun}), 
consists of an object $\lim \mathsf{E} \,{\in}\, \Cat{B}$ together with an equivalence 
\begin{equation}
  \label{eq:2-lim}
  \Cat{B}(b, \lim \mathsf{E}) \simeq [\Cat{X},\Cat{B}](\underline{b}, \mathsf{E}).
\end{equation}
of categories for all $b \,{\in}\, \Cat{B}$.
Here $[\Cat{X},\Cat{B}]$ is the  bicategory of 2-functors from $\Cat{X}$ to $\Cat{B}$ and
$\underline{b}\colon \Cat{X} \,{\to}\, \Cat{B}$ is the constant 2-functor to $b \,{\in}\, \Cat{B}$.
In \eqref{eq:2-lim} $\Cat{X}$ is typically a bicategory that encodes diagrams of a given shape. 
In our context, $\Cat{X}$ is based on the 2-graph $\Cat{G}$ from Definition \ref{definition:2-graph}:
Note first that each 1- and 2-cell of $\Cat{G}$ comes with a notion of source and target, 
hence it makes sense to speak about composable cells in $\Cat{G}$. 
The objects of $\Cat{X}$ are the vertices of $\Cat{G}$. 
The 1-morphisms are strings of composable  1-cells of $\Cat{G}$. A 2-morphism  between 
two strings of 1-cells is a string of composable 2-cells in $\Cat{G}$ up to equivalence. 
The equivalences of composable 2-cells are induced by the axioms of a tricategory and 
the axioms (\ref{eq:axiom-3trace}),~(\ref{eq:adj-equiv-varphi}) of a 3-trace. 

\begin{theorem}
  \label{theorem:coherence-3-tr}
 Let $\Cat{T}$ be a $3$-category with a $3$-trace $\tr$, and let $\{f_{i} \}_{i \in \{1,\ldots, n\}}$ 
 be a set of cyclically composable $1$-morphisms. If it exists, the $2$-limit of the $3$-traces of the 
 $\{f_{i}\}$ under the admissible equivalences depends only on the cyclic set $\{f_{i}\}$.
\end{theorem}

\begin{proof}
The evaluations of the 1-morphisms $\{f_{i}\}$ define a 2-functor 
$\mathsf{E}\colon \Cat{X} \,{\to}\, \Cat{B}$. 
With the 2-limit under admissible equivalences we mean the 2-limit $\lim \mathsf{E}$ of this 2-functor. 
Proposition \ref{proposition:2-graph-connected} implies that the bicategory $\Cat{X}$ is contractible: The 
functor that sends an object  to its equivalence class in $\Cat{X}$, a 1-morphism to its isomorphism 
class and a 2-morphism to the identity is a biequivalence. The equivalence classes in $\Cat{X}$ 
encode precisely the cyclic set associated to $\{f_{i}\}$. It follows that also the 2-limit of 
the evaluations only depends on this cyclic set. 
\end{proof}

We can thus denote the 2-limit of the 3-traces of a set of cyclically composable 1-morphisms $\{f_{i}\}$ by 
$\tr(\{f_{i}\})$. 

\medskip

We expect that every tricategory $\Cat{T}$ which has in addition a symmetric monoidal structure 
and in which all objects are 1-du\-a\-li\-zable has a canonical 3-trace with values in 
the bicategory $\Cat{T}(\unit, \unit)$ where $\unit$ is the unit of the symmetric monoidal structure. 

We are mostly interested in the case that $\Cat{B} \,{=}\, \Categ$, the 2-category of 
essentially small linear categories. We call a 3-trace on $\Cat{T}$ to $\Categ$ 
\emph{representable} if for every object $x$ in $\Cat{T}$ there exists a 1-morphism 
$X\colon x \,{\to}\, x$ together with a natural 2-isomorphism from the 2-functor 
$\tr_{x}$ to the 2-functor $\Cat{T}(X,-)\colon \Cat{T}(x,x) \,{\to}\, \Categ$. 

The 3-category $\BimCat$ is special in the sense that it belongs to both subclasses of 
tricategories: On the one hand it is expected to possess a symmetric monoidal structure, given by 
the Deligne product with unit $\Vect$, for which every object is 1-dualizable \cite{DSS}. Note that
$\BimCat(\Vect, \Vect)$ is the sub-2-category of $\Categ$ consisting of finite linear categories. 
On the other hand, according to Proposition \ref{proposition:circtens-2-fun}\,(\rm ii)  
the 2-functors $\circtensor$ (which define a 3-trace, as we will see next) are representable 
by the bimodule categories $\tCCC$ for every finite tensor category $\Cat{C}$.

\medskip

In a tricategory $\Cat{T}$ with 3-trace with values in $\Cat{B}$, every object 
$x \,{\in}\, \Cat T$ has, via the identity $1$-morphisms in $\Cat{T}$, an associated object 
  \begin{equation}
    \cent_{\Cat{B}}(x) := \tr_{x}(1_{x}) 
  \end{equation}
in $\Cat B$, which may be called a generalized Drinfeld center of the object $x \,{\in}\, \Cat T$. 

\begin{remark}
  In case the tricategory $\Cat{T}$ has duals for 1-morphisms, these duals yield canonical 
  1-morphisms in $\Cat{B}$ between the generalized Drinfeld centers as follows:
  Consider a 1-morphism  $M\colon a \,{\to}\, b$. A dual for $M$ is a 1-morphism
  $M^{\#}\colon b \,{\to}\, a$ in $\Cat{T}$ together with 
  2-morphisms $\ev{M}\colon M^{\#} \Box M \,{\Rightarrow}\, 1_{a} $ and  
  $\coev{M}\colon 1_{b} \,{\Rightarrow}\, M \Box M^{\#}$ such that the snake identity holds 
  in $\Cat{T}$ up to an invertible 3-morphism. These data allow us to define the 1-morphism
  \begin{equation}
    \label{eq:adj-centers}
    \mathsf{F}= \tr_{a}(\ev{M}) \circ \varphi(M,M^{\#}) \circ \tr_{b}(\coev{M})
  \end{equation}
  between  $\cent_{\Cat{B}}(b)$ and $\cent_{\Cat{B}}(a)$.
  In case there are even duals for the 2-morphisms in $\Cat{T}$, taking the (right, say) 
  dual of the evaluation and coevaluation morphisms of $M$, we obtain an adjunction 
  in $\Cat{B}$ between the generalized Drinfeld centers. 
  In case $M$ is invertible, the adjunction is an adjoint equivalence. 

  For the example of bimodule categories, this leads to adjunctions between the twisted 
  Drinfeld centers for every separable bimodule category between finite tensor categories. 
\end{remark}

%%%%%%%%%%%%%%%%%%%%%%%%%%%%%%%%%%%%%%%%%%%%%%%%%%%%%%%%%%%%%%%%%%%%%%%%%%%%%%%% 

\subsection{The \ct\ as 3-trace} \label{sec:categ-assoc-circl}

In this subsection we show that the \ct\ introduced in Section \ref{sec3} provides a 3-trace on the 
tricategory $\BimCat$ with values in the bicategory $\Categ$. In particular we establish compatibility 
as in Definition \ref{definition:3-trace} with the relative tensor product 
of bimodule categories, which is the composition of 1-morphisms in this case. 

As a new ingredient, we need the notion of multi-balancedness of a functor:
\begin{definition}
  \label{definition:multi-bal}
  Let $\Cat{C}$ and $\Cat{D}$ be finite tensor categories and $\Cat{A}$ a linear category. 
  A functor $\mathsf{F}\colon \CMD {\times} \DNC \,{\rightarrow}\, \Cat{A}$ is called 
  \emph{multi-balanced} if it is both $\Cat{D}$-balanced in the sense of
  Definition \ref{definition:balanced-functors} and $\Cat{C}$-balanced in the sense of
  Definition \ref{definition:circular-balanced}, i.e.\ if it is equipped with natural isomorphisms 
  \begin{equation}
    \label{eq:C-mulit-bal}
    \beta_{c, m\times n} :\quad
    \mathsf{F}(c \act m \times n) \stackrel{\cong}{\longrightarrow} \mathsf{F}(m \times n \ract c)
  \end{equation}
  for all $c \,{\in}\, \Cat{C}$ and 
  \begin{equation}
    \label{eq:D-bal}
    \gamma_{d, m \times n} :\quad
    \mathsf{F}( m \ract d \times n) \stackrel{\cong}{\longrightarrow} \mathsf{F}(m \times d \act n)
  \end{equation}
  for all $d \,{\in}\, \Cat{D}$ that satisfy the respective pentagon and triangle identities, and if
  furthermore the two balancing constraint are compatible, i.e.\ the diagrams 
  \begin{equation}
    \label{eq:mulit-bal-comm-diag}
    \begin{tikzcd}[column sep=large] 
      { \begin{array}{r}
         \mathsf{F}((c \act m) \ract d \times n)  \quad \\
          \cong \mathsf{F}(c \act( m \ract d) \times n )   \! \end{array}}
      \ar{r}{\gamma_{d, c\act m \times n}}
      \ar{d}{\beta_{c, m \ract d \times n}} &\mathsf{F}(c \act m \times d \act n)
      \ar[left]{d}{\beta_{c, m \times d \act n}}
      \\
      \mathsf{F}(m \ract d \times n \ract c) \ar{r}{\gamma_{d, m \times n \ract c}}
      & \mathsf{F}(m \times d \act n \ract c)
    \end{tikzcd}
  \end{equation}
  commute for all objects $c \,{\in}\, \Cat{C}$, $d \,{\in}\, \Cat{D}$, $m \,{\in}\, \CMD$,
  $n \,{\in}\, \DNC$. 
  \\[2pt]
  More generally, one can consider strings of cyclically composable  bimodule categories
  $\underline{\Cat{M}}$: These are composable bimodule categories, such that for the outermost 
  bimodule categories, the respective outer tensor categories agree. Then multi-balanced
  functors and natural transformations from $\underline{\Cat{M}}$ to $\Cat{A}$ are defined
  analogously, with the consistency condition $\eqref{eq:mulit-bal-comm-diag}$ required
  between any two consecutive arguments of a multi-balanced functor. We denote the category
  formed by these by $\Funbal(\underline{\Cat{M}},\Cat{A})$, and the category of those
  multi-balanced functors which are in addition right exact by
  $ \Funbalre(\underline{\Cat{M}}, \Cat{A})$.
\end{definition} 

The \ct\ can be extended to the  \ct\ of a cyclically composable string of bimodule categories:
\begin{definition}
  \label{definition:string-circular-tensor}
  Given a  cyclically composable  
  string $\underline{\Cat{M}}$ of bimodule categories over finite tensor categories,
  a \emph{\ct} $\circtensor \, \underline{\Cat{M}}$
  is an abelian category that is universal with respect to multi-balanced functors, i.e.\ for
  which there exists a multi-balanced functor
  \begin{equation}
    \mathsf{B}\colon\quad \underline{\Cat{M}} \rightarrow \circtensor \, \underline{\Cat{M}}
  \end{equation}
  such that for every linear category $\Cat{A}$, composition with $\mathsf{B}$ is
  endowed with the structure of an adjoint equivalence 
  \begin{equation}
    \label{eq:adj-equ-multi-bal}
    \Funre( \circtensor \, \underline{\Cat{M}}, \Cat{A}) \simeq \Funbalre(\underline{\Cat{M}}, \Cat{A})
  \end{equation}
  of categories.
\end{definition}

\begin{proposition}
  \label{proposition:existence-circular-tensor}
  For any cyclically composable string $\underline{\Cat{M}}$ of bimodule categories over 
  finite tensor categories, the \ct\ $\circtensor \, \underline{\Cat{M}}$ exists.
\end{proposition}

\begin{proof}
  We consider the example of the string $\DMC {\times} \CNE \,{\times} \EKD$. The general case is 
  treated analogously. First we choose a bracketing, say $(-\Box-)\Box-$ and a cut point, say $\EKD$.
  The relative tensor product provides for all linear categories $\Cat{A}$ an equivalence of categories 
  \begin{equation}
    \label{eq:rel-tens}
    \Funbalre(\DMC {\times} \CNE {\times} \EKD, \Cat{A})
    \cong \Funbalre( (\DMC \boxtensor{\Cat{C}}  \CNE ) \boxtensor{\Cat{E}} \EKD, \Cat{A}).
  \end{equation}
  It follows that the trace $\circtensor  (\DMC \boxtensor{\Cat{C}}  \CNE ) \boxtensor{\Cat{E}} \EKD $ 
  satisfies the required universal property. 
\end{proof}

\medskip

In the application to TFT that was described in Section \ref{sec:5}, 
this category will be assigned to a circle with $n$ marked points
at which surface defects labeled by the 1-morphisms are located.
It thus provides the labels for generalized Wilson lines at which $n$ surface defects meet.

Next  we establish cyclic equivalences of the \ct, thereby in particular justifying our terminology.

\begin{proposition}
  \label{proposition:circ-equiv-1}
  Let $\CMD$ and $\DNC$ be  bimodule categories over finite tensor categories $\Cat{C}$ and $\Cat{D}$. 
  \begin{propositionlist}
  \item[{\rm (i)}] 
    Switching the arguments in a bilinear functor induces for all linear categories $\Cat{A}$ an equivalence 
    \begin{equation}
      \label{eq:induced-equiv-S}
      \Funbalre(\CMD \times \DNC, \Cat{A}) 
      \xrightarrow{~\widetilde{\varphi}_{\Cat{M},\Cat{N}}}~ \Funbalre(\DNC \times \CMD, \Cat{A}),
    \end{equation}
    that satisfies 
    $\widetilde{\varphi}_{\Cat{N},\Cat{M}} \circ \widetilde{\varphi}_{\Cat{M},\Cat{N}} = \id$. 
  \item[{\rm (ii)}] 
    The equivalence from part {\rm (i)} induces an equivalence
    \begin{equation}
      \label{eq:circular-equiv}
      \circtensor \CMD \boxtensor{\Cat{D}} \DNC 
      \xrightarrow{~{\varphi}_{\Cat{M},\Cat{N}}}~ \circtensor \DNC \boxtensor{\Cat{C}} \CMD,
    \end{equation}
    that is unique up to unique natural isomorphism, together with an equivalence 
    ${\varphi}_{\Cat{N},\Cat{M}} \circ {\varphi}_{\Cat{M},\Cat{N}} \simeq \id$. 
  \item[{\rm (iii)}] 
    For bimodule categories $\DMC {\times} \CNE \,{\times} \EKD$, the following  diagram of equivalences commutes:
    \begin{equation}
      \label{eq:comm-equiv}
      \begin{tikzcd}[column sep=large]
        \Funbalre( \DMC {\times} \CNE \, {\times} \EKD, \Cat{A})
        \ar{d}{\widetilde{\varphi}_{\Cat{M}, \Cat{N} \times \Cat{K}}}
        \ar{r}{ \widetilde{\varphi}_{\Cat{M} \times \Cat{N}, \Cat{K}}}  &
        \Funbalre( \EKD {\times} \DMC \, {\times} \CNE, \Cat{A})
        \ar{dl}{\widetilde{\varphi}_{\Cat{K} \times \Cat{M}, \Cat{N}}} \\
        \Funbalre( \CNE {\times} \EKD \, {\times} \DMC, \Cat{A}). & 
      \end{tikzcd}
    \end{equation}
    This induces a natural isomorphism of functors 
    \begin{equation}
      \label{eq:m-for-bimdod}
      m_{\Cat{M},\Cat{N},\Cat{K}}:\quad
      \varphi_{\Cat{K} \boxtensor{\Cat{D}} \Cat{M}, \Cat{N}} \circ \varphi_{\Cat{M} \boxtensor{\Cat{C}} \Cat{N}, \Cat{K}}
      \rightarrow \varphi_{\Cat{M}, \Cat{N} \boxtensor{\Cat{E}}\Cat{K}} \,,
    \end{equation}
    where for simplicity we omitted the associator in $\BimCat$.
  \item[{\rm (iv)}] 
    For a bimodule category $\CMC$,  
    the following diagram of equivalences commutes up to a canonical natural isomorphism:
    \begin{equation}
      \label{eq:towards-kappa}
      \begin{tikzcd}
        \Funbalre(\Cat{C} \times \CMC, \Cat{A}) \ar{r}{\widetilde{\varphi}_{\Cat{C},\Cat{M}}} 
        & \Funbalre(\CMC \times \Cat{C}, \Cat{A})  \ar{d}{\phi_{2}} \\
        \Funbalre(\CMC, \Cat{A}) \ar{u}{\phi_{1}} \ar{r}{\id} & \Funbalre(\CMC, \Cat{A}).
      \end{tikzcd}
    \end{equation}
    Here the functor $\phi_{1}$ maps a balanced functor $\mathsf{F}\colon \CMC \,{\to}\, \Cat{A}$ 
    to the functor $\Cat{C} \times \Cat{M} \ni c \times m \mapsto \mathsf{F}(c \act m)$, 
    and $\phi_{2}$ is defined by mapping a balanced functor 
    $\mathsf{G}:\Cat{M} \times \Cat{C}\, {\to}\,\Cat{A}$ to the functor 
    $\Cat{M} \,{\ni}\, m \,{\mapsto}\, \mathsf{G}(m {\times} \unit_{\Cat{C}})$.
    The diagram {\rm\eqref{eq:towards-kappa}} defines a natural isomorphism 
    $\kappa_{\Cat{M}}\colon \varphi_{\Cat{C}, \Cat{M}} \,{\to}\, \id$.
  \end{propositionlist}
\end{proposition}
\begin{proof}
  The first part is clear. 
  For the second part consider the following diagram 
  \begin{equation}
    \label{eq:comm-diag-circ-equiv}
    \begin{tikzcd}
      \Funbalre(\CMD \times \DNC, \Cat{A}) 
      \ar{r}{\widetilde{\varphi}_{\Cat{M},\Cat{N}}} \ar{d}{\Psi_{\Cat{M},\Cat{N}}} 
      &\Funbalre(\DNC \times \CMD, \Cat{A}) \ar{d}{\Psi_{\Cat{N},\Cat{M}}}\\
      \Funre(\circtensor \CMD \boxtensor{\Cat{D}} \DNC, \Cat{A}) \ar{r}{} &
      \Funre(\circtensor \DNC \boxtensor{\Cat{C}} \CMD, \Cat{A}).
    \end{tikzcd}
  \end{equation}
  Since all arrows in this diagram are equivalences of categories, the diagram establishes 
  $\circtensor \DNC \boxtensor{\Cat{C}} \CMD$ as a \ct\ of $\CMD \times \DNC$. 
  By the universal property of the \ct\, this induces the equivalence 
  $\varphi_{\Cat{M},\Cat{N}}\colon \circtensor \CMD \boxtensor{\Cat{D}} \DNC  \rightarrow   
  \circtensor \DNC \boxtensor{\Cat{C}} \CMD$.
  \\
  The natural isomorphism in the remaining parts of the proposition are also obtained 
  from the universal property of the \ct: Between any two equivalences of \ct\ of 
  a cyclically composable string of bimodule categories, 
  that are induced by the universal property, there is a unique natural isomorphism. 
\end{proof}

\begin{theorem}
  \label{theorem:ct-is-3trace}
  The \ct\ provides a $3$-trace on the tricategory $\BimCat$ with values in $\Categ$. 
  The $3$-trace is representable by the bimodule categories $\tCCC$ for finite tensor categories $\Cat{C}$.
\end{theorem}
\begin{proof}
  According to Proposition \ref{proposition:circtens-2-fun}, the \ct\ provides the
  2-functors $\BimCat(\Cat{C},\Cat{C}) \rightarrow \Categ$. The structural data of a 3-trace
  are provided in Proposition \ref{proposition:circ-equiv-1}. It remains to be shown that
  the two axioms in the definition of a 3-trace are satisfied. 
  To this end note that the natural isomorphisms $m$ and $\kappa$ in
  Proposition \ref{proposition:circ-equiv-1} are defined \emph{uniquely}. This implies that
  for a string of four cyclically composable bimodule categories, where there are two ways to 
  define $m$ by the universal property, these two definitions have to agree. This shows that
  axiom (\ref{eq:axiom-3trace}) is satisfied for the \ct. The other axiom is shown similarly. 
\end{proof}

%%%%%%%%%%%%%%%%%%%%%%%%%%%%%%%%%%%%%%%%%%%%%%%%%%%%%%%%%%%%%%%%%%%%%%%%%%%%%%%% 

\vskip 3.3em
% \newpage

{\small \noindent{\sc Acknowledgments:}
  We are grateful to Alessandro Valentino and to Christopher Schommer-Pries for helpful discussions.
  \\
  JF is supported by VR under project no.\ 621-2013-4207. CS is partially supported
  by the Collaborative Research Centre 676 ``Particles, Strings and the Early
  Universe - the Structure of Matter and Space-Time'' and by the DFG Priority
  Programme 1388 ``Representation Theory''. 
  GS thanks the MPIM for excellent working conditions and Hamburg University for 
  great hospitality. We thank the Erwin-Schr\"odinger-Institute (ESI) for the
  hospitality during the program ``Modern trends in topological field theory''
  while this work was started; this program was also supported by the network
  ``Interactions of Low-Dimensional Topology and Geometry with Mathematical
  Physics'' (ITGP) of the European Science Foundation.
}

 \newpage
%%%%%%%%%%%%%%%%%%%%%%%%%%%%%%%%%%%%%%%%%%%%%%%%%%%%%%%%%%%%%%%%%%%%%%%%%%%%%%%% 

  \bibliographystyle{halpha}%{hplain}     % (uses file "plain.bst")  
  \bibliography{bimod}

%%%%%%%%%%%%%%%%%%%%%%%%%%%%%%%%%%%%%%%%%%%%%%%%%%%%%%%%%%%%%%%%%%%%%%%%%%%%%%%% 
\end{document}